\documentclass{amsart}
\usepackage[margin=1in]{geometry}
\usepackage{amsmath}
\usepackage{amssymb}
\usepackage{amsthm}
\usepackage{amscd}
\newcounter{pcounter}

\newcommand{\R}{\mathcal{R}}

\newcommand{\ZZ}{\Bbb Z}
\newcommand{\RR}{\Bbb R}
\newcommand{\TT}{\Bbb T}
\newcommand{\NN}{\Bbb N}

\newcommand{\CC}{\Bbb C}
\newcommand{\ip}[1]{\langle #1 \rangle}

\newcommand{\widetidle}{\widetilde}
\newcommand{\ovelrine}{\overline}
\newcommand{\varpesilon}{\varepsilon}

\newcommand{\varespilon}{\varepsilon}
\newcommand{\vaerpsilon}{\varepsilon}

\newtheorem{?}{Question}
\newtheorem{theorem}{Theorem}
\newtheorem{question}[theorem]{Question}
\newtheorem{conjecture}[theorem]{Conjecture}
\newtheorem{definition}[theorem]{Definition}
\newtheorem{proposition}[theorem]{Proposition}
\newtheorem{cor}[theorem]{Corollary}
\newtheorem{lemma}[theorem]{Lemma}

\newcommand{\FF}{\Bbb F}

\DeclareMathOperator{\Gr}{Gr}

\DeclareMathOperator{\Span}{Span}

\DeclareMathOperator{\tr}{tr}

\DeclareMathOperator{\Ima}{Im}

\DeclareMathOperator{\rea}{re}
\DeclareMathOperator{\ima}{ima}
\DeclareMathOperator{\Rea}{Re}

\DeclareMathOperator{\Hom}{Hom}

\DeclareMathOperator{\alg}{alg}

\DeclareMathOperator{\Tr}{Tr}

\DeclareMathOperator{\op}{op}

\DeclareMathOperator{\Ball}{Ball}

\numberwithin{theorem}{section}

\begin{document}

\title{$1$-Bounded Entropy and Regularity Problems in von Neumann Algebras}      
\author{Ben Hayes}
\address{Stevenson Center\\
         Nashville, TN 37240}
\email{benjamin.r.hayes@vanderbilt.edu}\thanks{This work is partially supported by  NSF Grant  DMS-1362138.}
\date{\today}
\maketitle

\begin{abstract} We define and investigate the singular subspace $\mathcal{H}_{s}(N\subseteq M)$ of an inclusion of tracial von Neumann algebras. The singular subspace is a canonical $N$-$N$ subbimodule of $L^{2}(M)$ containing the quasinormalizer introduced  in \cite{PopaQR}, the one-sided quasinormalizer introduced  in \cite{FangGaoSmith}, and the wq-normalizer introduced in \cite{GalatanPopa} (following upon work in \cite{IPP} and \cite{PopaCohomologyOE}). We then obtain a weak notion of regularity (called spectral regularity) by demanding that the singular subspace of $N\subseteq M$ generates $M.$ By abstracting Voiculescu's original proof of absence of Cartan subalgebras in \cite{FreeEntropyDimensionIII} we show that there can be no diffuse, hyperfinite subalgebra of $L(\FF_{n})$ which is spectrally regular. Our techniques are robust enough to repeat this process by transfinite induction and rule out chains of spectrally regular inclusions of algebras starting from a diffuse, hyperfinite subalgebra and ending in $L(\FF_{n}).$ We use this to prove some conjectures made by Galatan-Popa in their study of smooth cohomology of $\textrm{II}_{1}$-factors (see \cite{GalatanPopa}). Our results may be regarded as a consistency check for the possibility of existence of a ``good" cohomology theory of $\textrm{II}_{1}$-factors. We can also use our techniques to show that if $U_{t}$ is a one-parameter orthogonal group on a real Hilbert space $\mathcal{H}$ and the spectral measure of its generator is singular with respect to the Lebesgue measure, then the continuous core of the free Araki-Woods factor $\Gamma(\mathcal{H},U_{t})''$ is not isomorphic to $L(\FF_{t}\overline{\otimes} B(\ell^{2}(\NN))$ for any $t\in (1,\infty].$ In particular, $\Gamma(\mathcal{H},U_{t})''\not\cong \Gamma(L^{2}(\RR,m),\lambda_{t})''$ where $m$ is Lebesgue measure and $\lambda$ is the left regular representation. This was previously only know when the spectral measure of the generator of $U_{t}$ had all of its convolution powers singular with respect to Lebesgue measure. We give similar applications to crossed products by free  Bogoliubov actions in the spirit of \cite{HoudayerShlyakhtenko}.
\end{abstract}
\tableofcontents

\section{Introduction}
	This paper is concerned with using Voiculescu's microstates free entropy dimension to prove indecomposability results in von Neumann algebras. Recall that if $N\subseteq M$ is an inclusion of von Neumann algebras, we define the \emph{normalizer} of $N$ inside $M$ by
	\[\mathcal{N}_{M}(N)=\{u\in \mathcal{U}(M):u^{*}Nu=N\}.\]
We say that $N$ is \emph{regular} in $M$ (or that $N\subseteq M$ is a regular inclusion) if $\mathcal{N}_{M}(N)''=M.$

 	 Perhaps the most famous application of Voiculescu's microstates free entropy dimension is Voiculescu's fundamental result (see Theorem 7.3 of \cite{FreeEntropyDimensionIII}) that if $M$ is a von Neumann algebra with a faithful, normal, tracial state $\tau$ and $x_{1},\dots,x_{n}$ are self-adjoint generators of $M$ which have microstates free entropy dimension bigger than $1,$ then $M$ has no Cartan subalgebras. In fact, Voiculescu proves the stronger statement that such an $M$ does not have a regular, diffuse, hyperfinite subalgebra. This applies in particular when $M=L(\FF_{n})$ for $n\in\NN$ or more generally the interpolated free group factors $L(\FF_{t}),t\in (1,\infty]$ defined by Dykema and R\v{a}dulescu (see \cite{DykemaRadulescu}).

 	 	The absence of Cartan subalgebras  in $L(\FF_{n})$ was later proved by Ozawa and Popa (see \cite{OzPopaCartan}) using Popa's deformation/rigidty theory. They additionally proved that $L(\FF_{n})$ is strongly solid, i.e. that the normalizer of any diffuse amenable subalgebra of $L(\FF_{n})$ remains amenable. This result in particular rules out chains of inclusions
\begin{equation}\label{C:ChainIntro!@RW}
Q=Q_{0}\subseteq Q_{1}\subseteq \cdots \subseteq Q_{\alpha}=L(\FF_{n})
\end{equation}
where $Q$ is hyperfinite and $Q_{\beta-1}\subseteq Q_{\beta}$ is a regular inclusion. The absence of such a chain as in (\ref{C:ChainIntro!@RW}) was also obtained by Hadwin-Li in \cite{HadwinLi}. The result of Ozawa-Popa proves the much stronger statement that $Q_{\beta}$ is in fact hyperfinite for all $\beta.$ Popa's deformation/rigidity theory has since been used to achieve several landmark achievements in this field as well as deduce \emph{uniqueness} of Cartan subalgebras in crossed product algebras by various groups (just to name a few examples see e.g. \cite{PopaL2Betti},\cite{OzPopaCartan},\cite{OzPopaII},\cite{ChifanSinclair},\cite{IPP},\cite{ChifanPeterson},\cite{AdrianBern},\cite{AdrianL2Betti},\cite{PopaVaesGMS},\cite{PopaVaesFree},\cite{PopaVaesHyp}).

	In this paper, we will generalize the absence of chains such as $(\ref{C:ChainIntro!@RW})$  by considering weakened versions of the normalizer. For example we will consider the quasi-normalizer and one-sided quasi-normalizer:
\[q\mathcal{N}_{M}(N)=\left\{x\in M:\mbox{ there exists $x_{1},\dots,x_{n}\in N$ with } xN\subseteq \sum_{j=1}^{n}Nx_{j}\mbox{ and }Nx\subseteq \sum_{j=1}^{n}x_{j}N \right\},\]
\[q^{1}\mathcal{N}_{M}(N)=\left\{x\in M:\mbox{ there exists $x_{1},\dots,x_{n}\in N$ with } xN\subseteq \sum_{j=1}^{n}Nx_{j}\right\}.\]
 We say that $N$ is quasi-regular (respectively one-sided quasi-regular) in $M$ if $W^{*}(q\mathcal{N}_{M}(N))=M$ (respectively $W^{*}(q^{1}\mathcal{N}_{M}(N))=M)$.  Quasiregularity was introduced by Izumi-Longo-Popa in \cite{IzumiLongoPopa} (see Definition 3.7) under the term \emph{discreteness} (see also \cite{PimnserPopa} Proposition 1.3) the term ``quasiregular" was first introduced by Popa in \cite{PopaQR}. The one-sided quasi-normalizer was defined by  Izumi-Longo-Popa in \cite{IzumiLongoPopa}. All of these notions are inspired by, and can be traced back to, the work of Popa on Orthogonal Masas in \cite{PopaOMasas}. In private communication with Popa, we have learned that Popa defined, in a 2005 mini-course in the conference ``Noncommmutative Geometry and Operator algebras at Vanderbilt Univeristy \cite{PopaMC}, a related notion of an interwining space of $Q$ into $P$ inside of $M,$ when $P,Q$ are subalgebra of $M$ (even if $P$ is not a subset of $Q$). We define the quasi-normalizing algebra of $N\subseteq M$ and the one-sided quasi-normalizing subalgebra of $N\subseteq M$ by
\[W^{*}(q\mathcal{N}_{M}(N)),\]
\[W^{*}(q^{1}\mathcal{N}_{M}(N)),\]
respectively. Motivated by the strong solidity results of Ozawa and Popa we can inductively define chains of algebras
\begin{equation}\label{E:OSQN}
N_{0}=N_{1}\subseteq N_{2}\subseteq\cdots \subseteq N_{\alpha}\subseteq M
\end{equation}
by saying that
\begin{equation}\label{E:OSQN1}
N_{\beta}=W^{*}(q^{1}\mathcal{N}_{M}(N_{\beta-1})).
\end{equation}
It is then reasonable to investigate for what $M$ and $N$ we can guarantee that $N_{\alpha}\ne M$ for all $\alpha.$ For example it is natural to ask if this is true if $N$ is hyperfinite and $M=L(\FF_{t})$ with $t\in (1,\infty].$ We note that a similar chain of algebras was considered in Definition 1.2.2 of \cite{IPP}.

	A similar question has already been asked by Galatan-Popa  \cite{GalatanPopa}. If $N\subseteq M$ are diffuse von Neumann algebras, we define the step-$1$ wq-normalizer by
\[N_{M}^{wq}(N)=\{u\in \mathcal{U}(M):u^{*}Nu\cap N\mbox{is diffuse}\}.\]
In \cite{GalatanPopa}, Galatan-Popa define  (in the spirit of definition 1.2.2 of \cite{IPP} and Definition 2.3 of \cite{PopaCohomologyOE}) the wq-normalizing subalgebra as the smallest von Neumann subalgebra $Q$ of $M$ containing $N$ for which $\mathcal{N}^{wq}_{M}(Q)=\mathcal{U}(Q).$ Equivalently, one considers a chain (indexed by ordinals $\alpha$) of algebras
\[N_{0}=N\subseteq N_{1}\subseteq N_{2}\subseteq \cdots \subseteq N_{\alpha}\subseteq M\]
defined by
\[N_{\beta}=W^{*}(\mathcal{N}^{wq}_{M}(N_{\beta-1}))\mbox{ if $\beta$ is a succesor ordinal},\]
\[N_{\beta}=\overline{\bigcup_{\beta'<\beta}N_{\beta'}}^{SOT}\mbox{ if $\beta$ is a limit ordinal},\]
then $Q=N_{\alpha}$ where $\alpha$ is the first index at which the chain stabilizes, i.e. $N_{\alpha}=N_{\alpha+1}.$ In \cite{GalatanPopa} (see Remark 3.9) Galatan-Popa conjecture that the wq-normalizing subalgebra of $N\subseteq L(\FF_{2})$ is never $L(\FF_{2})$ if $N$ is diffuse and hyperfinite. In this paper, we shall answer this question affirmatively by analyzing the structure of $\ell^{2}(\FF_{2})$ as an $N$-$N$ bimodule.

	Let $M$ be a von Neumann algebra and recall that two  $M$-$M$ bimodules $\mathcal{H},\mathcal{K}$ are \emph{disjoint} if there is no non-zero, bounded, $M$-$M$ bimodular map $T\colon \mathcal{H}\to \mathcal{K}.$ Since adjoints of $M$-$M$ bimodular maps are $M$-$M$ bimodular, this notion is symmetric. For our purposes, a tracial von Neumann algebra is a pair $(M,\tau)$ where $M$ is a von Neumann algebra and $\tau$ is a faithful, normal, tracial state on $M.$ A von Neumann subalgebra $N$ of a von Neumann algebra $M$ is a weak operator topology closed $*$-subalgebra of $M$ which shares the same multiplicative identity.  Let $(M,\tau)$ be a tracial von Neumann algebra and $N\subseteq M$ a von Neumann subalgebra. For $\xi\in L^{2}(M)$ we use $L^{2}(N\xi N)$ for
\[\overline{\left\{\sum_{j=1}^{k}x_{j}\xi y_{j}:x_{j},y_{j}\in N,j=1,\dots,k\right\}}^{\|\cdot\|_{2}}.\]
The following $N$-$N$ subbimodule of $L^{2}(M)$ will be the crucial object of study in the paper.

\begin{definition}\emph{Let $(M,\tau)$ be a tracial von Neumann algebra and $N\subseteq M$ a von Neumann subalgebra. Define the} singular subspace of $L^{2}(M)$ over $N$ \emph{by}
\[\mathcal{H}_{s}(N\subseteq M)=\{\xi \in L^{2}(M):L^{2}(N\xi N) \mbox{\emph{ is disjoint from }} L^{2}(N)\otimes L^{2}(N^{op})\mbox{\emph{ as an $N$-$N$ bimodule}}\}.\]
\end{definition}

To motivate this definition we show that the singular subspace contains every version of normalizer in our discussion.
\begin{proposition}\label{P:spectralregularintro} Let $(M,\tau)$ be a tracial von Neumann algebra and let $N\subseteq M$ be a diffuse von Neumann subalgebra. We then have the following inclusions:
\[q^{1}\mathcal{N}_{M}(N)\subseteq \mathcal{H}_{s}(N\subseteq M),\]
\[\mathcal{N}_{M}^{wq}(N)\subseteq \mathcal{H}_{s}(N\subseteq M).\]
\end{proposition}

 Regard $L^{2}(M)$ as a collection of closed densely-defined operators affiliated to $M.$ Thus if $X\subseteq L^{2}(M)$ and $\xi=u_{\xi}|\xi|$ for $\xi\in X$ is the polar decomposition (as an unbounded operator), we use $W^{*}(X)$ for
\[W^{*}(\{u_{\xi}:\xi\in X\}\cup \{f(|\xi|):\xi\in X, f\mbox{ is a bounded Borel function on $[0,\infty)$}\}).\]
We can now state our main theorem.

\begin{theorem}\label{T:mainintro1} Let $t\in (1,\infty]$ and let $N$ be a diffuse, hyperfinite von Neumann subalgebra of $L(\FF_{t}).$ Inductively define the following algebras $N_{\alpha}$ for every ordinal $\alpha:$
\[N_{0}=N,\]
\[N_{\alpha}=W^{*}(\mathcal{H}_{s}(N_{\alpha-1}\subseteq M))\mbox{ if $\alpha$ is a sucessor ordinal},\]
\[N_{\alpha}=\overline{\bigcup_{\alpha'<\alpha}N_{\alpha'}}^{SOT}\mbox{ if $\alpha$ is a limit ordinal.}\]
Then for any ordinal $\alpha$ we have $N_{\alpha}\ne L(\FF_{t}).$
\end{theorem}
 The proof  of Theorem \ref{T:mainintro1} may be regarded as an abstraction of Voiculescu's original proof of absence of Cartan subalgebras in \cite{FreeEntropyDimensionIII}. The  new feature of the proof is that it is essentially ``linear'' and only relies on a basic fact about representations of $C^{*}$-algebras. It appears that these techniques have not been exploited so far in the study of free entropy dimension. The main exception is the case where $N$ is abelian, where  Voiculescu proved  in Corollary 7.6 of \cite{FreeEntropyDimensionIII}   that $\mathcal{H}_{s}(N\subseteq L(\FF_{t}))\ne L^{2}(M)$ (but not in this language). Note that Proposition \ref{P:spectralregularintro} shows that all previous notations of normalizers are contained in the singular subspace. So Theorem \ref{T:mainintro1} proves, in a very strong sense, the absence of chains of regular subalgebras  (even under a very weak notion of regularity) starting from a hyperfinite algebra and ending in $L(\FF_{t})$. In particular, we may easily deduce the conjecture of Galatan and Popa from Theorem \ref{T:mainintro1}.

\begin{cor}  Fix $t\in (1,\infty].$ For any diffuse, hyperfinite von Neumann subalgebra $N$ of $L(\FF_{t})$ the wq-normalizing algebra of $N$ is not $L(\FF_{t}).$
\end{cor}

The key tool in the proof of these theorems is Voiculescu's free entropy dimension.  In fact we may replace $L(\FF_{t})$ with any algebra which has microstates free-entropy dimension bigger than $1$ with respect to some set of generators. We can also replace``hyperfinite" in all of our theorems with any algebra which has microstates free entropy dimension one with respect to every set of generators in a suitably strong sense. The strong version of having  microstates free entropy dimension one with respect to every set of generators we will use is the concept of being strongly $1$-bounded developed by Kenley Jung in \cite{JungSB}. This led to the development of other related notions  in \cite{HadwinShen},\cite{HadwinLi},\cite{Shen}. These notions are numerical invariants of a von Neumann algebra which, when finite, imply that it has microstates free entropy dimension at most one with respect to every set of generators. We remark that being strongly $1$-bounded implicitly comes with a numerical invariant though this was not spelled out in  \cite{JungSB}. Specifically, it is clear from the proof of Theorem 3.2 in \cite{JungSB} that if $(M,\tau)$ is a finite von Neumann algebra, and $F,G\subseteq M_{sa}$ are finite sets with $W^{*}(F)=W^{*}(G)=M,$ then in the notation of Section $2$ of \cite{JungSB} we have
\[\sup_{\varepsilon>0}K_{\varepsilon}(\Xi(F))=\sup_{\varepsilon>0}K_{\varepsilon}(\Xi(G)).\]
We will call this common quantity the $1$-bounded entropy of $M$. We feel that this is justified because the above quantity can be regarded as a limit of relative entropies (i.e. it is computed in terms of packing numbers of microstates relative to a fixed self-adjoint). We show in Lemma \ref{L:orbit} that Shen's $1$-Embedding Dimension developed in \cite{Shen} agrees with $1$-bounded entropy.  We will need a version of $1$-bounded entropy which works for not necessarily finitely generated algebras. The techniques we use to define the $1$-bounded entropy for not necessarily finitely generated algebras are well known to experts in free entropy dimension (see e.g.  \cite{DimaCost},\cite{DimeFullIII}). We use $h(M)$ for the $1$-bounded entropy of $M$ (the precise definition is given in Definition \ref{D:1bbsection1ldga}), the $h$ being used because it is some form of entropy.  The proof that $h(M)$ is an invariant, even in the case of an infinite set of generators for $M,$ follows easily from the techniques of Jung. However, since this is  not present in the literature we have decided to include it in the appendix. It is clear from our definitions of the $1$-bounded entropy as well as section 2 of \cite{JungSB} that if $M$ is a finitely-generated von Neumann algebra, then $h(M)<\infty$ if and only if $M$ is strongly $1$-bounded. We prove this explicitly in Proposition \ref{P:equivstrongly1bddappendix}. From these comments it follows that if $h(M)<\infty,$ then $\delta_{0}(F)\leq 1$ for any finite $F\subseteq M_{sa}$ with $M=W^{*}(M)$ (for example, $h(L(\FF_{n}))=\infty$). The methods for defining microstates free entropy with respect to an infinite family of generators are well known (see e.g. Section 4 of \cite{DimeFullIII}) and it is clear from the definitions that if $h(M)<\infty,$ then $\delta_{0}(F)\leq 1$ for any set $F$ of self-adjoint generators of $M.$ This implies, for example, that $h(L(\FF_{\infty}))=\infty.$ A byproduct of our methods  shows that if $F$ is a $1$-bounded  set of self-adjoint elements and if $F$ is a nonamenability set, then $W^{*}(F)$ is strongly $1$-bounded (previously Jung required that $F$ contain an element with finite free entropy).

 For the reader's convenience, we mention more examples of algebras with microstates free entropy dimension bigger than one with respect to some set of generators.  By Proposition 6.8 of \cite{FreeEntropyDimensionIII}, any free product of diffuse, finitely generated tracial von Neumann algebras which embed into an ultrapower of $\R$ have free entropy dimension larger than $1$ with respect to some set of generators. Additionally, by Brown-Dykema-Jung (see \cite{DykemaJungBrown}), we can exhibit amalgamated free products which have free entropy dimension bigger than $1$ with respect to some set of generators.  Thus our theorems apply to each of these examples.

We now restate our main theorems using $1$-bounded entropy in a somewhat more abstract way.

\begin{theorem}\label{T:mainintro} Let $(M,\tau)$ be a tracial von Neumann algebra and suppose that $X\subseteq L^{2}(M)$ is a set of (not necessarily self-adjoint) generators for $M.$ Let $\omega\in\beta\NN\setminus\NN$ be a non-principal ultrafilter and let $N$  be a diffuse von Neumann subalgebra of $M^{\omega}.$ Set
\[\mathcal{H}=\overline{\sum_{x\in X}L^{2}(NxN)}^{\|\cdot\|_{2}}.\]
If $\mathcal{H}$ is disjoint from $L^{2}(N)\otimes L^{2}(N^{op})$ as an $N$-$N$ bimodule, then
\[h(M)\leq h(N).\]
Thus if $h(N)=0$ (e.g. $N$ is hyperfinite), then $h(M)=0$ (so for example $M$ cannot be isomorphic to $L(\FF_{n})$).
\end{theorem}
When $N$ is abelian, and $N\subseteq M,$ and $\mathcal{H}=L^{2}(M)$ the preceding theorem is covered by Corollary 7.6 in \cite{FreeEntropyDimensionIII}, with the statement $h(M)=0$ being replaced by $\delta_{0}(F)\leq 1$ for any finite $F\subseteq M_{sa}$ with $W^{*}(F)=M.$
\begin{cor}\label{T:introchains} Let $(M,\tau)$ be a tracial von Neumann algebra and suppose that $h(M)>0$ (e.g. if $\delta_{0}(F)>1$ for a finite $F\subseteq M_{sa}$ which generates $M$). Let $N\subseteq M$ be a diffuse von Neumann subalgebra with $h(N)=0.$ Inductively define algebras $N_{\alpha}$ for all ordinals $\alpha$ as follows:
\[N_{0}=N,\]
\[N_{\alpha}=W^{*}(\mathcal{H}_{s}(N_{\alpha-1}\subseteq M))\mbox{ if $\alpha$ is successor ordinal},\]
\[N_{\alpha}=\ovelrine{\bigcup_{\alpha'<\alpha}N_{\alpha'}}^{SOT},\mbox{ if $\alpha$ is a limit ordinal}.\]
Then for all $\alpha$ we have that $N_{\alpha}\ne M.$ In particular if $N$ hyperfinite and $M=L(\FF_{t}),t>1,$ then $N_{\alpha}\ne M$ for any ordinal $\alpha.$

\end{cor}


The conjecture of Galatan-Popa is motivated by their study of cohomology of $\textrm{II}_{1}$-factors in \cite{GalatanPopa},   in Section 5 of which they speculate on whether a``good'' cohomology theory for $\textrm{II}_{1}$ factors is possible. Such a theory should be nonvanishing, calculable in interesting cases, and in the case of group von Neumann algebras should reflect the cohomology of the group. Their discussion implies that if one can find a nice cohomology theory of $\textrm{II}_{1}$ factors, then it would follow that $L(\FF_{n})$ does not had a wq-regular, diffuse, hyperfinite von Neumann subalgebra. It thus becomes reasonable to try and prove that $L(\FF_{n})$ does not have a wq-regular, diffuse, hyperfinite von Neumann subalgebra as a ``consistency'' check for the postulate that there is a ``good'' cohomology theory for $\textrm{II}_{1}$ factors. As the preceding Theorem proves this fact, it still seems reasonable to  conjecture that there is a ``good'' cohomology theory for $\textrm{II}_{1}$-factors. In \cite{GalatanPopa} Galatan-Popa also introduce Property (C') for an inclusion $N\subseteq M$ of tracial von Neumann algebras. Property (C') is a version of asymptotic commutativity weaker than Property (C) as defined by Popa in \cite{PopaPropC} (which is itself weaker than having Property $(\Gamma)$). They show that if $M$ is a $\textrm{II}_{1}$-factor which has a diffuse, wq-regular von Neumann subalgebra having Property (C') in $M$, then $M$ has vanishing $1$-cohomology with values in any smooth bimodule. In particular, if $L(\FF_{n})$ had a smooth bimodule with non-vanishing $1$-cohomology, then $L(\FF_{n})$ could not have such a diffuse, wq-regular von Neumann subalgebra having property $(C')$ in $L(\FF_{n})$. We prove the absence of such a subalgebra of $L(\FF_{n})$ using $1$-bounded entropy. We remark here that Dykema already used free entropy dimension to prove that $L(\FF_{n})$ does not have Property (C) in \cite{DykemaFreeEntropy}.

\begin{cor} Let $(M,\tau)$ be a tracial von Neumann algebra. Suppose that $N$ is a diffuse, wq-regular subalgebra of $M.$ If $N\subseteq M$ has Property (C') as defined by Galatan-Popa in \cite{GalatanPopa}, then $h(M)\leq 0.$ In particular, $L(\FF_{n})$ does not have a diffuse wq-regular subalgebra $N$ so that $N\subseteq L(\FF_{n})$ has Property (C').
\end{cor}

In fact, in essentially every case, if Galatan-Popa show in \cite{GalatanPopa} that a  tracial von Neumann algebra $(M,\tau)$ has vanishing $1$-cohomology with values in any smooth bimodule, then we can show that $M$ has $1$-bounded entropy at most zero. The fact that algebras with  $1$-bounded entropy at most zero cannot be wq-regular, one-sided quasi-regular of step $\alpha$, etc. in algebras which have microstates free entropy bigger than $1$ with respect to some set of generators (e.g. $L(\FF_{n})$) may again be regarded as a good ``consistency'' check for the existence of a good cohomology theory of $\textrm{II}_{1}$-factors. The analogy between having vanishing $1$-cohomology and $1$-bounded entropy at most zero we view as similar to the theory of cost (as defined by Levitt in \cite{Levitt}) and $\ell^{2}$-Betti numbers (as defined by Atiyah in \cite{Atiyah}). For example, one knows by \cite{Gab2} that a group (or equivalence relation) with cost $1$  has vanishing first $\ell^{2}$-Betti number. The condition that a von Neumann algebra has Property (C') is analogous to Gaboriau's criterion (see \cite{Gab1} Critères VI.24.)  for a group to have cost $1$ if it can be generated by elements $s_{1},\dots,s_{k}$ with the property that $[s_{j},s_{j-1}]=1$ for all $j=2,\dots,k.$ We remark here that our techniques  have the defect that arguments involving free entropy dimension always have: any von Neumann algebra which has microstates free entropy dimension bigger than $1$ with respect to some set of generators must embed into an ultrapower of the hyperfinite $\textrm{II}_{1}$-factor. It is possible that the techniques of Galatan-Popa give more general conclusions since they do not require an algebra to embed into an ultrapower of the hyperfinite $\textrm{II}_{1}$-factor.

We can use our results  to give other examples of algebras $M$ with $h(M)\leq 0.$ In many cases by results of \cite{HoudayerShlyakhtenko} as well as \cite{Houdayerexotic} these algebras have the complete metric approximation property and are strongly solid. Thus they share many properties with free group factors, but are not isomorphic to them.

\begin{theorem}\label{T:qHoudShl} Let $q\in [-1,1]$ and let $\mathcal{H}$ be a real Hilbert space. Let $\Gamma_{q}(\mathcal{H})$ be the $q$-deformed free group factor defined by Bozjeko-Speicher in \cite{BozjekoSpeicher}. Suppose that $G$ is a countably infinite amenable group (more generally assume $h(L(G))<\infty$). Let $\pi\colon G\to \mathcal{O}(\mathcal{H})$ be an orthogonal representation and let $\alpha_{\pi}$ be the induced Bogulibov action of $G$ on $\Gamma_{q}(\mathcal{H}).$ Let
\[\pi_{\CC}\colon G\to \mathcal{U}(\mathcal{H}\otimes_{\RR}\CC)\]
be the complexification of $\pi.$ Let $\lambda_{C}\colon G\to \mathcal{U}(\ell^{2}(G))$ be the  representation induced by the conjugation action of $G$ on itself. Suppose that $\pi_{\CC}\otimes \lambda_{C}$ has no nonzero subrepresentation which embeds into the left regular representation. Then
\[h(\Gamma_{q}(\mathcal{H})\rtimes_{\alpha_{\pi}}G)\leq 0.\]
In particular, $\Gamma_{q}(\mathcal{H})\rtimes_{\alpha_{\pi}}G$  is not isomorphic to $L(\FF_{t})$ for any $t\in (1,\infty].$

\end{theorem}

Note that if there is a finite subset $F\subseteq\mathcal{H}$ with
\[\mathcal{H}=\overline{\Span\{\pi(g)\xi:g\in G,\xi\in F\}},\]
and if $G$ is finitely-generated, then $\Gamma_{q}(\mathcal{H})\rtimes_{\alpha_{\pi}}G$
is finitely-generated. Since finite generation and $h(\Gamma_{q}(\mathcal{H})\rtimes_{\alpha_{\pi}}G)<\infty$ is equivalent to being strongly $1$-bounded, we can use the above theorem to find new examples of strongly $1$-bounded algebras as defined by Jung.

	Consider the preceding theorem with $G=\ZZ,$ and $U=\pi(1)$ and let $U_{\CC}$ be the complexification of $U$.  Theorem \ref{T:qHoudShl} reduces to the statement that if the spectral measure of $U_{\CC}$ is singular with respect to the Lebesgue measure on $\TT=\RR/\ZZ,$ then
\[h(\Gamma_{q}(\mathcal{H})\rtimes_{\alpha_{\pi}}\ZZ)\leq 0.\]
In this case the non-isomorphism part of the previous theorem was only known (see Corollary A of  \cite{HoudayerShlyakhtenko}) when $q=0$ and there is a scalar measure $\nu$ in the same absolute continuity class of the spectral measure $U_{\CC}$ so that \emph{all} of the convolution powers of $\nu$ are singular with respect to Lebesgue measure. Here we are able to reduce to only the first convolution power because of our ability in Theorem \ref{T:mainintro1} to only assume that $\mathcal{H}_{s}(N\subseteq M)$ generates $M$ (and is not necessarily all of $L^{2}(M)$). One striking feature of our techniques is that the parameter $q$ plays \emph{no role whatsoever} in the proof of the preceding Theorem. Again, this is because the  parameter $q$ only plays a role in the higher-order structure of $\Gamma_{q}(\mathcal{H})$ and does not appear in the specific generating submodule we take. This illustrates the flexibility of our results.

	Let us mention how the preceding Theorem gives new examples of algebras which are strongly solid and have CMAP but are \emph{not} interpolated free group factors. Fix $p\in (2,\infty),$ by the Ivashev-Musatov theorem (see \cite{IvashevMusatov}), there is a symmetric probability measure $\mu$ on $\TT$ which is singular with respect to Lebesgue measure and has $\widehat{\mu}\in \ell^{p}(\ZZ).$ Consider the real Hilbert space:
\[\mathcal{H}=\{f\in L^{2}(\TT,\mu):f(-\theta)=\overline{f(\theta)}\}\]
and let $\pi\colon\ZZ\to \mathcal{O}(\mathcal{H})$
be the orthogonal representation given by
\[(\pi(n)f)(\theta)=e^{2\pi i n\theta}f(\theta).\]
Setting $U=\pi(1),$ it is easy to see that $U_{\CC}$ is unitarily equivalent to multiplication by $e^{2\pi i\theta}$ on $L^{2}(\TT,\mu).$ Thus $\mu$ is a measure in the same absolute continuity class as the spectral measure of $U_{\CC}.$
Since $\widehat{\mu}\in \ell^{p}(\ZZ),$ we know that $\pi$ is mixing and hence by \cite{HoudayerShlyakhtenko} we know that $\Gamma_{q}(\mathcal{H})\rtimes_{\alpha_{\pi}}\ZZ$
is strongly solid and has CMAP. As $\mu$ is singular with respect to Lebesgue measure we have
\[h(\Gamma_{q}(\mathcal{H})\rtimes_{\alpha_{\pi}}\ZZ)=0\]
and thus $\Gamma_{q}(\mathcal{H})\rtimes_{\alpha_{\pi}}\ZZ$ is not isomorphic to an interpolated free group factor.  Since $\widehat{\mu}\in l^{p}(\ZZ),$ we know that some convolution power of $\mu$ is absolutely continuous with respect to Lebesgue measure. Thus the fact that $\Gamma_{q}(\mathcal{H})\rtimes_{\alpha_{\pi}}\ZZ$ is not an interpolated free group factor is not covered by the results of \cite{HoudayerShlyakhtenko} (even when $q=0$).

Despite the fact that for most of the paper we stick to tracial von Neumann algebras (in order to use results on microstates free entropy dimension), we can deduce nontrivial consequences for semifinite von Neumann algebras. By applying Tomita-Takesaki theory, we can deduce consequences for certain type $\textrm{III}$ factors including the free Araki woods factors introduced by Shlyakhtenko in \cite{DimaFAW} and their $q$-deformations defined by Hiai in \cite{Hiai}. I thank Dimitri Shlyakhtenko for alerting me to this application.

Let $M$ be a von Neumann algebra with a faithful, normal, semifinite trace $\tau$. We can still make sense of what it means for a  set $X$ of measurable operators affiliated to $(M,\tau)$ to generate $M,$ mutatis mutandis, as in the case when $\tau$ is a finite trace.

\begin{theorem} Let $M$ be a semifinite von Neumann algebra and fix a faithful, normal, semifinite trace $\tau$ on $M.$ Suppose that $N\subseteq M$ is a diffuse von Neumann subalgebra so that $\tau\big|_{N}$ is still semifinite. If $p$ is a projection in $M$ with $\tau(p)<\infty,$ define $\tau_{p}\colon pMp\to\CC$ by
\[\tau_{p}(x)=\frac{\tau(x)}{\tau(p)}.\]
Suppose that for every projection $q\in N$ with $\tau(q)<\infty,$ we have $h(qNq,\tau_{q})\leq 0.$
Suppose additionally that there exists $X\subseteq L^{2}(M,\tau)$ such that
\[W^{*}(X)=M,\]
\[\overline{\sum_{x\in X}L^{2}(NxN)}\mbox{ is disjoint from $L^{2}(N)\otimes L^{2}(N^{op})$ as an $N$-$N$ bimodule}.\]
Then $h(pMp,\tau_{p})\leq 0$ for every $p\in M$ with $\tau(p)<\infty.$

\end{theorem}

\begin{cor} Let $q\in [-1,1]$ and let $\mathcal{H}$ be a real Hilbert space. Let $t\mapsto U_{t},t\in\RR$ be a one-parameter orthogonal group on $\mathcal{H}.$ Let $\Gamma_{q}(\mathcal{H},U_{t})''$ be the $q$-deformed free Araki-Woods algebra. Let $U_{t,\CC}$ be the complexification of $U_{t}$ and let $U_{t,\CC}=e^{itA}$ with $A$ a closed, self-adjoint operator. Suppose that the spectral measure of $A$ is singular with respect to Lebesgue measure. Then the continuous core of $\Gamma_{q}(\mathcal{H},U_{t})''$ is not isomorphic to $L(\FF_{s})\overline{\otimes}B(\mathcal{K})$ for any $s\in (1,\infty]$ and any separable Hilbert space $\mathcal{K}.$ In particular, $\Gamma_{q}(\mathcal{H},U_{t})''$ is not isomorphic to $\Gamma(L^{2}(\RR,m),\lambda_{t})''$ where $\lambda$ is the left regular representation and $m$ is Lebesgue measure.

\end{cor}

As in Theorem \ref{T:qHoudShl} the preceding Corollary was only known when $q=0$ and when there is a scalar measure $\nu$ in the absolute continuity class of the spectral measure of $A$ so that \emph{all} of its convolution powers are singular with respect to Lebesgue measure (see \cite{DimeFullIII}). Also, as in Theorem \ref{T:qHoudShl}, the parameter $q$ plays no role in the proof at all. By applying the same remarks after Theorem \ref{T:qHoudShl}, and in Section 4 of \cite{HoudayerShlyakhtenko} we find new examples of one-parameter orthogonal groups $U_{t}$ so that $\Gamma_{0}(\mathcal{H},U_{t})''\not\cong \Gamma_{0}(L^{2}(\RR,m),\lambda_{t})''.$

We mention one last application, related to the following question of Peterson.

\begin{question}\label{C:JPintro1}\emph{If $t\in(1,\infty]$ and $N$ is a finitely-generated, nonamenable von Neumann subalgebra of $L(\FF_{t}),$ does there exist a finite $F\subseteq N_{sa}$ so that $N=W^{*}(F)$ and $\delta_{0}(F)>1?$}\end{question}

Motivated by our results, we make a conjecture.

\begin{conjecture}\label{C:MeIntro1}\emph{If $t\in (1,\infty]$ and $N\subseteq L(\FF_{t})$ is a maximal amenable von Neumann subalgebra, then as $N$-$N$ bimodules:}
\[L^{2}(L(\FF_{t}))\ominus L^{2}(N)\leq [L^{2}(N)\otimes L^{2}(N)]^{\oplus \infty}.\]
\end{conjecture}

We can use Theorem \ref{T:main} to relate these two questions.

\begin{cor} If Question \ref{C:JPintro1} has an affirmative answer, then Conjecture \ref{C:MeIntro1} is true.
\end{cor}

The above corollary reveals that  it may be important to investigate the validity of Conjecture \ref{C:MeIntro1} in order to understand maximal amenable subalgebras of $L(\FF_{t}),t\in (1,\infty].$

	We remark that every known example of a MASA  $N\subseteq L(\FF_{n})$ which is also maximal amenable the conclusion of the preceding corollary is known.  Conjecture \ref{C:MeIntro1} is straightforward for the generator MASA (proven to be maximal amenable by Popa in \cite{PopaMaximalAmenable}). The radial MASA is known to be maximal amenable by work of  Cameron, Fang, Ravichandran and White (see \cite{CFRW}) using Popa's asymptotic orthogonality property developed in \cite{PopaMaximalAmenable}. Conjecture \ref{C:MeIntro1} for the radial MASA was verified by Sinclair-Smith in \cite{SinclairSmithLapalcian}. It is known by \cite{GJS},\cite{Hartglass} that planar algebras always complete to interpolated free group factors and, using Popa's asymptotic orthogonality property,  Brothier proved  in \cite{Brothier} that the cup subalgebra of a planar algebra is maximal amenable.  Conjecture \ref{C:MeIntro1} is known for the cup subalgebra of a  planar algebra by Theorem 4.9 of \cite{JSW}.  To the best of our knowledge, these are all known cases of maximal amenable subalgebras of interpolated free group factors. We mention that the introduction to  \cite{DykemaMurkherjee} discusses questions similar to Conjecture \ref{C:MeIntro1}.

\textbf{Acknowledgments} I thank Arnaud Brothier, Ionut Chifan, Marius Junge, Brent Nelson, Sorin Popa, Thomas Sinclair and Dimitri Shlyakhtenko for stimulating discussions related to this work. I thank Thomas Sinclair for suggesting Kenley Jung's work to me, this ultimately led to the foundations of this paper. I thank Stephanie Lewkiewicz for comments on an earlier version of this paper. I thank Chenxu Wen for alerting me to some very confusing typos in a previous version of this paper. Part of this work was done when I was visiting University of Iowa, as well as University of California, Los Angeles. I thank the University of Iowa and the University of California, Los Angeles for their hospitality and providing a stimulating environment in which to work. I would also like to thank the anonymous referees for their numerous comments which greatly improved the paper.

\subsection{Notational Remarks}\label{S:notation}

For a $*$-algebra $A$ we use $A_{sa}$ for the self-adjoint elements of $A.$ We use the phrase \emph{tracial von Neumann algebra} to mean a pair $(M,\tau)$ where $M$ is a von Neumann algebra and $\tau$ is a faithful, normal, tracial state on $M.$ For $x\in M,$ we use $\|x\|_{2}^{2}=\tau(x^{*}x),$ and we use $\|x\|_{\infty}$ for the operator norm on $x.$ If $A$ is a $C^{*}$-algebra (which is not endowed with a trace), we use $\|a\|$ for the norm of an element $a\in A.$ We thus reserve the notation $\|a\|_{\infty}$ for the case that $a$ is an element in some tracial von Neumann algebra (to distinguish it from one of the other noncommutative $L^{p}$-norms). This is similar to using $\|f\|$ for an element $f\in C(X)$ (when $C(X)$ has no given measure), and using $\|f\|_{\infty}$ for the $L^{\infty}$ norm of $f\in L^{\infty}(X,\mu)$ for some measure $\mu$ on $X.$ We will say that a normal element $a$ in a von Neumann algebra has \emph{diffuse spectrum} if $W^{*}(a)$ is diffuse. If $(a_{i})_{i\in I}$ are elements in some $C^{*}$-algebra $A$ a function $R\in [0,\infty)^{I}$ will be called a \emph{cutoff parameter} for $(a_{i})_{i\in I}$ if $\|a_{i}\|<R_{i}$ for all $i\in I.$
If $X\subseteq A,$ a function $R\in [0,\infty)^{X}$ will be called a cutoff parameter if it is a cutoff parameter for $(x)_{x\in X}.$  If $F,G$ are sets and $R\in [0,\infty)^{F},S\in [0,\infty)^{G}$ we let $R\vee S\in [0,\infty)^{F\sqcup G}$ be defined by
\[(R\vee S)_{x}=\begin{cases}
R_{x},& \textnormal{ if $x\in F$}\\
S_{x},& \textnormal{if $x\in G$.}
\end{cases}\]

	For a set $I,$ we will use $\CC\ip{X_{i}:i\in I}$ for the algebra of noncommutative complex polynomials in the $X_{i}$ (i.e. the free $\CC$-algebra in the variables $X_{i}$). We call elements of $\CC\ip{X_{i}:i\in I}$ \emph{noncommutative polynomials}. We give $\CC\ip{X_{i}:i\in I}$ the unique $*$-algebra structure which makes the $X_{i}$ self-adjoint. For self-adjoint elements $a_{i},i\in I$ in some $*$-algebra $A$, and $P\in \CC\ip{X_{i}:i\in I}$ we will denote by $P(a_{i}:i\in I)$
the image of $P$ under the unique $*$-homomorphism $\CC\ip{X_{i}:i\in I}\to A$
sending $X_{i}\to a_{i}.$ Note that this makes sense if $I$ is a subset of $A$ itself. Thus the expression $P(x:x\in I)$
is sensible if  $I\subseteq A_{sa},$ $P\in \CC\ip{X_{x}:x\in I}.$
 For $A\in M_{k}(\CC),$ we use
\[\tr(A)=\frac{1}{k}\sum_{j=1}^{k}A_{jj},\]
we shall also use
\[\|A\|_{2}=\tr(A^{*}A)^{1/2}.\]
Note that we can also make sense of $\|A\|_{2}$ for $A\in B(\mathcal{H}),$ with $\mathcal{H}$ a finite-dimensional Hilbert space. In case of potential confusion, for example when $\mathcal{H}=M_{k}(\CC)$ with the above Hilbert space norm, for a finite-dimensional Hilbert space $\mathcal{H}$ and $A\in B(\mathcal{H})$ we will sometimes use $\|A\|_{L^{2}(\tr_{\mathcal{H}})}$ instead of $\|A\|_{2}.$
For $I$ a finite set and $A\in M_{k}(\CC)^{I}$ we let
\[\|A\|_{2}=\left(\sum_{x\in I}\|A_{x}\|_{2}^{2}\right)^{1/2}.\]
and
\[\|A\|_{\infty}=\max_{x\in I}\|A_{x}\|_{\infty}.\]

	Suppose that $(X,d)$ is a metric space and $A,B\subseteq X.$ If $\delta>0$ we say that $A$ is \emph{$\delta$-contained} in $B,$ and write $A\subseteq_{\delta}B,$ if for every $x\in A$ there is a $y\in B$ with $d(x,y)\leq \delta.$ We say that $A\subseteq X$ is \emph{$\delta$-dense} if $X\subseteq_{\delta}A.$ We say that $A\subseteq X$ is \emph{$\delta$-separated} if for all $x\ne y$ in $A$ we have $d(x,y)>\delta.$ We let $K_{\delta}(X,d)$ be the minimal cardinality of a $\delta$-dense subset of $X.$ We leave it as an exercise to the reader to verify that if $A,B\subseteq X,$ if $\varpesilon,\delta>0$ and $A\subseteq_{\delta}B,$ then
\[K_{2(\varepsilon+\delta)}(A,d)\leq K_{\varepsilon}(B,d).\]
If $X$ is a Banach space, and $d$ is the metric induced by its norm and $A\subseteq X,$ we will typically use $K_{\varepsilon}(A,\|\cdot\|)$
instead of
$K_{\varepsilon}(A,d).$
Additionally, in case of ambiguity we will use
\[A\subseteq_{\delta,d} B\]
to specify the metric (with similar notation as above for when $d$ is induced by a norm).
If $(V,\|\cdot\|)$ is a normed vector space we use $\Ball(V,\|\cdot\|)=\{v\in V:\|v\|\leq 1\}.$

\section{Definition of $1$-Bounded Entropy and Some Technical Lemmas}

Let us first recall the definition of Voiculescu's microstates space.
\begin{definition}\label{D:1bbsection1ldga}\emph{Let $(M,\tau)$ be a tracial von Neumann algebra and $F$  a finite subset of $M_{sa}.$ For $m\in\NN,k\in\NN,\gamma>0$ define} Voiculescu's microstates space for $F,$ \emph{denoted $\Gamma(F;m,\gamma,k),$ to be the set of all $A\in M_{k}(\CC)_{sa}^{F}$ so that}
\[|\tr(P(A_{x}:x\in F))-\tau(P(x:x\in F))|<\gamma\]
\emph{for all monomials $P\in \CC\ip{X:X\in F}$ of degree at most $m.$  Let $R_{F}\in [0,\infty)^{F}$ be a cutoff parameter, we then set} $\Gamma_{R_{F}}(F;m,\gamma,k)$
\emph{to be the set of all $A\in\Gamma(F;m,\gamma,k)$ so that} $\|A_{x}\|_{\infty}\leq R_{F,x}$
\emph{for all $x\in F.$ If $G$ is another finite subset of $M_{sa}$ and $R_{G}\in [0,\infty)^{G}$  is a cutoff parameter we shall often write} $\Gamma_{R_{F}\vee R_{G}}(F,G;m,\gamma,k)$
\emph{instead of} $\Gamma_{R_{F}\vee R_{G}}(F\sqcup G;m,\gamma,k).$
\emph{We define} Voiculescu's microstates space for $F$ in the presence of $G$, \emph{denoted} $\Gamma_{R_{F}\vee R_{G}}(F:G;m,\gamma,k)$,
\emph{to be the set of all $A\in (M_{k}(\CC)_{sa})^{F}$ so that there is a $B\in (M_{k}(\CC)_{sa})^{G}$ with}
\[(A,B)\in \Gamma_{R_{F}\vee R_{G}}(F,G;m,\gamma,k).\]
\emph{We adopt similar notation for} $\Gamma(F,G;m,\gamma,k),$ $\Gamma(F:G;m,\gamma,k).$
\emph{For notational convenience, if $F\subseteq F_{0}\subseteq N_{sa},G\subseteq G_{0}\subseteq M_{sa}$ and $R_{F_{0}}\in [0,\infty)^{F_{0}},R_{G_{0}}\in [0,\infty)^{G_{0}}$ we will often use} $\Gamma_{R_{F_{0}}\vee R_{G_{0}}}(F:G;m,\gamma,k)$
\emph{for}
\[\Gamma_{R_{F_{0}}\big|_{F}\vee R_{G_{0}}\big|_{G}}(F:G;m,\gamma,k).\]
\end{definition}

Given a tracial von Neumann algebra $(M,\tau)$ and $a\in M_{sa},$ a sequence $(A_{k})_{k=1}^{\infty}$ with $A_{k}\in M_{k}(\CC)_{sa}$ is said to be a sequence of microstates for $a$ if $\sup_{k}\|A_{k}\|_{\infty}<\infty$ and for all $P\in \CC[x]$
\[\lim_{k\to\infty}\tr(P(A_{k}))=\tau(P(a)).\]

We now turn to our definition of $1$-bounded entropy.

\begin{definition}\emph{Let $(M,\tau)$ be a diffuse tracial von Neumann algebra. Fix a self-adjoint $a\in M$ with diffuse spectrum and a sequence $(A_{k})_{k=1}^{\infty}$ of microstates for $a.$ Let $X,Y\subseteq M_{sa}$ be  given and $R\in [0,\infty)^{X},R'\in [0,\infty)^{Y}$ be cutoff parameters and let $L=\sup_{k}\|A_{k}\|_{\infty}.$}
\emph{For finite $F\subseteq X,G\subseteq Y,$ natural numbers $m,k,$ and $\gamma>0$ we set}
\[\Xi_{A_{k},R\vee R'}(F:G;m,\gamma,k)=\{x\in M_{k}(\CC)_{sa}^{F}:(A_{k},x)\in \Gamma_{L\vee R\vee R'}(a,F:G;m,\gamma,k)\}.\]
\emph{For positive real numbers $\varepsilon,\gamma$ and natural numbers $m,k$ we then let}
\[K_{\varepsilon}(\Xi_{(A_{k})_{k=1}^{\infty},R\vee R'}(F:G;m,\gamma),\|\cdot\|_{2})=\limsup_{k\to\infty}\frac{1}{k^{2}}\log K_{\varepsilon}(\Xi_{A_{k}, R\vee R'}(F:G;m,\gamma,k),\|\cdot\|_{2}),\]
\[K_{\varepsilon}(\Xi_{(A_{k})_{k=1}^{\infty},R\vee R'}(F:G),\|\cdot\|_{2})=\inf_{\substack{m\in\NN,\\ \gamma>0}}K_{\varepsilon}(\Xi_{(A_{k})_{k=1}^{\infty},R\vee R'}(F:G;m,\gamma),\|\cdot\|_{2}),\]
\[K_{\varepsilon}(\Xi_{(A_{k})_{k=1}^{\infty},R\vee R'}(F:Y),\|\cdot\|_{2})=\inf_{G\subseteq Y\textnormal{ finite}}K_{\varepsilon}(\Xi_{(A_{k})_{k=1}^{\infty},R\vee R'}(F:G),\|\cdot\|_{2}),\]
\[h_{(A_{k})_{k=1}^{\infty}}(F:Y)=\sup_{\varepsilon>0}K_{\varepsilon}(\Xi_{(A_{k})_{k=1}^{\infty},R\vee R'}(F:Y),\|\cdot\|_{2}),\]
\[h_{(A_{k})_{k=1}^{\infty}}(X:Y)=\sup_{\textnormal{ finite} F\subseteq X}h_{(A_{k})_{k=1}^{\infty}}(F:Y).\]
\end{definition}
We also let $h_{(A_{k})_{k=1}^{\infty}}(F:G,\|\cdot\|_{\infty})$ be the number obtained by replacing every instance of $K_{\varepsilon}(\cdots,\|\cdot\|_{2})$ in the above definition with $K_{\varepsilon}(\cdots,\|\cdot\|_{\infty}).$ We will often write $(A_{k})$ instead of $(A_{k})_{k=1}^{\infty}.$  In the appendix, it is shown that if $a,b\in W^{*}(X)$ have diffuse spectrum and $(A_{k})_{k=1}^{\infty},(B_{k})_{k=1}^{\infty}$ are microstates for $a,b$ respectively, then
\[h_{(A_{k})_{k=1}^{\infty}}(X:Y)=h_{(B_{k})_{k=1}^{\infty}}(X:Y).\]
Thus we will simply write $h(X:Y)$
instead of
$h_{(A_{k})_{k=1}^{\infty}}(X:Y).$
In the appendix it is also shown that
\[h(X:Y)=h(W^{*}(X):W^{*}(Y)),\]
if $W^{*}(X)\subseteq W^{*}(Y),$ we will use this frequently without mention. All these facts are proved in a similar manner to \cite{JungSB},\cite{HadwinLi},\cite{Shen}. Lastly, we set
\[h(M)=h(M:M).\]
We call $h(M)$ the  $1$-bounded entropy of $M$ and $h(N:M)$ the $1$-bounded entropy of $N$ in the presence of $M.$  It is clear from our definition and \cite{JungSB} that if $M$ is finitely generated, then $h(M)<\infty$ if and only if $M$ is strongly $1$-bounded (see  Proposition \ref{P:equivstrongly1bddappendix} for a precise proof). We use the notation $h(N:M)$ because our definition is ``entropic" and we prefer to think of $h(M)$ as a  reasonable  entropy for strongly $1$-bounded algebras (since necessarily the free entropy dimension of such algebras is at most one with respect to every set of generators).

We list here some basic properties of $1$-bounded entropy for a tracial von Neumann algebra $(M,\tau)$ which will frequently be used throughout the paper (most of which are proved in Appendix \ref{S:properties1bddent}).

\begin{list}{Property \arabic{pcounter}: ~ }{\usecounter{pcounter}}
\item If $N\subseteq P\subseteq M$ are von Neumann algebras, then
\[h(N:M)\leq h(P:M) \mbox{ and }  h(N:M)\leq h(N:P).\]
\item If $N_{\alpha}$ is an increasing net of von Neumann subalgebras of $M$ and $N=\overline{\bigcup_{\alpha}N_{\alpha}}^{SOT},$ then $h(N:M)=\lim_{\alpha}h(N_{\alpha}:M)$ and $h(N)\leq \liminf_{\alpha}h(N_{\alpha}),$ (see Lemma \ref{L:increasingunions} and Corollary \ref{C:increasingunions}),
\item If $N_{j},j=1,2$ are von Neumann subalgebras of $M$ and $N_{1}\cap N_{2}$ is diffuse, then
\[h(N_{1}\vee N_{2}:M)\leq h(N_{1}:M)+h(N_{2}:M),\]
 (see Lemma \ref{L:diffuseintersection}).
\item If $(M_{j},\tau_{j})$ are tracial von Neumann algebras and $\mu_{j},j=2,\dots$ are such that $M=\oplus_{j=1}^{\infty}M_{j},$ $\sum_{j}\mu_{j}=1$ and if we define $\tau((x_{j})_{j})=\sum_{j}\mu_{j}\tau_{j}(x_{j}),$ then
\[h(M)\leq \sum_{j=1}^{\infty}\mu_{j}^{2}h(M_{j})\]
with the convention that if one of the terms on the right hand side is $-\infty,$ then the sum is $-\infty$  (see Proposition \ref{P:compressionformula}).
\item If $p\in M$ is a nonzero projection and $\tau_{p}\colon pMp\to\CC$ is defined by $\tau_{p}(x)=\frac{\tau(x)}{\tau(p)},$ then
\[h(pMp)\leq \frac{1}{\tau(p)^{2}}h(M)\]
(see Proposition \ref{P:compressionformula}).
\item If $F\subseteq M_{sa}$ is a finite set with $W^{*}(F)=M$ and $\delta_{0}(F)>1,$ then $h(M)=\infty$ (follows from the original definition of strongly $1$-bounded in \cite{JungSB}).
\end{list}

 	The fact that $h(M)$ makes sense for von Neumann algebras which are not a priori finitely generated will be particularly useful in this paper. Our main goal is to show-- even under very weak notions of regularity-- that algebras which  have microstates free entropy dimension bigger than $1$ with respect to some set of generators (e.g. $L(\FF_{n})$)  cannot have  subalgebras with finite $1$-bounded entropy which are regular. Moreover, we will want to rule out chains of algebras
\begin{equation}\label{E:chinssection2}
N_{0}\subseteq N_{1}\subseteq N_{2}\subseteq \cdots N_{\alpha}=M
\end{equation}
so that $N_{0}$ is strongly $1$-bounded and $N_{\beta}\subseteq N_{\beta+1}$ satisfies some weak regularity condition. When we consider chains of algebras as in $(\ref{E:chinssection2}),$ we will have no a priori control as to whether $N_{1},N_{2},\cdots$ are finitely-generated. However, we will still want to prove nonexistence of such a chain, by showing that $h(N_{\alpha})<\infty$ for every $\alpha$ (in fact one will have to get a uniform bound on $h(N_{\alpha})$ for our inductive arguments to work). Thus it will be useful to drop the standard assumption of finite generation and work with this extended notion of being strongly $1$-bounded. This is similar to the point of view of Hadwin-Li in \cite{HadwinLi}.

We need an elementary lemma. If $A$ is a $C^{*}$-algebra, if $\mathcal{H},\mathcal{K}$ are Hilbert spaces, and $\pi\colon A\to B(\mathcal{H}),\rho\colon A\to B(\mathcal{K})$ are $*$-homomorphisms we say that $\pi$ and $\rho$ are disjoint if $\pi,\rho$ do not have nonzero, isomorphic subrepresentations. We write $\pi\perp\rho$ if $\pi,\rho$ are disjoint. Lastly, we let $\Hom_{A}(\pi,\rho)$ be the space of bounded, linear, $A$-equivariant maps $T\colon \mathcal{H}\to \mathcal{K}.$ The following lemma is well known, but we include the short proof for completeness.

\begin{lemma}\label{L:disjointness} Let $A$ be a $C^{*}$-algebra, and $\mathcal{H}_{j},j=1,2$ Hilbert spaces. Let $\pi_{j}\colon A\to B(\mathcal{H}_{j}),j=1,2$ be $*$-homomorphisms. The following are equivalent:

(i): $\pi_{1}\perp \pi_{2},$

(ii): $\Hom_{A}(\pi_{1},\pi_{2})=\{0\},$

(iii): $\Hom_{A}(\pi_{2},\pi_{1})=\{0\},$

(iv): For any $\varepsilon>0$ and any $\xi_{1},\dots,\xi_{k}\in\mathcal{H}_{1},\eta_{1},\dots,\eta_{l}\in \mathcal{H}_{2},$  there is an $a\in A$ such that $\|a\|\leq 1$ and
\[\max_{1\leq j\leq  k}\|\pi_{1}(a)\xi_{j}-\xi_{j}\|<\varepsilon,\]
\[\max_{1\leq j\leq l}\|\pi_{2}(a)\eta_{j}\|<\varepsilon.\]

\end{lemma}

\begin{proof} The equivalence of $(i)$ and $(ii)$ is a standard exercise in using the polar decomposition. The equivalence of $(ii)$ and $(iii)$ is clear by taking adjoints. It is also easy to show that $(iv)$ implies $(ii).$ It thus remains to show $(ii)$ and $(iii)$ imply $(iv).$

So assume that $(ii)$ and $(iii)$ hold. We claim that
\begin{equation}\label{E:closureclaimsection2alshgla}
\begin{bmatrix}
1& 0\\
0& 0
\end{bmatrix}\in \overline{(\pi_{1}\oplus \pi_{2})(A)}^{SOT}.
\end{equation}
Suppose that we can show (\ref{E:closureclaimsection2alshgla}). Given  $\xi_{1},\dots,\xi_{k}\in\mathcal{H}_{1},\eta_{1},\dots,\eta_{l}\in \mathcal{H}_{2},$   Kaplansky's density theorem implies that we can find an $a\in A$ so that
\[\max(\|\pi_{1}(a)\|,\|\pi_{2}(a)\|)\leq 1,\]
\[\|\pi_{1}(a)\xi_{j}-\xi_{j}\|<\varepsilon\mbox{ for $j=1,\dots,k$,}\]
\[\|\pi_{2}(a)\eta_{j}\|<\varepsilon\mbox{ for $j=1,\dots,l.$}\]
Let $J=\ker(\pi_{1}\oplus \pi_{2}).$ Because injective $*$-homomorphisms are isometric, $\pi_{1}\oplus \pi_{2}$ induces an isometric inclusion $A/J\to B(\mathcal{H}_{1}\oplus \mathcal{H}_{2}).$
 So by definition of the quotient norm, we can choose $a$ as above with $\|a\|<1+\varepsilon.$
Setting $b=\frac{a}{1+\varespilon}$ we see that $\|b\|\leq 1$ and
\[\|\pi_{1}(b)\xi_{j}-\xi_{j}\|<\varespilon+\left|1-\frac{1}{1+\varepsilon}\right|\|\xi_{j}\|\mbox{ for $j=1,\dots,k$,}\]
\[\|\pi_{2}(b)\eta_{j}\|<\varepsilon+\left|1-\frac{1}{1+\varepsilon}\right|\|\eta_{j}\|\mbox{ for $j=1,\dots,l$},\]
and since $\varepsilon$ is arbitrary, $(iv)$ is an easy consequence of this. So it is enough to prove (\ref{E:closureclaimsection2alshgla}).

Let $T\in (\pi_{1}\oplus\pi_{2})(A)',$ by the Double Commutant Theorem it is enough to show that
\[\begin{bmatrix}
1& 0\\
0& 0
\end{bmatrix}\]
commutes with $T.$ Write
\[T=\begin{bmatrix}
T_{11}&T_{12}\\
T_{21}&T_{22}
\end{bmatrix},\]
with $T_{ij}\in B(\mathcal{H}_{i},\mathcal{H}_{j}).$ Since $T\in (\pi_{1}\oplus\pi_{2})(A)',$ we have $T_{ij}\in \Hom_{A}(\pi_{i},\pi_{j})$ and thus $T_{12},T_{21}=0.$ We now know that
\[T=\begin{bmatrix}
T_{11}&0\\
0&T_{22}
\end{bmatrix}.\]
It is now easy to see that $T$ commutes with
\[\begin{bmatrix}
1& 0\\
0& 0
\end{bmatrix}.\]

\end{proof}

It is trivial to see that we can take $a$ as in $(iv)$ to live in a prescribed dense subset of $A.$ We will use this in the sequel without comment. We will also need the following volume-packing estimate. We use $\ell^{2}(k)$ for $\CC^{k}$ with the $\ell^{2}$-norm:
\[\|z\|_{2}^{2}=\sum_{j=1}^{k}|z_{j}|^{2}.\]
We also use $\Ball(\ell^{2}(k))=\{\xi\in\ell^{2}(k):\|\xi\|_{2}\leq 1\}.$

\begin{lemma}\label{L:volumepackingestimate} Let $R,\eta,\varepsilon>0.$ Then for any $T\in M_{k}(\CC)$ with $\|T\|_{\infty}\leq 1$
we have
\[K_{2\varepsilon}(T(R\Ball(\ell^{2}(k))),\|\cdot\|_{2})\leq \left(\frac{3R+\varepsilon}{\varpesilon}\right)^{2kR^{2}\frac{\|T\|_{2}^{2}}{\varepsilon^{2}}}.\]

\end{lemma}

\begin{proof} Set $\eta=\|T\|_{2}$ and let $P=\chi_{(\frac{\varepsilon}{R},1]}(|T|).$
Then
\[\tr(P)\leq \frac{R^{2}}{\varepsilon^{2}}\tr(T^{*}T)=\frac{R^{2}\eta^{2}}{\varepsilon^{2}}.\]
Thus $P\ell^{2}(k)$ is a space of real dimension at most $2\frac{R^{2}\eta^{2}}{\varepsilon^{2}}k.$ Let $S$ be a maximal subset of $R\Ball(P\ell^{2}(k)),\|\cdot\|_{2})$ satisfying the condition that $\|A-B\|_{2}\geq \varespilon$ for all $A\ne B$ with $A,B\in S.$ By the triangle inequality we have
\[\left(R+\frac{\varepsilon}{3}\right)\Ball(P\ell^{2}(k),\|\cdot\|_{2})\supseteq \bigcup_{\xi\in S}\xi+\frac{\varepsilon}{3}\Ball(P\ell^{2}(k),\|\cdot\|_{2})\]
and since $S$ is $\varepsilon$-separated the right-hand side of the above equation is a disjoint union.  Computing volumes shows that
\[\left(R+\frac{\varepsilon}{3}\right)^{2\tr(P)k}\geq |S|\left(\frac{\varepsilon}{3}\right)^{2\tr(P)k},\]
so
\[|S|\leq \left(\frac{3R+\varepsilon}{\varepsilon}\right)^{2\tr(P)k}\leq\left(\frac{3R+\varepsilon}{\varepsilon}\right)^{2\frac{R^{2}\eta^{2}}{\varepsilon^{2}}k}.\]
Now let $\xi\in R\Ball(\ell^{2}(k))$ and choose a $\xi'\in S$ so that
\[\|P\xi-\xi'\|_{2}\leq \varepsilon.\]
Then
\[\|T\xi-T\xi'\|_{2}\leq \|TP\xi-T\xi'\|+\|T(1-P)\xi\|_{2}\leq\varpesilon+\varepsilon=2\varepsilon,\]
so
\[K_{2\varpesilon}(T(R\Ball(\ell^{2}(k))),\|\cdot\|_{2})\leq |S|\leq \left(\frac{3R+\varepsilon}{\varpesilon}\right)^{2\frac{R^{2}\eta^{2}}{\varepsilon^{2}}k}.\]

\end{proof}
We will need a technical lemma which will allow us to switch from a $\varepsilon$-dense set with respect to $\|\cdot\|_{2}$ to a $\varepsilon$-dense with respect to $\|\cdot\|_{\infty}.$ We will use this to show that
\[h(X:Y,\|\cdot\|_{\infty})=h(X:Y),\]
which will be important in the proof of Theorem \ref{T:mainintro1}.
\begin{lemma}\label{L:switching} There is a universal constant $C>0$ with the following property. For any $\varepsilon>0$ there is  a $t_{0}>0$ depending only upon $\varepsilon$ so that for any $0<t<t_{0},$  any $k\in\NN,$  and any  $\Omega\subseteq R\Ball(M_{k}(\CC),\|\cdot\|_{\infty})$ we have
\[K_{2(2R+2)\varespilon}(\Omega,\|\cdot\|_{\infty})\leq k\frac{t^{2}}{\varepsilon^{2}}\left(\frac{C}{t}\right)^{2k^{2}\frac{t^{2}}{\varepsilon^{2}}(1-\frac{t^{2}}{\varepsilon^{2}})}\left(\frac{3R+\varepsilon}{\varepsilon}\right)^{2k^{2}\frac{t^{2}}{\varepsilon^{2}}}K_{t}(\Omega,\|\cdot\|_{2}).\]

\end{lemma}
\begin{proof} By a result of S. Szarek (see \cite{Szarek}), there is a $C>0$ so that for all $k,l\in \NN$ and any $\delta>0$ we have
\[K_{\delta}(\Gr(l,k-l),\|\cdot\|_{\infty})\leq \left(\frac{C}{\delta}\right)^{l(k-l)},\]
here $\Gr(l,k-l)$ is the space of orthogonal projections in $M_{k}(\CC)$ of rank $l.$ Choose $t_{0}$ with $0<t_{0}<\min(C,\varepsilon)$ and so that $0<t<t_{0}$ implies $\frac{t^{2}}{\varepsilon^{2}}<\frac{1}{2}.$ Fix $k\in\NN,$ and $\Omega\subseteq R\Ball(M_{k}(\CC),\|\cdot\|_{\infty}).$   Choose a $t$-dense set $S\subseteq\Omega$ with respect to $\|\cdot\|_{2}$ with $|S|=K_{t}(\Omega,\|\cdot\|_{2}).$
 Given $A\in \Omega,$ we may choose an $A'\in S$ so that $\|A-A'\|_{2}\leq t.$
Set $P'=\chi_{[0,\varepsilon]}(|A-A'|)$ and observe that
\[\tr(1-P')\leq \frac{1}{\varepsilon^{2}}\tr(|A-A'|^{2})\leq \frac{t^{2}}{\varepsilon^{2}}.\]
For every integer $l$ with $1\leq l\leq\frac{t^{2}}{\varespilon^{2}}k$ choose a $\varepsilon$-dense subset $G_{l}$ of $\Gr(l,k-l)$ with \[|G_{l}|=K_{\varepsilon}(\Gr(l,k-l),\|\cdot\|_{\infty}).\]
As the function $\phi(x)=x(1-x)$ is increasing for $0<x<1/2,$ it follows that for all integers $1\leq l\leq k\frac{t^{2}}{\varepsilon^{2}}$ we have
\[|G_{l}|\leq\left(\frac{C}{t}\right)^{2k^{2}\frac{l}{k}\left(1-\frac{l}{k}\right)}\leq \left(\frac{C}{t}\right)^{2k^{2}\frac{t^{2}}{\varepsilon^{2}}(1-\frac{t^{2}}{\varepsilon^{2}})}.\]
For every integer $1\leq l\leq k\frac{t^{2}}{\varepsilon^{2}},$ and every $P\in G_{l}$ choose a maximal subset $D_{P}$ of $R\Ball(M_{k}(\CC)P,\|\cdot\|_{\infty})$ subject to the condition that $\|A-B\|_{\infty}\geq \varepsilon$ for all $A,B\in D_{P}$ with $A\ne B.$
By the triangle inequality we have
\[\left(R+\frac{\varepsilon}{3}\right)\Ball(M_{k}(\CC)P,\|\cdot\|_{\infty})\supseteq \bigcup_{\xi\in D_{P}}\xi+\frac{\varepsilon}{3}\Ball(M_{k}(\CC)P,\|\cdot\|_{\infty}),\]
since $D_{P}$ is $\varepsilon$-separated the right-hand side of the above equation is a disjoint union.  Computing volumes shows that
\[\left(R+\frac{\varepsilon}{3}\right)^{2\tr(P)k^{2}}\geq |D_{P}|\left(\frac{\varepsilon}{3}\right)^{2\tr(P)k^{2}}.\]
As $\tr(P)=\frac{l}{k},$ we have
\[|D_{P}|\leq \left(\frac{3R+\varepsilon}{\varepsilon}\right)^{2kl}\leq\left(\frac{3R+\varepsilon}{\varespilon}\right)^{2k^{2}\frac{t^{2}}{\varepsilon^{2}}}.\]
Note that, by maximality, we have that $D_{P}$ is $\varepsilon$-dense in $R\Ball(M_{k}(\CC)P,\|\cdot\|_{\infty})$ with respect to $\|\cdot\|_{\infty}.$

	Now choose an integer $l$ with $1\leq l\leq \frac{t^{2}}{\varepsilon^{2}}k$ so that $k\tr(1-P')=l$ and let $P\in G_{l}$ be such that
\[\|P-(1-P')\|_{\infty}\leq \varepsilon.\]
Thus
\[\|P'-(1-P)\|_{\infty}\leq \varepsilon,\]
so
\[\|A(1-P)-A'(1-P)\|_{\infty}\leq 2R\varepsilon+\|(A-A')P'\|_{\infty}\leq (2R+1)\varepsilon.\]
Choose a $C\in D_{P}$ so that
\[\|AP-C\|_{\infty}\leq\varepsilon,\]
since $CP=C$ we have
\[\|A-A'(1-P)-C\|_{\infty}\leq \|AP-C\|_{\infty}+\|A(1-P)-A'(1-P)\|_{\infty}\leq (2R+2)\varepsilon.\]
Since $A$ was arbitrary we see that
\[\Omega\subseteq_{(2R+2)\varepsilon,\|\cdot\|_{\infty}}\bigcup_{l=1}^{\lfloor{k\frac{t^{2}}{\varepsilon^{2}}\rfloor}}\bigcup_{P\in G_{l}}S(1-P)+D_{P}\]
and thus
\begin{align*}
K_{2(2R+2)\varepsilon}(\Omega,\|\cdot\|_{\infty})&\leq\sum_{l=1}^{\lfloor{k\frac{t^{2}}{\varepsilon^{2}}\rfloor}}\sum_{P\in G_{l}}|S||D_{P}|\\
&\leq k\frac{t^{2}}{\varepsilon^{2}}\left(\frac{C}{t}\right)^{2k^{2}\frac{t^{2}}{\varpesilon^{2}}(1-\frac{t^{2}}{\varepsilon^{2}})}\left(\frac{3R+\varepsilon}{\varepsilon}\right)^{2k^{2}\frac{t^{2}}{\varepsilon^{2}}}|S|\\
&=k\frac{t^{2}}{\varepsilon^{2}}\left(\frac{C}{t}\right)^{2k^{2}\frac{t^{2}}{\varepsilon^{2}}(1-\frac{t^{2}}{\varepsilon^{2}})}\left(\frac{3R+\varepsilon}{\varepsilon}\right)^{2k^{2}\frac{t^{2}}{\varepsilon^{2}}}K_{t}(\Omega,\|\cdot\|_{2})).
\end{align*}

\end{proof}

The relevant fact about Lemma \ref{L:switching} will be that
\[\lim_{t \to 0}\lim_{k\to \infty}\frac{1}{k^{2}}\log\left(k\frac{t^{2}}{\varepsilon^{2}}\left(\frac{C}{t}\right)^{2k^{2}\frac{t^{2}}{\varepsilon^{2}}(1-\frac{t^{2}}{\varepsilon^{2}})}\left(\frac{3R+\varepsilon}{\varepsilon}\right)^{2k^{2}\frac{t^{2}}{\varepsilon^{2}}}\right)=0\]
when $\varepsilon$ is fixed. We use this to show that the computation of $h(M:N)$ can be done by replacing $K_{\varepsilon}(\cdots,\|\cdot\|_{2})$ with $K_{\varepsilon}(\cdots,\|\cdot\|_{\infty}).$ More precisely, we have the following.
\begin{cor}\label{C:switching} Let $(M,\tau)$ be a tracial von Neumann algebra, and let $N\subseteq P$ be diffuse von Neumann subalgebras. Then
\[h(N:P)=h(N:P,\|\cdot\|_{\infty}).\]
\end{cor}

\begin{proof} Let $F,G$ be finite sets of self-adjoint elements in $N,P$ with $W^{*}(F)\subseteq W^{*}(G)$ and fix cutoff parameters $R_{F}\in [0,\infty)^{F},R_{P}\in [0,\infty)^{P}.$
Fix an element $a\in N_{sa}$ so that $W^{*}(a)$ is diffuse and a sequence $(A_{k})_{k=1}^{\infty}$ of microstates for $a.$ Let $L=\sup_{k}\|A_{k}\|_{\infty}$
and $\varespilon>0$ and set
\[R=\max(\max_{a\in F}R_{F,a},\max_{b\in G}R_{P,b},L).\]
Let $C,t_{0}>0$ be as in the preceding Lemma for $\varepsilon>0.$ Note that we may regard
\[\Xi_{(A_{k}),R_{F}\vee R_{P}}(F:G;m,\gamma,k)\subseteq M_{k}(\CC)^{F}\subseteq B(\ell^{2}(\{1,\dots,k\}\times F)).\]
Under this identification, for any $A\in M_{k}(\CC)^{F}$ we have that $\|A\|_{2}$ (as defined in Section \ref{S:notation}) equals $\sqrt{|F|}\|A\|_{L^{2}(\tr_{\ell^{2}(\{1,\dots,k\}\times F)})}.$ Using these remarks it follows from Lemma \ref{L:volumepackingestimate} that for $0<t<t_{0}$ and every $m,k\in\NN,\gamma>0$ we have
\begin{align*}
K_{2(2R+2)\varepsilon}(\Xi_{(A_{k}),R_{F}\vee R_{P}}(F:G;m,\gamma,k),\|\cdot\|_{\infty})&\leq k\frac{t^{2}}{|F|\varepsilon^{2}}\left(\frac{C\sqrt{|F|}}{t}\right)^{2k^{2}\frac{t^{2}}{|F|\varepsilon^{2}}\left(1-\frac{t^{2}}{|F|\varepsilon^{2}}\right)}\left(\frac{3R+\varepsilon}{\varepsilon}\right)^{2k^{2}\frac{t^{2}}{|F|\varepsilon^{2}}}\\
&\times K_{t}(\Xi_{(A_{k}),R_{F}\vee R_{P}}(F:G;m,\gamma,k),\|\cdot\|_{2}).
\end{align*}
Taking $\log$ of both sides, dividing by $k^{2}$ and letting $k\to \infty$ and then taking the infimum over $m,\gamma$ shows that
\begin{align*}
K_{2(2R+2)\varepsilon}(\Xi_{(A_{k}),R_{F}\vee R_{P}}(F:G),\|\cdot\|_{\infty})&\leq 2\frac{t^{2}}{|F|\varepsilon^{2}}\left(1-\frac{t^{2}}{|F|^{2}\varepsilon^{2}}\right)\log\left(\frac{C\sqrt{|F|}}{t}\right)+2\frac{t^{2}}{\varepsilon^{2}|F|}\log\left(\frac{3R+\varepsilon}{\varepsilon}\right)\\
&+K_{t}(\Xi_{(A_{k}),R_{F}\vee R_{P}}(F:G),\|\cdot\|_{2})\\
&\leq  2\frac{t^{2}}{|F|\varepsilon^{2}}\left(1-\frac{t^{2}}{|F|^{2}\varepsilon^{2}}\right)\log\left(\frac{C\sqrt{|F|}}{t}\right)+2\frac{t^{2}}{\varepsilon^{2}|F|}\log\left(\frac{3R+\varepsilon}{\varepsilon}\right)\\
&+K_{t}(\Xi_{(A_{k}),R_{F}\vee R_{P}}(F:P),\|\cdot\|_{2})
\end{align*}
for all sufficiently small $t.$ Letting $t\to 0$ shows that
\[K_{2(2R+2)\varepsilon}(\Xi_{(A_{k}),R_{F}\vee R_{G}}(F:G),\|\cdot\|_{\infty})\leq h(F:P,\|\cdot\|_{2}).\]
Taking the infimum over $G$ and letting $\varepsilon\to 0$ proves that $h(F:P)\leq h(F:P,\|\cdot\|_{\infty}).$
We can now take the supremum over all $F$ to see that $h(N:P)\leq h(N:P,\|\cdot\|_{\infty}).$
Since it is trivial that $h(N:P,\|\cdot\|_{\infty})\leq h(N:P),$ the proof is complete.

\end{proof}

\section{Proof of  The Main Result}

For the proof of the main result we need some more terminology.  Let $I$ be a set and $ P\in \CC\ip{X_{i}:i\in I}\otimes_{\textnormal{alg}} \CC\ip{X_{i}:i\in I}^{\op}.$ For any $C^{*}$-algebra $A$ and $(a_{i})_{i\in I}\in (A_{sa})^{I},$ let $P(a_{i}:i\in I)\in A\otimes_{\textnormal{alg}} A^{\op}$
be the image of $P$ under the unique $*$-homomorphism sending $X_{i}\otimes 1$ to $a_{i}\otimes 1$ and $1\otimes X_{j}^{\op}$ to $1\otimes a_{j}^{\op}$ for $i,j\in I.$
 For a given $R\in [0,\infty)^{I}$ we set
 \[\|P\|_{R,\infty}=\sup_{a_{i}} \|P(a_{i}:i\in I)\|,\]
where the supremum is over all $(a_{i})_{i\in I}\in (A_{sa})^{I}$ where
\begin{itemize}
\item  $A$ is some $C^{*}$-algebra, \\
\item $\|a_{i}\|_{\infty}\leq R_{i}$ for all $i\in I,$\\
\item $A\otimes_{\textnormal{alg}} A^{\op}$ is endowed with the maximal tensor product norm.
\end{itemize}
 It is easy to see that $\|P\|_{R,\infty}$ is finite for every $P$ and that $\|P\|_{R,\infty}$ is a $C^{*}$-norm. We let \[C_{R}\ip{X_{i}\otimes X_{j}^{\op}:i,j\in I}\]
be the completion of $\CC\ip{X_{i}:i\in I}\otimes \CC\ip{X_{i}:i\in I}^{\op}$ under this norm, so
$C_{R}\ip{X_{i}\otimes X_{j}^{\op}:i,j\in I}$
is naturally a $C^{*}$-algebra. One can easily see that
\[C_{R}\ip{X_{i}\otimes X_{j}^{\op}:i,j\in I})\cong (\Large{*}_{i\in I}C([-R_{i},R_{i}]))\otimes_{\textnormal{max}}(\Large{*}_{i\in I}C([-R_{i},R_{i}])),\]
where the free product in question is the \emph{full} free product, but this be irrelevant for us. We only care about the universal property of $C_{R}\ip{X_{i}\otimes X_{j}^{\op}:i,j\in I},$ that given self-adjoints elements $a_{i}$ in some $C^{*}$-algebra $A$ with $\|a_{i}\|\leq R_{i},$  there exists a unique homomorphism
\[C_{R}\ip{X_{i}\otimes X_{j}^{\op}:i,j\in I}\to A\otimes_{\textnormal{max}}A^{\op}\]which sends  $X_{i}\otimes 1$ to $a_{i}\otimes 1$ and $1\otimes X_{j}^{\op}$ to $1\otimes a_{j}^{\op}$ for $i,j\in I.$

Let $(M,\tau)$ be a tracial von Neumann algebra and $\mathcal{H}$ an $M-M$ bimodule. For $x\in M\otimes_{\textnormal{alg}}M^{\op}$ and $\xi\in\mathcal{H},$ we let
\[x\# \xi=\sum_{j=1}^{k}a_{j}\xi b_{j}\]
if
\[x=\sum_{j=1}^{k}a_{j}\otimes b_{j}^{\op}.\]
Note that if $(x_{i})_{i\in I}\in (M_{sa})^{I}$ and  $R\in [0,\infty)^{I}$ is a cutoff constant for the $x_{i},$ then for any \[P\in \CC\ip{X_{i}:i\in I}\otimes_{\textnormal{alg}} \CC\ip{X_{i}:i\in I}^{\op}\]
 we have
\[\|P(x_{i}:i\in I)\#\xi\|\leq \|P\|_{R,\infty}\|\xi\|,\]
this is automatic from the definition of $\|P\|_{R,\infty}.$

Our goal is to show that if $N\subseteq M$ are von Neumann algebras, if $N$ has $1$-bounded entropy at most zero, and $N$ is regular in $M$ in a very weak sense, then $M$  has $1$-bounded entropy at most zero. To state our results nicely we introduce the following canonical $N$-$N$ submodule of $L^{2}(M).$
\begin{definition}\emph{Let $(M,\tau)$ be a tracial von Neumann algebra and let $N$ be a von Neumann subalgebra of $M.$ We define the} singular subspace of $L^{2}(M)$ over $N$\emph{ by}
\[\mathcal{H}_{s}(N\subseteq M)=\{\xi\in L^{2}(M):L^{2}(N\xi N)\mbox{ \emph{is disjoint from} } L^{2}(N)\otimes L^{2}(N^{op})\mbox{ \emph{as an $N$-$N$ bimodule}}\}.\]
\end{definition}In this definition, we are regarding $N$-$N$ bimodules as representations of $N\otimes_{\max}N^{op}$ and using the definition of disjointness introduced before Lemma \ref{L:disjointness}. In order to motivate the definition we prove the following proposition, which shows that $\mathcal{H}_{s}(N\subseteq M)$ contains all of the other weak versions of the normalizer we have discussed.

\begin{proposition}\label{P:weaknormalizer} Let $(M,\tau)$ be a tracial von Neumann algebra and let $N\subseteq M$ be a diffuse von Neumann subalgebra. Viewing $M\subseteq L^{2}(M,\tau)$ we have
\[q^{1}\mathcal{N}_{M}(N)\subseteq \mathcal{H}_{s}(N\subseteq M),\]
\[\mathcal{N}^{wq}_{M}(N)\subseteq \mathcal{H}_{s}(N\subseteq M).\]

\end{proposition}

\begin{proof}

 Given $x\in q^{1}\mathcal{N}_{M}(N),$ choose $x_{1},\dots,x_{k}\in M$ so that
\[xN\subseteq\sum_{j=1}^{n}Nx_{j}.\]
Notice that
\[NxN\subseteq\sum_{j=1}^{n}Nx_{j}\]
and this clearly implies that  $L^{2}(NxN)$ has finite right dimension over $N.$ Thus any $N$-$N$ subbimodule of $L^{2}(NxN)$ must have finite right dimension over $N$ and, as $N$ is diffuse, clearly no such subbimodule can be embedded as an $N$-$N$ bimodule into $L^{2}(N)\otimes L^{2}(N^{op})$ unless it is zero.

	For the second statement, let $u\in \mathcal{N}^{wq}_{M}(N),$ and let $T\colon L^{2}(NuN)\to L^{2}(N)\otimes L^{2}(N^{op})$
be a bounded $N$-$N$ bimodular map. Since $u^{*}Nu\cap N$ is diffuse, we can find a sequence $v_{k}\in \mathcal{U}(u^{*}Nu\cap N)$ so that $v_{k}\to 0$ in the weak operator topology. Let $v_{k}=u^{*}w_{k}u$
with $w_{k}\in N$ and observe that
\begin{align*}
\|T(u)\|_{2}^{2}=\|w_{k}T(u)\|_{2}^{2}&=\ip{w_{k}T(u),w_{k}T(u)}\\
&=\ip{w_{k}T(u),T(w_{k}u)}\\
&=\ip{w_{k}T(u),T(uv_{k})}\\
&=\ip{w_{k}T(u)v_{k}^{*},T(u)}.
\end{align*}
Because $v_{k}$ and $w_{k}$ tend to $0$ weakly, it is easy to see that for every $\xi\in L^{2}(N)\otimes L^{2}(N^{op})$ we have $\ip{w_{k}\xi v_{k}^{*},\xi}\to 0.$ Applying this observation to the above string of equalities shows that $T(u)=0.$ Since $T$ is $N$-$N$ bimodular we have that $T=0$ and this implies that $u\in \mathcal{H}_{s}(N\subseteq M)$ by Lemma \ref{L:disjointness}.

\end{proof}
For later use, let us note the following.

\begin{proposition}\label{P:Lebesgue} Let $(M,\tau)$ be a tracial von Neumann algebra and let $N$ be a von Neumann subalgebra of $M.$ Then
\[\mathcal{H}_{s}(N\subseteq M)\]
is a closed $N$-$N$ submodule of $L^{2}(M,\tau).$ Moreover, if we let
\[\mathcal{H}_{a}=L^{2}(M)\ominus (\mathcal{H}_{s}(N\subseteq M)),\]
then $\mathcal{H}_{a}$ embeds into an infinite direct sum of $L^{2}(N)\otimes L^{2}(N^{op})$ as an $N$-$N$ bimodule.

\end{proposition}

\begin{proof} For Hilbert $N$-$N$ bimodules $\mathcal{H},\mathcal{K},$ we let $\Hom_{N-N}(\mathcal{H},\mathcal{K})$ be all bounded, $N$-$N$ bimodular maps $\mathcal{H}\to\mathcal{K}.$ Then by Lemma \ref{L:disjointness}
\[\mathcal{H}_{s}(N\subseteq M)=\bigcap_{T\in \Hom_{N-N}(L^{2}(M),L^{2}(N)\otimes L^{2}(N^{op}))}\ker(T)\]
and is thus a closed $N$-$N$ submodule of $L^{2}(M).$

For the second part, let $(\xi_{j})_{j\in J}$ be a maximal family of vectors in $\mathcal{H}_{a}$ so that
\[L^{2}(N\xi_{j}N)\perp L^{2}(N\xi_{k}N)\mbox{ if $j,k\in J$ and $j\ne k$},\]
\[L^{2}(N\xi_{j}N)\leq L^{2}(N)\otimes L^{2}(N^{op})\mbox{ as $N$-$N$ bimodules, for all $j\in J.$}\]
If
\[\zeta\in \mathcal{H}_{a}\ominus \left(\sum_{j\in J}L^{2}(N\xi_{j}N)\right),\]
then we must have that
\[L^{2}(N\zeta N)\perp \sum_{j\in J}L^{2}(N\xi_{j}N).\]
We claim that $L^{2}(N\zeta N)$ is disjoint from $L^{2}(N)\otimes L^{2}(N^{op})$ as an $N$-$N$ bimodule. Indeed, if
\[T\in \Hom_{N-N}(L^{2}(N\zeta N),L^{2}(N)\otimes L^{2}(N^{op}))\]
 and $T=U|T|$ is the polar decomposition, then $L^{2}(N|T|(\zeta)N)$ embeds into $L^{2}(N)\otimes L^{2}(N^{op})$ as an $N$-$N$ bimodule via $U$. Moreover $L^{2}(N|T|(\zeta )N)\perp L^{2}(N\xi N)$ so by maximality $|T|(\zeta)=0$ and so $T(\zeta)=0.$ By $N$-$N$ bimodularity we know that $T=0.$ So $L^{2}(N\zeta N)$ is disjoint from $L^{2}(N)\otimes L^{2}(N^{op})$ as an $N$-$N$ bimodule and thus
\[\zeta\in \mathcal{H}_{a}\cap \mathcal{H}_{s}(N\subseteq M)=\{0\},\]
which proves the proposition.

\end{proof}

\begin{definition}\emph{Let $(M,\tau)$ be a tracial von Neumann algebra and let $N$ be a von Neumann subalgebra of $M.$ Define the} spectral normalizing algebra of $N$ inside $M$\emph{ to be $W^{*}(\mathcal{H}_{s}(N\subseteq M)).$ In general, for all ordinals $\alpha$ define the von Neumann algebras $N_{\alpha}$ by transfinite recursion as follows:}
\[N_{0}=N,\]
\[N_{\alpha}=W^{*}(\mathcal{H}_{s}(N_{\alpha-1}\subseteq M))\mbox{ \emph{if $\alpha$ is a sucessor ordinal}},\]
\[N_{\alpha}=\overline{\bigcup_{\alpha'<\alpha}N_{\alpha'}}^{SOT} \mbox{\emph{ if $\alpha$ is a limit ordinal}}.\]
\emph{We call $N_{\alpha}$ the step-$\alpha$-spectral normalizing algebra of $N$ inside $M.$ We say that $N\subseteq M$ is} spectrally regular of step $\alpha$ \emph{if $N_{\alpha}=M.$}
\end{definition}
Similarly, one can define what it means for a subalgebra $N\subseteq M$ to be regular, quasi-regular, one-sided quasi-regular etc of step $\alpha.$ Of all these notions, being spectrally regular of step $\alpha$ for some ordinal $\alpha$ is the weakest. We now prove Theorem \ref{T:mainintro} after stating it in this new language. If $A$ is a $*$-algebra, we define $\widetidle{x}$ for $x\in A\otimes_{\alg}A^{op}$ by saying that $(a\otimes b^{\op})^{\widetilde{}}=b^{*}\otimes (a^{*})^{op}$
and extending by linearity.  We leave it as an exercise to verify that if $I$ is any set, and $P\in \CC\ip{X_{i}\otimes X_{j}^{\op}:i,j\in I}$ then for any $R\in [0,\infty)^ {I}$ we have $\|P\|_{R,\infty}=\|\widetidle{P}\|_{R,\infty}.$
We will need the definition of  $1$-bounded entropy with respect to unbounded generators. We start with the definition of the microstates space with respect to unbounded operators.
\begin{definition}\label{D:measuremicrostates}\emph{Let $(M,\tau)$ be a tracial von Neumann algebra, and let $(x_{i})_{i\in I}$ a  finite collection of self-adjoint measurable operators affiliated to $(M,\tau).$ For natural numbers $m,k,$  positive real numbers $\gamma,R,\eta,$ and a finite $F\subseteq C_{b}(\RR,\RR)$ let}
\[\Gamma^{\eta}_{R}((x_{i})_{i\in I};F,m,\gamma,k),\]
\emph{be the set of all $X\in M_{k}(\CC)_{sa}^{I}$ so that}
\[(\phi(X_{i}))_{i\in I,\phi\in F}\in\Gamma((\phi(x_{i}))_{i\in I,\phi\in F};m,\gamma,k)\]
\emph{and}
\[\max_{i\in I}\tr(\chi_{[R,\infty)}(|X_{i}|))<\eta.\]
\emph{Given a finite subset $G$ of self-adjoint measurable operators affiliated to $(M,\tau)$ we let} $\Gamma^{\eta}_{R}((x_{i})_{i\in I}:G;F,m,\gamma,k)$
\emph{be the set of all $C\in M_{k}(\CC)_{sa}^{I}$ so that there is a $B\in M_{k}(\CC)_{sa}^{G}$ with}
\[(C,B)\in \Gamma^{\eta}_{R}((x_{i})_{i\in  I},G;F,m,\gamma,k)).\]
\emph{Given $a\in M$ with diffuse spectrum and $A\in M_{k}(\CC)$ we let} $\Xi^{\eta}_{A,R}((x_{i})_{i\in I}:G;F,m,\gamma,k)$
\emph{be the set of all $C\in M_{k}(\CC)^{I}$ so that} $(C,A)\in\Gamma^{\eta}_{R}((x_{i})_{i\in I},a:G;F,m,\gamma,k).$
\end{definition}
Recall that if $(M,\tau)$ is a tracial von Neumann algebra, then the \emph{measure topology} on $M$ is the unique vector space topology defined by saying that the sets
\[U(\varepsilon)=\{x\in M:\mbox{ there is a projection $p\in M$ with } \tau(p)\geq 1-\varepsilon,\|px\|_{\infty}<\varepsilon\}, \mbox{ $\varespilon\in (0,\infty)$,}\]
form a basis of neighborhoods of zero. Let $I$ be a finite set If $\Omega\subseteq M_{n}(\CC)^{I},$ we say that $S\subseteq \Omega$ is a $\varepsilon$-dense subset with respect to the measure topology if for all $A\in\Omega,$ there is a $B\in S$ and a projection $P\in M_{n}(\CC)^{I}$ so that for all $i\in I$
\[\|P_{i}(A_{i}-B_{i})\|_{\infty}<\varepsilon,\]
\[\tau(P_{i})\geq 1-\varepsilon.\]
We will write  $\Omega\subseteq_{\varepsilon,\textnormal{meas}}S$ to mean that $S$ is $\varespilon$-dense in $\Omega$ with respect to the measure topology. We let $K_{\varepsilon}(\Omega,\mbox{meas})$ be the smallest cardinality of an $\varepsilon$-dense subset of $\Omega$ with respect to the measure topology.

\begin{definition}\label{D:1bddentrunbddop1}\emph{Let $(M,\tau)$ be a tracial von Neumann algebra,  let $(x_{i})_{i\in I}$ be collections of self-adjoint measurable operators affiliated to $(M,\tau)$ and $G\subseteq M_{sa}.$  Fix $a\in M$ with diffuse spectrum and let $(A_{k})_{k=1}^{\infty}$ be a sequence of microstates for $a.$  For $m\in\NN,\gamma,R,\eta,\varepsilon>0$ and a finite $I_{0}\subseteq I$,$G_{0}\subseteq G$,$F\subseteq C_{b}(\RR,\RR)$ define}
\[K_{\varepsilon}(\Xi_{(A_{k})_{k=1}^{\infty},R}^{\eta}((x_{i})_{i\in I_{0}}:G_{0};F,m,\gamma),\mbox{\emph{meas}})=\limsup_{k\to\infty}\frac{1}{k^{2}}\log K_{\varepsilon}(\Xi_{(A_{k})_{k=1}^{\infty},R}^{\eta}((x_{i})_{i\in I}:G_{0};F,m,\gamma,k),\mbox{\emph{meas}}),\]
\[K_{\varepsilon}(\Xi_{(A_{k})_{k=1}^{\infty},R}^{\eta}((x_{i})_{i\in I}:G_{0}),\mbox{\emph{meas}})=\inf_{\substack{\textnormal{$F\subseteq C_{b}(\RR)$ finite},\\ m\in\NN,\\ \gamma>0}}K_{\varepsilon}(\Xi_{(A_{k})_{k=1}^{\infty},R}^{\eta}((x_{i})_{i\in I_{0}}:G_{0};F,m,\gamma),\mbox{\emph{meas}}),\]
\[K_{\varespilon}(\Xi_{(A_{k})_{k=1}^{\infty},R}^{\eta}((x_{i})_{i\in I}:G),\mbox{\emph{meas}})=\inf_{G_{0}\subseteq G\textnormal{ finite }}K_{\varepsilon}(\Xi_{(A_{k})_{k=1}^{\infty},R}^{\eta}((x_{i})_{i\in I}:G_{0}),\mbox{\emph{meas}})\]
\[K_{\varepsilon}(\Xi_{(A_{k})_{k=1}^{\infty}}^{\eta}((x_{i})_{i\in I_{0}}:G),\mbox{\emph{meas}})=\sup_{R>0}K_{\varepsilon}(\Xi_{(A_{k})_{k=1}^{\infty},R}^{\eta}((x_{i})_{i\in I}:G),\mbox{\emph{meas}}),\]
\[K_{\varepsilon}(\Xi_{(A_{k})_{k=1}^{\infty}}((x_{i})_{i\in I_{0}:}G),\mbox{\emph{meas}})=\inf_{\eta>0}K_{\varepsilon}(\Xi_{(A_{k})_{k=1}^{\infty}}^{\eta}((x_{i})_{i\in I_{0}}:G),\mbox{\emph{meas}}),\]
\[h(((x_{i})_{i\in I_{0}}:G),\mbox{\emph{meas}})=\sup_{\varepsilon>0}K_{\varepsilon}(\Xi_{(A_{k})_{k=1}^{\infty}}^{\eta}((x_{i})_{i\in I_{0}}:G),\mbox{\emph{meas}}),\]
\[h((x_{i})_{i\in I}:G,\mbox{\emph{meas}})=\sup_{\textnormal{finite } I_{0}\subseteq I}h((x_{i})_{i\in I_{0}}:G),\mbox{\emph{meas}}).\]

\end{definition}
As in the case of $h(N:M),$ we will abuse notation and use $h((x_{i})_{i\in I}:M,\mbox{meas})$
for $h((x_{i})_{i\in I}:M_{sa},\mbox{meas}).$
It is shown in Appendix \ref{S:unbounded} that $h((x_{i})_{i\in I}:M,\mbox{meas})=h(W^{*}((x_{i})_{i\in I}):M)$
and we will use this in the proof of Theorem \ref{T:mainintro}. We need the following elementary facts about almost containment in the measure topology, whose proofs are entirely direct and will be left to the reader.

\begin{proposition}\label{P:itseasymeas} Let $I$ be a finite set and $k\in \NN.$ Suppose that $\Xi_{j}\subseteq M_{k}(\CC)^{I},j=1,2,3$ and $\varepsilon,\delta>0.$We then have the following properties of almost containment in the measure topology:

(a)If $\Xi_{1}\subseteq_{\varespilon,\textnormal{meas}}\Xi_{2}\mbox{ and } \Xi_{2}\subseteq_{\delta,\textnormal{meas}}\Xi_{3},$
then $\Xi_{1}\subseteq_{\varepsilon+\delta,\textnormal{meas}}\Xi_{3}.$

(b): If  $\Xi_{1}\subseteq_{\varespilon,\textnormal{meas}}\Xi_{2},$ then we have \[K_{2(\varepsilon+\delta)}(\Xi_{1},\textnormal{meas})\leq K_{\delta}(\Xi_{2},\textnormal{meas}).\]

(c):  If $\Xi_{1}\subseteq_{\varepsilon,\|\cdot\|_{2}}\Xi_{2},$ then $\Xi_{1}\subseteq_{\sqrt{\varepsilon},\textnormal{meas}}\Xi_{2}.$

\end{proposition}

\begin{theorem}\label{T:main} Let $(M,\tau)$ be a tracial von Neumann algebra. If $N$ is a diffuse von Neumann subalgebra of $M,$ then $h(W^{*}(\mathcal{H}_{s}(N\subseteq M)):M)=h(N:M).$
\end{theorem}

\begin{proof} Fix an $a\in N_{sa}$ with diffuse spectrum and $(A_{k})_{k=1}^{\infty}$ a sequence of microstates for $a.$  Let $(x_{i})_{i\in I}$ be an enumeration of $\mathcal{H}_{s}(N\subseteq M)).$ Since the $x_{i}$'s are (a priori) unbounded operators we will use the formulation of $1$-bounded entropy for unbounded operators given in Definition \ref{D:1bddentrunbddop1}. Let $\varepsilon\in \left(0,\frac{1}{16}\right)$ and fix finite $I_{0}\subseteq I,\Omega\subseteq M_{sa}$ and an $\eta\in (0,\varepsilon).$ Let $R_{M}\in [0,\infty)^{M_{sa}}$ be a cutoff parameter and set $R_{N}=R_{M}\big|_{N_{sa}}.$

	Under the canonical homomorphism $\CC\ip{Y_{y}:y\in N_{sa}}\to N$
which sends $Y_{y}$ to $y,$ we may regard $L^{2}(M),L^{2}(N)\otimes L^{2}(N^{op})$ as $C_{R_{N}}\ip{Y_{y}\otimes Y_{x}^{\op}:x,y\in N_{sa}}$ modules. It is easy (but crucial!) to observe that for every $\xi\in\mathcal{H}_{s}(N\subseteq M)$ the $C_{R_{N}}\ip{Y_{y}\otimes Y_{x}^{\op}:x,y\in N_{sa}}$ modules $L^{2}(N\xi N)$ and $L^{2}(N)\otimes L^{2}(N^{op})$ are disjoint. Let $\kappa\in (0,\frac{\varespilon}{\sqrt{|I_{0}|}}),$ by Lemma \ref{L:disjointness} applied to the vectors $(x_{i})_{i\in I_{0}}\in \mathcal{H}_{s}(N\subseteq M)$ and $1\otimes 1\in L^{2}(N)\otimes L^{2}(N^{op})$ we may choose a $P\in \CC\ip{Y_{y}:y\in N_{sa}}\otimes_{\textnormal{alg}} \CC\ip{Y_{y}:y\in N_{sa}}^{\op}$ so that
\[\|P(y:y\in N_{sa})\# x_{i}-x_{i}\|_{2}<\kappa\mbox{ for all $i\in I_{0}$},\]
\[\|P(y:y\in N_{sa})\|_{2}=\|P(y:y\in N_{sa})(1\otimes 1)\|_{2}<\kappa,\]
\[\|P\|_{R_{N}}\leq 1.\]
Set
\[P_{\rea}=\frac{P+\widetilde{P}}{2},\]
\[P_{\ima}=\frac{P-\widetilde{P}}{2i},\]
then
\[\|P_{\rea}(y:y\in N_{sa})\# \Rea(x_{i})-P_{\ima}(y:y\in N_{sa})\#\Ima(x_{i})-\Rea(x_{i})\|_{2}<\kappa\mbox{ for all $i\in I_{0}$},\]
\[\|P_{\ima}(y:y\in N_{sa})\# \Rea(x_{i})+P_{\rea}(y:y\in N_{sa})\#\Ima(x_{i})-\Ima(x_{i})\|_{2}<\kappa\mbox{ for all $i\in I_{0}$},\]
\[\max(\|P_{\rea}\|_{R_{N}},\|P_{\ima}\|_{R_{N}})\leq 1,\]
\[\max(\|P_{\rea}\|_{2},\|P_{\ima}\|_{2})<\kappa.\]
Let $G\subseteq N_{sa}$ be a sufficiently large finite set so that
\[P_{\rea},P_{\ima}\in \CC\ip{Y_{y}:y\in G}\otimes_{\textnormal{alg}} \CC\ip{Y_{y}:y\in G}^{\op}.\]
We may choose a $R_{0}>0$ so that for all $R\geq R_{0}$ there is a $\phi\in C_{c}(\RR,\RR)$ with
\[\|\phi\|_{C_{b}(\RR)}\leq R,\]
\[\mbox{$\phi(t)=t$ for all $t$ with $|t|\leq R,$}\]
 and for all $i\in I_{0}$
\[\|P_{\rea}(y:y\in N_{sa})\# \phi(\Rea(x_{i}))-P_{\ima}(y:y\in N_{sa})\# \phi(\Ima(x_{i}))-\phi(\Rea(x_{i}))\|_{2}<2\kappa,\]
\[\|P_{\ima}(y:y\in N_{sa})\# \phi(\Rea(x_{i}))+P_{\rea}(y:y\in N_{sa})\#\phi(\Ima(x_{i}))-\phi(\Ima(x_{i}))\|_{2}<2\kappa.\]
We may assume that
\[R_{0}>\max\left(\max_{y\in G}R_{N,y},\sup_{k}\|A_{k}\|_{\infty}\right).\]
Now fix $R\geq R_{0}$ and let $\phi$ be as above.

	We now use that the above inequalities are approximately satisfied in our microstate space. Precisely, we may find a finite subset $F'$ of $C_{b}(\RR,\RR)$ with $\phi\in F',$ an $m'\in\NN$  and a $\gamma'>0$ so that for any
\[(T,B,C)\in \Gamma_{R}^{\eta}((\Rea(x_{i}))_{i\in I_{0}},(\Ima(x_{i}))_{i\in I_{0}},G;F',m',\gamma',k)\]
 with $\|C_{y}\|_{\infty}<R_{N,y}\mbox{ for all $y\in G$},$ we have
\[\|P_{\rea}(C_{y}:y\in G)\# \phi(T_{i})-P_{\ima}(C_{y}:y\in G)\#\phi(B_{i})-\phi(T_{i})\|_{2}<2\kappa\mbox{ for all $i\in I_{0}$,}\]
\[\|P_{\ima}(C_{y}:y\in G)\# \phi(T_{i})+P_{\rea}(C_{y}:y\in G)\#\phi(B_{i})-\phi(B_{i})\|_{2}<2\kappa\mbox{ for all $i\in I_{0},$}\]
\[\max(\|P_{\rea}(C_{y}:y\in G)\|_{2},\|P_{\ima}(C_{y}:y\in G)\|)<\kappa.\]
We may choose  finite $F\subseteq C_{b}(\RR,\RR),\Omega\subseteq  M_{sa},$ a natural number  $m\in\NN$ and a $\gamma>0$ so that if $(T,B)\in\Gamma_{R}^{\eta}((\Rea(x_{i}))_{i\in I_{0}},(\Ima(x_{i}))_{i\in I_{0}}:G,\Omega;F,m,\gamma,k),$ then there is a $C\in M_{k}(\CC)^{G}$ with
\[(T,B,C)\in \Gamma_{R}^{\eta}((\Rea(x_{i}))_{i\in I_{0}},(\Ima(x_{i}))_{i\in I_{0}},G:\Omega,F',m',\gamma',k),\]
\[\|C_{y}\|_{\infty}<R_{N,y}\mbox{ for all $y\in G$.}\]
For $(T,B)\in M_{k}(\CC)_{sa}^{I_{0}}\oplus M_{k}(\CC)_{sa}^{I_{0}},$ we use $\phi(T,B)=((\phi(T_{i}))_{i\in I_{0}},(\phi(B_{i}))_{i\in I_{0}}).$ Observe that we have the following approximate containment in the measure topology:
\begin{equation}\label{E:almostcontain12ladg}\Xi_{A_{k},R}^{\eta}((\Rea(x_{i}))_{i\in I_{0}},(\Ima(x_{i}))_{i\in I_{0}}:G,\Omega;F,m,\gamma,k))\subseteq_{\eta,\mbox{meas}}\phi(\Xi_{A_{k},R}^{\eta}((\Rea(x_{i}))_{i\in I_{0}},(\Ima(x_{i}))_{i\in I_{0}}:G,\Omega;F,m,\gamma,k)).
\end{equation}
From the above almost containment, we can estimate the size of an almost dense (in the measure topology) subset of
\[\Xi_{A_{k},R}^{\eta}((\Rea(x_{i}))_{i\in I_{0}},(\Ima(x_{i}))_{i\in I_{0}}:G,\Omega;F,m,\gamma,k),\mbox{meas}),\]
by finding a set $\Sigma\subseteq M_{k}(\CC)^{\oplus I}\oplus M_{k}(\CC)^{\oplus}$ where we have a good control on the size of an almost dense (in $\|\cdot\|_{2}$) subset of $\Sigma$ and so that $\Sigma$ almost contains (in the measure topology) the right hand side of (\ref{E:almostcontain12ladg}).

	Let $1<D$ be such that for any $k\in \NN$ and any $L,L'\in M_{k}(\CC)^{G}$ we have
\[\|P_{\rea}(L_{y}:y\in G)-P_{\rea}(L'_{y}:y\in G)\|_{\infty}\leq D\|L-L'\|_{\infty},\]
\[\|P_{\ima}(L_{y}:y\in G)-P_{\ima}(L'_{y}:y\in G)\|_{\infty}\leq D\|L-L'\|_{\infty}.\]
Fix an $\varepsilon'\in \left(0,\frac{\varepsilon}{DR\sqrt{|I_{0}|}}\right)$ and choose a $S\subseteq \Xi_{A_{k},R_{N}\vee R_{M}}(G:\Omega;m',\gamma',k)$ which is $\varepsilon'
$-dense with respect to $\|\cdot\|_{\infty}$ and which has
\[|S|=K_{\varepsilon'}(\Xi_{A_{k},R_{N}\vee R_{M}}(G:\Omega;m,\gamma,k),\|\cdot\|_{\infty}).\]
Let $H_{I_{0}}=M_{k}(\CC)^{I_{0}}\oplus M_{k}(\CC)^{I_{0}}$ with the Hilbert space norm $\|\cdot\|_{2},$ i.e.
\[\|(L,E)\|_{2}^{2}=\sum_{i\in I_{0}}\|L_{i}\|_{2}^{2}+\sum_{i\in I_{0}}\|E_{i}\|_{2}^{2}.\]
Suppose we are given $(T,B)\in \Xi_{A_{k},R}^{\eta}((\Rea(x_{i}))_{i\in I_{0}},(\Ima(x_{i}))_{i\in I_{0}}:G, \Omega;F,m,\gamma,k)).$ Choose a $C\in M_{k}(\CC)^{G}$ with
\[(T,B,C)\in \Xi_{A_{k},R}^{\eta}((\Rea(x_{i}))_{i\in I_{0}},(\Ima(x_{i}))_{i\in I_{0}},G: \Omega;F',m',\gamma',k),\]
\[\|C_{y}\|_{\infty}<R_{N,y}\mbox{ for all $y\in G$}.\]
Let $C'\in S$ be such that $\|C-C'\|_{\infty}\leq \varepsilon'.$
We then have for all $i\in I_{0}:$
\[\|P_{\rea}(C'_{y}:y\in G)\# \phi(T_{i})-P_{\ima}(C_{y}':y\in G)\#\phi(B_{i})-\phi(T_{i})\|_{2}<2\kappa+2\varepsilon'DR<\frac{6\varepsilon}{\sqrt{|I_{0}|}},\]
\[\|P_{\ima}(C'_{y}:y\in G)\#\phi(T_{i})+P_{\rea}(C_{y}':y\in G)\#\phi(B_{i})-\phi(B_{i})\|_{2}<2\kappa+2\varepsilon'DR<\frac{6\varepsilon}{\sqrt{|I_{0}|}}.\]
For each $C'\in S,$ define linear transformations $\Psi_{j,C'}\colon M_{k}(\CC)^{I_{0}}\oplus M_{k}(\CC)^{I_{0}}\to M_{k}(\CC)^{I_{0}}\oplus M_{k}(\CC)^{I_{0}},j=1,2$ by
\[\Psi_{1,C'}(L,E)=((P_{\rea}(C'_{y}:y\in G)\# L_{i})_{i\in I_{0}},(P_{\rea}(C'_{y}:y\in G)\# E_{i})_{i\in I_{0}}),\]
\[\Psi_{2,C'}(L,E)=((-P_{\ima}(C'_{y}:y\in G)\# E_{i})_{i\in I_{0}},(P_{\ima}(C'_{y}:y\in G)\# L_{i})_{i\in I_{0}}),\]
and set $\Psi_{C'}=\Psi_{1,C'}+\Psi_{2,C'}.$
We have then shown that
\[\phi(\Xi_{A_{k},R}^{\eta}((\Rea(x_{i}))_{i\in I_{0}},(\Ima(x_{i}))_{i\in I_{0}}:G,\Omega;F,m,\gamma,k)))\subseteq_{6\varepsilon,\|\cdot\|_{2}}\bigcup_{C'\in S}\Psi_{C'}(\{(L,E):\|(L,E)\|_{2}\leq \sqrt{2}R|I_{0}|^{1/2}\})
\]
and part (c) of Proposition \ref{P:itseasymeas} shows that
\begin{equation}\label{E:almostcontainldkaghal}
\phi(\Xi_{A_{k},R}^{\eta}((\Rea(x_{i}))_{i\in I_{0}},(\Ima(x_{i}))_{i\in I_{0}}:G,\Omega;F,m,\gamma,k)))\subseteq_{\sqrt{6\varepsilon},\textnormal{meas}}\bigcup_{C'\in S}\Psi_{C'}(\{(L,E):\|(L,E)\|_{2}\leq \sqrt{2}R|I_{0}|^{1/2}\}).
\end{equation}
We will use (\ref{E:almostcontain12ladg}),(\ref{E:almostcontainldkaghal}) to estimate the size of an almost dense subset of
\[\Xi_{A_{k},R}^{\eta}((\Rea(x_{i}))_{i\in I_{0}},(\Ima(x_{i}))_{i\in I_{0}}:G,\Omega;F,m,\gamma,k),\]
by estimating $K_{\varespilon}(\Psi_{C'}(\{(L,E):\|(L,E)\|_{2}\leq R\}),\mbox{meas})$ for each fixed $C'\in S.$

	Let
\[\theta\colon M_{k}(\CC)\otimes M_{k}(\CC)^{op}\to B((M_{k}(\CC),\|\cdot\|_{2}))\]
be defined by $\theta(A)(B)=A\# B.$  Since $M_{k}(\CC)\otimes M_{k}(\CC)^{op}$ is simple and has a unique trace, we know that $\theta$ is a trace-preserving isomorphism. Thus for all $C'\in S$
\[\|\theta(P_{\rea}(C_{y}':y\in G))\|_{2}=\|P_{\rea}(C_{y}':y\in G)\|_{2}<\kappa,\]
\[\|\theta(P_{\ima}(C_{y}':y\in G))\|_{2}=\|P_{\ima}(C_{y}':y\in G)\|_{2}<\kappa.\]
It follows that with respect to our given Hilbert space structure on $M_{k}(\CC)^{I_{0}}\oplus M_{k}(\CC)^{I_{0}}$ we have
\[\|\Psi_{C'}\|_{L^{2}(\tr_{H_{I_{0}}})}\leq \|\Psi_{1,C'}\|_{L^{2}(\tr_{H_{I_{0}}})}+\|\Psi_{2,C'}\|_{L^{2}(\tr_{H_{I_{0}}})}<2\kappa.\]
Fix a $C'\in S$ and choose a
\[\Delta_{C'}\subseteq \{\Psi_{C'}(L,E):L\in M_{k}(\CC)_{sa}^{I_{0}},E\in M_{k}(\CC)_{sa}^{I_{0}},\|L\|_{\infty},\|E\|_{\infty}\leq R\}\]
which is $\varespilon$-dense with respect to $\|\cdot\|_{2}.$
By Lemma \ref{L:volumepackingestimate} we may choose such a $\Delta_{C'}$ which has
\begin{align*}
|\Delta_{C'}|&\leq K_{\frac{\varepsilon}{2}}(\{\Psi_{C'}(L,E):\|(L,E)\|_{2}\leq \sqrt{2}R|I_{0}|^{1/2}\},\|\cdot\|_{2})\\
&\leq \left(\frac{3\sqrt{2}R|I_{0}|^{1/2}+\frac{\varepsilon}{2}}{\frac{\varepsilon}{2}}\right)^{512|I_{0}|^{2}R^{2}k^{2}\frac{\kappa^{2}}{\varepsilon^{2}}}.
\end{align*}
Part (c) of Proposition \ref{P:itseasymeas} shows that
\begin{equation}\label{E:almostcontainment3ladghlah}
\bigcup_{C'\in S}\Psi_{C'}(\{(L,E):\|(L,E)\|_{2}\leq \sqrt{2}R|I_{0}|^{1/2}\})\subseteq_{\sqrt{\varespilon},\textnormal{meas}}\bigcup_{C'\in S}\Delta_{C'}.
\end{equation}
Combining (\ref{E:almostcontain12ladg}),(\ref{E:almostcontainldkaghal}),(\ref{E:almostcontainment3ladghlah}) and using (a) of Proposition \ref{P:itseasymeas} shows that
\[\Xi_{A_{k},R}^{\eta}((\Rea(x_{i}))_{i\in I_{0}},(\Ima(x_{i}))_{i\in I_{0}}:G,\Omega;F,m,\gamma,k))\subseteq_{(2+\sqrt{6})\sqrt{\varespilon},\mbox{meas}}\bigcup_{C'\in S}\Delta_{C'}.\]
So (b) of Proposition \ref{P:itseasymeas} shows that
\begin{align*}
K_{2(2+\sqrt{6})\sqrt{\varepsilon}}(\Xi_{(A_{k}),R}^{\eta}((\Rea (x_{i}))_{i\in I_{0}},(\Ima(x_{i}))_{i\in I_{0}}:G,\Omega;F,m,\gamma,k),\mbox{meas})&\leq \sum_{C'\in S}|\Delta_{C'}|\\
&\leq \left(\frac{3\sqrt{2}R|I_{0}|^{1/2}+\frac{\varepsilon}{2}}{\frac{\varepsilon}{2}}\right)^{512|I_{0}|^{2}R^{2}k^{2}\frac{\kappa^{2}}{\varepsilon^{2}}}\\
&\times K_{\varepsilon'}(\Xi_{A_{k}, R_{N}\vee R_{M}}(G:\Omega;m',\gamma',k),\|\cdot\|_{\infty}).
\end{align*}
Taking $\frac{1}{k^{2}}\log$ of both sides and letting $k\to \infty$ we have
\begin{align*}
K_{2(2+\sqrt{6})\sqrt{\varepsilon}}(\Xi_{(A_{k}),R}^{\eta}((\Rea (x_{i}))_{i\in I_{0}},(\Ima(x_{i}))_{i\in I_{0}}:G,\Omega;F,m,\gamma),\mbox{meas})\leq &\frac{512|I_{0}|^{2}R^{2}\kappa^{2}}{\varepsilon^{2}}\log\left(\frac{3\sqrt{2}R|I_{0}|^{1/2}+\frac{\varespilon}{2}}{\frac{\varepsilon}{2}}\right)\\
&+K_{\varepsilon'}(\Xi_{(A_{k}),R_{N}\vee R_{M}}(G:\Omega;m',\gamma'),\|\cdot\|_{\infty}).
\end{align*}
A fortiori,
\begin{align*}
K_{2(2+\sqrt{6})\sqrt{\varepsilon}}(\Xi_{(A_{k}),R}^{\eta}((\Rea (x_{i}))_{i\in I_{0}},(\Ima (x_{i}))_{i\in I_{0}}:M_{sa}),\mbox{meas})\leq &\frac{512|I_{0}|^{2}R^{2}\kappa^{2}}{\varepsilon^{2}}\log\left(\frac{3\sqrt{2}R|I_{0}|^{1/2}+\frac{\varespilon}{2}}{\frac{\varepsilon}{2}}\right)\\
&+K_{\varepsilon'}(\Xi_{(A_{k}),R_{N}\vee R_{M}}(G:\Omega;m',\gamma'),\|\cdot\|_{\infty})
\end{align*}
and since $I_{0},\kappa,R$ do not depend upon $\Omega,\gamma',m',$ we can take the infimum over all $\Omega,\gamma',m'$ to see that
\begin{align*}
K_{2(2+\sqrt{6})\sqrt{\varepsilon}}(\Xi_{(A_{k}),R}^{\eta}((\Rea (x_{i}))_{i\in I_{0}},(\Ima (x_{i}))_{i\in I_{0}}:M_{sa}),\mbox{meas})\leq &\frac{512|I_{0}|^{2}R^{2}\kappa^{2}}{\varepsilon^{2}}\log\left(\frac{3R\sqrt{2}|I_{0}|^{1/2}+\frac{\varespilon}{2}}{\frac{\varepsilon}{2}}\right)\\
&+K_{\varepsilon'}(\Xi_{(A_{k}),R_{N}\vee R_{M}}(G:M_{sa}),\|\cdot\|_{\infty}).
\end{align*}
Since the above inequality holds for all sufficiently small $\varespilon'>0$ we can let $\varepsilon'\to 0$ to see that:
\begin{align*}
K_{2(2+\sqrt{6})\sqrt{\varepsilon}}(\Xi_{(A_{k}),R}^{\eta}((\Rea (x_{i}))_{i\in I_{0}},(\Ima (x_{i}))_{i\in I_{0}}:M_{sa}),\mbox{meas})\leq &\frac{512|I_{0}|^{2}R^{2}\kappa^{2}}{\varepsilon^{2}}\log\left(\frac{3R\sqrt{2}|I_{0}|^{1/2}+\frac{\varespilon}{2}}{\frac{\varepsilon}{2}}\right)\\
&+h(G:M).
\end{align*}
 A fortiori,
\begin{align*}K_{2(2+\sqrt{6})\sqrt{\varepsilon}}(\Xi_{(A_{k}),R}^{\eta}((\Rea (x_{i}))_{i\in I_{0}},(\Ima (x_{i}))_{i\in I_{0}}:M_{sa}),\mbox{meas})\leq &\frac{512|I_{0}|^{2}R^{2}\kappa^{2}}{\varepsilon^{2}}\log\left(\frac{3R\sqrt{2}|I_{0}|^{1/2}+\frac{\varespilon}{2}}{\frac{\varepsilon}{2}}\right)\\
&+h(N:M).
\end{align*}
Since the second term of the right-hand side of this inequality is now independent of $\kappa,$ we can let $\kappa\to 0$ to show that
\[K_{2(2+\sqrt{6})\sqrt{\varepsilon}}(\Xi_{(A_{k}),R}^{\eta}((\Rea (x_{i}))_{i\in I_{0}},(\Ima (x_{i}))_{i\in I_{0}}:M_{sa}),\mbox{meas})\leq h(N:M).\]
Taking the supremum over all $R>0,$ and then taking the infimum over all $\eta>0$ we have
\[K_{2(2+\sqrt{6})\sqrt{\varepsilon}}(\Xi_{(A_{k})}((\Rea (x_{i}))_{i\in I_{0}},(\Ima (x_{i}))_{i\in I_{0}}:M_{sa}),\mbox{meas})\leq h(N:M;\|\cdot\|_{\infty})\]
and we can now take the supremum over $\varepsilon>0$ and $I_{0}$ to complete the proof.
\end{proof}

\section{Applications of The Main Result}

\begin{cor}\label{C:main} Let $(M,\tau)$ be a tracial von Neumann algebra and $N$ a diffuse von Neumann subalgebra of $M.$ For any ordinal $\alpha$ let $N_{\alpha}$ be the step-$\alpha$ spectral normalizing algebra of $N$ inside $M.$ We then have that
\[h(N_{\alpha}:M)=h(N:M).\]
In particular if $N$ is spectrally regular in $M$ of step $\alpha,$ then
\[h(M)\leq h(N).\]
Thus if $M=W^{*}(F)$ for some finite $F\subseteq M_{sa}$ and $\delta_{0}(F)>1$, then no diffuse subalgebra $N\subseteq M$ with $h(N)<\infty$ (e.g. if $N$ is hyperfinite) is spectrally regular in $M$ of step $\alpha.$
\end{cor}

\begin{proof}  We prove that $h(N_{\alpha}:M)=h(N:M)$
 by transfinite induction, the case $\alpha=0$ being tautological. The case of a successor ordinal follows from Theorem \ref{T:main} and the case of a limit ordinal is handled by Lemma \ref{L:increasingunions}.
The ``in particular'' part follows since if $N_{\alpha}=M,$ then
\begin{align*}
h(M)&=h(M:M)\\
&=h(N_{\alpha}:M)\\
&=h(N:M)\\
&\leq h(N:N)\\
&=h(N)\\
&<\infty.
\end{align*}
Whereas if $F\subseteq M_{sa}$ is such that $\delta_{0}(F)>1$ and $M=W^{*}(F),$ then $h(M)=\infty.$

\end{proof}

\begin{cor}\label{C:weakercorollaries} Let $(M,\tau)$ be a tracial von Neumann algebra and $N\subseteq M$  a diffuse von Neumann subalgebra.

(i): Let $N_{\alpha}$  be the step-$\alpha$ one-sided quasi-normalizing algebra of $N$ inside $M,$ then \[h(N_{\alpha}:M)=h(N:M).\]
In particular, if $M=W^{*}(F)$ for a finite $F\subseteq M_{sa}$ with $\delta_{0}(F)>1$ and $N$ has $h(N)<\infty$ (e.g. if $N$ is hyperfinite), then $N_{\alpha}\ne M$ for any ordinal $\alpha.$

(ii): Let $N_{\alpha}$ be the step-$\alpha$ wq-normalizing algebra of $N$ inside $M.$ We then have that \[h(N_{\alpha}:M)=h(N:M).\]
In particular, if $M=W^{*}(F)$ for a finite $F\subseteq M_{sa}$ with $\delta_{0}(F)>1$ and $N$ has $h(N)<\infty$ , then $N_{\alpha}\ne M$ for any ordinal $\alpha.$

\end{cor}

\begin{proof} Both statements are automatic from the inclusions
\[q^{1}\mathcal{N}_{M}(N)\subseteq \mathcal{H}_{s}(N\subseteq M),\]
\[\mathcal{N}_{M}^{wq}(N)\subseteq \mathcal{H}_{s}(N\subseteq M).\]

\end{proof}

 We thus automatically affirmatively answer the conjecture of Galatan-Popa.
\begin{cor}\label{C:wqreg} Let $(M,\tau)$ be a tracial von Neumann algebra. Let $N$ be a diffuse von Neumann subalgebra of $M$ and let $Q$ be the wq-normalizing algebra of $N\subseteq M.$ If $h(N:M)=0,$ then $h(Q:M)=0.$ Thus if $M=W^{*}(x_{1},\dots,x_{n})$ where $x_{j}\in M_{sa}$ and $\delta_{0}(x_{1},\dots,x_{n})>1$ (e.g. $M=L(\FF_{t})$ for some $t>1$), then $M$ has no diffuse, hyperfinite, wq-regular von Neumann subalgebra.
\end{cor}
\begin{proof} This is just a rephrasing of part $(ii)$ of the preceding corollary.
\end{proof}

We want to prove that if $h(M)>0$ (e.g. $M=L(\FF_{n})$) then no algebra with $h(N)=0$ can be spectrally regular inside $M,$ even if it exists in only an approximate sense.

\begin{definition}\label{D:approxspecreg}\emph{Let $(M,\tau)$ be a tracial von Neumann algebra. For a tracial von Neumann algebra $(N,\tau_{N})$ and an ordinal $\alpha,$ we say that} $N$ is approximately $\alpha$-spectrally regular in \emph{$M$ if there is a set $I$, a free ultrafilter $\omega$ on $I$ and a trace-preserving embedding $N\subseteq M^{\omega}$ so that  the step $\alpha$-spectral normalizer of $N$ inside $M^{\omega}$ contains $M.$}
\end{definition}

To ease our goal, we will prove the following proposition.

\begin{proposition}\label{P:approximate} Let $(M,\tau)$ be a tracial von Neumann algebra.  Fix a free ultrafilter $\omega$ on  a set $I.$  Then for any diffuse von Neumann subalgebra $Q$ of $M$ we have
\[h(Q:M)=h(Q:M^{\mathcal{\omega}}).\]
\end{proposition}
\begin{proof} By Lemma \ref{L:increasingunions}, it suffices to show that for any finite $F\subseteq Q_{sa}$ we have
\[h(F:M)=h(F:M^{\mathcal{\omega}}).\]
 We automatically have
\[h(F:M^{\omega})\leq h(F:M),\]
so it suffices to show that
\[h(F:M)\leq h(F:M^{\mathcal{\omega}}).\]
Fix $a\in W^{*}(F)$ with diffuse spectrum and let $A_{k}$ be a sequence of microstates for $a.$ Let $R_{M^{\mathcal{\omega}}}\in [0,\infty)^{M^{\mathcal{\omega}}},R_{F}\in [0,\infty)^{F}$ be cutoff parameters. Let $G\subseteq M_{sa}^{\mathcal{\omega}}$ be a finite set and write $G=\{y_{1},\dots,y_{r}\}.$ Choose $y_{j,i},1\leq j\leq r,i\in I$ so that
\[y_{j}=(y_{j,i})_{i\to\mathcal{\omega}},1\leq j\leq r,\]
\[\sup_{i}\|y_{j,i}\|_{\infty}< R_{M^{\mathcal{\omega}},y_{j}}.\]
Set $G_{i}=\{y_{1,i},\dots,y_{j,i}\}.$
Given $m\in\NN,\gamma>0$ it is easy to see that the set of $i$ for which
\[\Xi_{A_{k},R_{F}\vee R_{M^{\omega}}}(F:G_{i};m,\gamma/2,k)\subseteq \Xi_{A_{k},R_{F}\vee R_{M^{\omega}}}(F:G;m,\gamma,k)\]
is in $\omega.$ For such $i,$
\[K_{2\varespilon}(\Xi_{A_{k},R_{F}\vee R_{M^{\omega}}}(F:G_{i};m,\gamma/2,k),\|\cdot\|_{2})\leq K_{\varpesilon}(\Xi_{A_{k},R_{F}\vee R_{M^{\omega}}}(F:G;m,\gamma,k),\|\cdot\|_{2}).\]
Taking $\frac{1}{k^{2}}\log$ of both sides and letting $k\to\infty$ we have
\begin{align*}
K_{2\varespilon}(\Xi_{(A_{k}),R_{F}\vee R_{M^{\omega}}}(F:G_{i}),\|\cdot\|_{2})&\leq K_{2\varespilon}(\Xi_{(A_{k}),R_{F}\vee R_{M^{\omega}}}(F:G_{i};m,\gamma/2),\|\cdot\|_{2})\\
&\leq  K_{\varpesilon}(\Xi_{(A_{k}),R_{F}\vee R_{M^{\omega}}}(F:G;m,\gamma),\|\cdot\|_{2}).
\end{align*}
A fortiori,
\[K_{2\varepsilon}(\Xi_{(A_{k}),R_{F}\vee R_{M^{\omega}}}(F:M),\|\cdot\|_{2})\leq K_{\varpesilon}(\Xi_{(A_{k}),R_{F}\vee R_{M^{\omega}}}(F:G;m,\gamma),\|\cdot\|_{2}).\]
Taking the infimum over all $G,m,\gamma$ we see that
\[K_{2\varepsilon}(\Xi_{(A_{k}),R_{F}\vee R_{M^{\omega}}}(F:M),\|\cdot\|_{2})\leq  K_{\varpesilon}(\Xi_{(A_{k}),R_{F}\vee R_{M^{\omega}}}(F:M^{\omega}),\|\cdot\|_{2})\]
and taking the supremum over all $\varespilon>0$ completes the proof.

\end{proof}

As an application we present a corollary which implies Theorem \ref{T:mainintro}.

\begin{cor}\label{C:main2} Suppose that $(M,\tau)$ is a tracial von Neumann algebra and that $N$ is a diffuse von Neumann subalgebra of $M.$
 If, for some ordinal $\alpha,$ we have that $N$ is approximately $\alpha$-spectrally regular in $M$ in the sense of Definition \ref{D:approxspecreg}, then
\[h(M)\leq h(N).\]
So if $h(M)>0,$ then there are no approximately $\alpha$-spectrally regular hyperfinite subalgebras in $M.$

\end{cor}\begin{proof}  Suppose that $N$ is approximately $\alpha$-spectrally regular inside $M.$ Let $I$ be a set and $\omega$ a free ultrafilter on $I$ and, for any ordinal $\alpha,$ let $N_{\alpha}$ be the step-$\alpha$ spectral normalizing algebra of $N$ inside $M^{\omega}.$  If $N_{\alpha}\supseteq M,$ then by Proposition \ref{P:approximate} and Theorem \ref{T:main},
\begin{align*}
h(M)=h(M:M)&=h(M:M^{\omega})\\
&\leq h(N_{\alpha}:M^{\omega})\\
&= h(N:M^{\omega})\\
&\leq h(N:N)\\
&=h(N).
\end{align*}

\end{proof}

Galatan-Popa conjecture in fact that $L(\FF_{t})$ cannot have a von Neumann subalgebra which is wq-regular satisfying a weaker property than hyperfiniteness  called Property (C'), we will show that this conjecture is true as well.

\begin{definition}\emph{Let $(M,\tau)$ be a tracial von Neumann algebra. A finite sequence of unitaries $v_{1},\dots,v_{k}$ each with diffuse spectrum is said to have} Property (C') in $M$ \emph{if for some (equivalently any) free ultrafilter $\omega\in\beta\NN\setminus\NN,$ there are mutually commuting unitaries $u_{1},\dots,u_{k}\in M^{\omega}$  with diffuse spectrum so that $[u_{j},v_{j}]=0$ for $1\leq j\leq k.$ Let $N$ be a von Neumann subalgebra of $M.$ We say that $N\subseteq M$ has} Property $(C')$ \emph{if there is a $V\subseteq \mathcal{U}(M)$ with $W^{*}(V)=M$ and so that every finite $F\subseteq V$ has Property $(C')$ in $M.$}
\end{definition}

In \cite{GalatanPopa} Remark 3.9 it is shown that if a sufficiently good cohomology theory is developed, then $L(\FF_{t})$ for $t>1$ does not have a wq-regular von Neumann subalgebra with Property (C'). We now show this using $1$-bounded entropy. This may be regarded as another consistency check for the postulate that there is a good cohomology theory for $\textrm{II}_{1}$-factors.

\begin{cor}\label{C:PropertyC'} Let $(M,\tau)$ be a tracial von Neumann algebra and suppose that $N$ is a diffuse von Neumann subalgebra of $M.$ If $N\subseteq M$ has Property (C'), then $h(N:M)\leq 0.$\end{cor}
\begin{proof} Our hypothesis implies that there exists $V\subseteq \mathcal{U}(M)$  with $N\subseteq W^{*}(V)$ and so that $F\subseteq V$  has Property (C') in $M$ for all $F\subseteq V$ finite.
It suffices to show that $h(W^{*}(V):M)\leq 0.$
Fix $v_{1},\dots,v_{n}\in V$ and set
\[Q=W^{*}(v_{1},\dots,v_{n}).\]
By Lemma \ref{L:increasingunions} it suffices to show  that $h(Q:M)\leq 0.$
Fix a free ultrafilter $\omega\in\beta\NN\setminus\NN,$ by Proposition \ref{P:approximate} it suffices to show that
$h(Q:M^{\omega})\leq 0.$

	 Since $\{v_{1},\dots,v_{n}\}$ has Property (C') in $M$, we may choose unitaries $u_{j}\in M^{\omega}$  for $1\leq j\leq n,$ so that
\begin{itemize}
\item $u_{j}$ has diffuse spectrum for $1\leq j\leq n,$
\item    $[u_{j},v_{j}]=0$ for $1\leq j\leq n,$
\item  $[u_{l},u_{j}]=0$ for all $1\leq l,j\leq n.$
\end{itemize}
Set $A=W^{*}(u_{1},\dots,u_{n})$
and for $1\leq k\leq n,$ let $P_{k}=W^{*}(u_{k},v_{k}).$
As $A$ and $P_{k}$ are abelian we have that $h(A)=h(P_{k})=0$ for $1\leq k\leq n.$ Moreover $A\cap P_{k}\supseteq W^{*}(u_{k})$ which is diffuse. By repeated applications of Lemma \ref{L:diffuseintersection} we see that
$h\left(A\vee\bigvee_{k=1}^{n}P_{k}\right)\leq 0.$
Setting $P=A\vee \bigvee_{k=1}^{n}P_{k}$
we now see  that
\[h(Q:M^{\omega})\leq h(P:M^{\omega})\leq h(P:P)=h(P)\leq 0\]
and this completes the proof.

\end{proof}

\begin{cor}\label{C:GalaPopa} Let $M=L(\FF_{n})$ for some $n\in\NN,n\geq 2$ (or more generally $M$ could be an interpolated free group factor) and let $N\subseteq M$ be a diffuse von Neumann subalgebra. If $N\subseteq M$ has property (C'), then $N$ is not wq-regular in $M.$

\end{cor}

\begin{proof} Corollary \ref{C:PropertyC'} implies that $h(N:M)\leq 0,$ so the desired claim now follows from Corollary \ref{C:weakercorollaries}.

\end{proof}

Corollary \ref{C:GalaPopa} solves another question of Galatan-Popa in \cite{GalatanPopa} (see Remark 3.9). For the next application, we are concerned with the following question of Peterson.

\begin{question}\label{C:JP}\emph{If $t\in(1,\infty]$ and $N$ is a finitely-generated, nonamenable von Neumann subalgebra of $L(\FF_{t}),$ does there exist a finite $F\subseteq N_{sa}$ so that $N=W^{*}(F)$ and $\delta_{0}(F)>1?$}\end{question}

Because of our results, we make the following conjecture.

\begin{conjecture}\label{C:Me}\emph{If $t\in (1,\infty)$ and $N\subseteq L(\FF_{t})$ is a maximal amenable von Neumann subalgebra, then as $N$-$N$ bimodules:}
\[L^{2}(L(\FF_{t}))\ominus L^{2}(N)\leq [L^{2}(N)\otimes L^{2}(N)]^{\oplus \infty}.\]
\end{conjecture}

It is not hard to use Theorem \ref{T:main} to relate these two questions.

\begin{cor} If Question \ref{C:JP} has an affirmative answer, then Conjecture \ref{C:Me} is true.
\end{cor}

\begin{proof} By Proposition \ref{P:Lebesgue} we may write $L^{2}(L(\FF_{t}))\ominus L^{2}(N)=\mathcal{H}_{a}\oplus \mathcal{H}_{s},$
where $\mathcal{H}_{a},\mathcal{H}_{s}$ are closed $N$-$N$ submodules of $L^{2}(L(\FF_{t}))\ominus L^{2}(N),$
with
\[\mathcal{H}_{a}\leq [L^{2}(N)\otimes L^{2}(N^{op})]^{\oplus \infty},\]
\[\mathcal{H}_{s}\mbox{ is disjoint from } L^{2}(N)\otimes L^{2}(N^{op}),\]
as $N$-$N$ bimodules. Suppose $\mathcal{H}_{s}\ne \{0\},$ take $\xi\in \mathcal{H}_{s}$ with $\xi\ne 0$ and set $Q=W^{*}(N,\xi).$
Since we are assuming that Peterson's question has a positive answer and $N$ is maximal amenable we have that $h(Q)=\infty.$ However, by Theorem \ref{T:main} we have
\[h(Q)=h(Q:Q)=h(N:Q)\leq h(N:N)=h(N)=0\]
and this is a contradiction.

\end{proof}

We now use Theorem \ref{T:main} to give new examples of tracial von Neumann algebras with  $1$-bounded entropy zero. In some cases these algebras are finitely generated, so we have new examples of strongly $1$-bounded algebras as defined by Jung.  We use the free Gaussian functor of Voiculescu (see \cite{VoiculescuGaussian},\cite{VoiculescuDykemaNica}), as well as the $q$-Gaussian functor of Bozjeko-Speicher (see \cite{BozjekoSpeicher}). We refer the reader to these references for the precise definition. For our purposes we note that for $q\in [-1,1]$ and a real Hilbert space $\mathcal{H}$ the $q$-Gaussian functor assigns a tracial von Neumann algebra $\Gamma_{q}(\mathcal{H})$ in a functorial way. Additionally, to every $\xi\in \mathcal{H}$ we have in a natural way a self-adjoint element $s_{q}(\xi)\in \Gamma_{q}(\mathcal{H}).$
Moreover, if  $G$ is a group and $\pi\colon G\to \mathcal{O}(\mathcal{H})$
is an orthogonal representation on a real Hilbert space we have an induced action $\alpha_{\pi},$ called the free Bogoliubov action, on $\Gamma_{q}(\mathcal{H})$ by $\alpha_{\pi,g}(s_{q}(\xi))=s_{q}(\pi(g)\xi).$
If $G$ is a discrete group we let $\lambda_{G}\colon G\to \mathcal{U}(\ell^{2}(G))$ be the left regular representation defined by
\[(\lambda_{G}(g)\xi)(x)=\xi(g^{-1}x).\]
We will denote by $\lambda_{G,\RR}$ the orthogonal representation of $G$ on $\ell^{2}(G,\RR)$ obtained by restriction.

Houdayer-Shlyakhtenko showed in \cite{HoudayerShlyakhtenko} that in many situations the crossed products $\Gamma_{q}(\mathcal{H})\rtimes G$
(when $G=\ZZ,q=0$) share many properties with interpolated free group factors (i.e. CMAP and strong solidity). However,  under certain assumptions on the spectral measure of $\pi$ they can show that such algebras are \emph{not} isomorphic to interpolated free group factors and our techniques will allow us to extend their results. The reader should contrast these results with the fact that
\[\Gamma_{0}(\ell^{2}(\ZZ,\RR))\rtimes_{\alpha_{\lambda_{\ZZ,\RR}^{\oplus n}}}\ZZ\cong L(\FF_{n+1}).\]
 If $\pi\colon G\to \mathcal{O}(\mathcal{H})$ is as above we let $\pi_{\CC}$ be the complexification of $\pi.$
\begin{cor}\label{C:HoudSH} Let  $G$ be a countably infinite, discrete group so that $h(L(G))<\infty.$ Suppose that $\pi\colon G\to \mathcal{O}(\mathcal{H})$ is an orthogonal representation on a real Hilbert space and let $\lambda_{C}\colon G\to \mathcal{U}(\ell^{2}(G))$ be the conjugation representation defined by
\[(\lambda_{C}(g)\xi)(x)=\xi(g^{-1}xg).\]
 If $\pi_{\CC}\otimes \lambda_{C}$ is disjoint from $\lambda_{G}$ (regarded as representations of the full $C^{*}$-algebra of $G$), then for any $q\in (-1,1],$ we have
\[h(\Gamma_{q}(\mathcal{H})\rtimes_{\alpha_{\pi}}G)\leq 0.\]
Consequently, if $\pi_{\CC}\otimes \lambda_{C}$ is disjoint from $\lambda_{G},$  then $\Gamma_{q}(\mathcal{H})\rtimes_{\alpha_{\pi}}G$ is not an interpolated free group factor.

\end{cor}

\begin{proof}Set $M=\Gamma_{q}(\mathcal{H})\rtimes_{\alpha_{\pi}}G$
and let $N$ be the copy of $L(G)$ inside $M.$ We use $u_{g}$ for the canonical unitaries in $M$ coming from the elements $g\in G.$  Consider the $L(G)$-$L(G)$ subbimodule of $L^{2}(M)$
\[\mathcal{K}=\overline{\Span\{s_{q}(\xi)u_{g}:\xi\in \mathcal{H},g\in G\}}^{\|\cdot\|_{2}},\]
where $s_{q}(\xi)$ is the canonical $q$-semicircular element in $\Gamma_{q}(\mathcal{H})$ corresponding to $\xi\in\mathcal{H}.$ We claim that $\mathcal{K}$ as an $L(G)$-$L(G)$ bimodule is disjoint from the coarse, which will prove the corollary by Theorem \ref{T:main}. Consider the representation of $G\times G$ on $\mathcal{K}$ given by $(g,h)\xi=u_{g}\xi u_{h}^{-1}.$
Our desired claim is equivalent to saying that this representation is disjoint from $\lambda_{G\times G}.$ It is easy to see that this representation is isomorphic to $\widetilde{\pi}\colon G\times G\to \mathcal{U}(\mathcal{H}_{\CC}\otimes \ell^{2}(G))$
given by $\widetidle{\pi}(g,h)(\xi\otimes \delta_{x})=\pi_{\CC}(g)\xi\otimes \delta_{gxh^{-1}}.$
If we restrict $\widetilde{\pi}$ to the copy of $G$ in $G\times G$ given by $\{(g,g):g\in G\}$
we get $\pi_{\CC}\otimes \lambda_{C}.$ Since $\pi_{\CC}\otimes\lambda_{C}$ is disjoint from $\lambda_{G}$ we  must have that $\widetidle{\pi}$ is disjoint from $\lambda_{G\times G}.$
\end{proof}

Note that if $G$ is finitely-generated and there is a finite $F\subseteq \mathcal{H}$ with
\[\mathcal{H}=\overline{\Span\{\pi(g)\xi:g\in G,\xi\in F\}},\]
then $\Gamma_{q}(\mathcal{H})\rtimes_{\alpha_{\pi}}G$ is finitely-generated. So we find new examples of von Neumann algebras which are strongly $1$-bounded in the sense of Jung.

	Suppose that we take $G=\ZZ$ in the corollary. We then see that if $\pi\colon \ZZ\to \mathcal{O}(\mathcal{H})$ is an orthogonal representation and $U=\pi(1)$ is such that the spectral measure of $U_{\CC}$ is disjoint from the Lebesgue measure, then $h(\Gamma_{q}(\mathcal{H})\rtimes_{\alpha_{\pi}}\ZZ)=0.$ In particular, $\Gamma_{q}(\mathcal{H})\rtimes_{\alpha_{\pi}}\ZZ$ is not isomorphic to an interpolated free group factor.
 We remark that  the previous result was obtained by Houdayer-Shlyakhtenko in \cite{HoudayerShlyakhtenko}, but when $q=0$ and when the spectral measure of $U$ was in the same absolute continuity class of a measure $\nu$ so that all of the convolution powers of $\nu$ are singular with respect to Lebesgue measure.
Our ability to restrict our attention to subbimodules which generate $\Gamma_{q}(\mathcal{H})\rtimes \ZZ$
allows us to ignore the higher-order structure of the $L(\ZZ)-L(\ZZ)$ bimodule $L^{2}(M).$ This is what allows us to ignore the higher convolution powers of $\nu$ as well as render the parameter $q$ \emph{completely irrelevant for the proof}. This is because $q$  is only involved in analyzing products of the form
\[s_{q}(\xi_{1})s_{q}(\xi_{2})\cdots s_{q}(\xi_{n}),\]
where $n\geq 2.$ Thus $q$ plays no part \emph{at all} in our analysis of the $L(\ZZ)-L(\ZZ)$ bimodule $\mathcal{K}.$ This illustrates the flexibility of our results in allowing the submodule in Theorem \ref{T:main} to merely generate $M$ and not be all of $L^{2}(M)$ itself.

We mention an application to type $\textrm{III}$ factors. We will be interested in the free Araki-Woods factors defined by Shlyakhtenko in \cite{DimaFAW}, as well as their $q$-deformations defined by Hiai in \cite{Hiai}. We caution the reader that for $q\ne 0$ it is not known that these are factors and thus we call these the $q$-deformed free Araki-Woods algebras. Our applications will be to the continuous core of such algebras, so we need to extend Theorem \ref{T:main} to semifinite von Neumann algebras. If $M$ is a von Neumann algebra with a faithful, normal, semifinite trace $\tau$ and $p$ is a projection in $M$ with $\tau(p)<\infty,$ we let $\tau_{p}$ be the trace on $pMp$ defined by
\[\tau_{p}(x)=\frac{\tau(x)}{\tau(p)}.\]
We first need to prove a compression fact.
\begin{proposition}\label{P:compressionfact} Let $(M,\tau)$ be a diffuse tracial von Neumann algebra and let $p\in M$ be a nonzero orthogonal projection. If $h(M)\leq 0,$ then $h(pMp)\leq0.$

\end{proposition}

\begin{proof}

Let us first handle the case when $M$ has diffuse center. In this case $Z(M)$ is a diffuse, abelian, regular subalgebra of $M$ and so by Theorem \ref{T:main} we have $h(M)\leq 0.$ It is also straightforward to argue that $pMp$ has diffuse center and so $h(pMp)\leq 0.$ Thus the proposition holds when $M$ has diffuse center.

Thus we may assume that
\[M=\bigoplus_{j=1}^{\infty}M_{j}\]
where the $M_{j}$ are $\textrm{II}_{1}$-factors. Let $z_{j}$ be the central projection corresponding to the unit of $M_{j}.$ If there is a $j$ so that $pz_{j}\ne 0,$ and $M_{j}$ does not embed into an ultrapower of $\R,$ then it is easy to see that $h(pMp)=-\infty.$ We may thus assume that $M_{j}$ embeds into an ultrapower of $\R$ for all $j$ with $pz_{j}\ne 0.$ Replacing $M$ with the direct sum of all the $M_{j}$ such that $M_{j}$ embeds into an ultrapower of $\R,$ we may assume that $M$ embeds into an ultrapower of $\R.$
	
	Since $M$ embeds into an ultrapower of $\R$ it is not hard to argue that for any $j=1,2,\dots,$
\[h(M_{j})\leq \frac{1}{\tau(z_{j})^{2}}h(M)\]
and so we must have that $h(M_{j})=0$ for all $j.$  For every $j$ so that $pz_{j}\ne 0$ we have, by Proposition \ref{P:compressionformula},
\[h(pz_{j}M_{j}pz_{j})\leq \frac{\tau(z_{j})^{2}}{\tau(pz_{j})^{2}}h(M_{j})=0.\]
So by Proposition \ref{P:compressionformula},
\[h(pMp)\leq \sum_{j=1}^{\infty}\tau(pz_{j})^{2}h(pz_{j}M_{j}pz_{j})=0.\]

\end{proof}

We need the following simple fact.

\begin{lemma}\label{L:simpledisjointness} Let $M$ be a von Neumann algebra, and $\mathcal{H},\mathcal{K}$ two disjoint $M$-$M$ bimodules. If $p\in M$ is a nonzero projection, then the $pMp$-$pMp$ bimodules $p\mathcal{H}p,p\mathcal{K}p$ are disjoint.

\end{lemma}

\begin{proof} By Lemma \ref{L:disjointness}, we can find a net $a_{\alpha}\in N\otimes_{\textnormal{alg}}N^{op}$ with
\[\|a_{\alpha}\|_{\textnormal{max}}\leq 1,\]
\[\|a_{\alpha}\# \xi-\xi\|\to 0\mbox{ for all $\xi\in \mathcal{H},$}\]
\[\|a_{\alpha}\# \zeta\|\to 0\mbox{ for all $\zeta\in \mathcal{K}$}.\]
Set $b_{\alpha}=(p\otimes p^{op})a_{\alpha}(p\otimes p^{op}),$
then $b_{\alpha}\in pNp\otimes_{\textnormal{alg}} (pNp)^{op}$ and
\[\|b_{\alpha}\#\xi -\xi\|\to 0\mbox{ for all $\xi\in p\mathcal{H}p,$}\]
\[\|b_{\alpha}\# \zeta\|\to 0\mbox{ for all $\zeta\in p\mathcal{K}p.$}\]
A priori $\|b_{\alpha}\|_{(pNp)\otimes_{\textnormal{max}}(pNp)^{op}}>1,$
but as $\|b_{\alpha}\|_{N\otimes_{\textnormal{max}}N^{op}}\leq 1$
one still has
\[\|b_{\alpha}\#\xi\|\leq \|\xi\|, \mbox{ for all $\xi\in p\mathcal{H}p,$}\]
\[\|b_{\alpha}\# \zeta\|\leq \|\zeta\|, \mbox{ for all $\zeta\in p\mathcal{K}p.$}\]
It is easy from the above to argue as in Lemma \ref{L:disjointness} that $\Hom_{pNp-pNp}(p\mathcal{H}p,p\mathcal{K}p)=\{0\}$
and this completes the proof.

\end{proof}

\begin{theorem}\label{T:seminifinitedijsiontness} Let $M$ be a diffuse semifinite von Neumann algebra and fix a faithful, normal, semifinite trace $\tau$ on $M.$ Let $N$ be a diffuse von Neumann subalgebra of $M$ so that $\tau\big|_{N}$ is semifinite. Suppose that there exists $S\subseteq L^{2}(M,\tau)$ so that $W^{*}(S)=M$ and such that
\[\overline{\sum_{x\in S}L^{2}(NxN)}^{\|\cdot\|_{2}}\]
is disjoint from $L^{2}(N)\otimes L^{2}(N^{op})$ as an $N$-$N$ bimodule. Lastly, suppose that for every projection $q$ in $N$ with $\tau(q)<\infty$ we have that $h(qNq,\tau_{q})\leq 0.$ Then for every projection $p\in M$ with $\tau(p)<\infty$ we have $h(pMp,\tau_{p})\leq 0.$

\end{theorem}

\begin{proof}
Let us first handle the case when $M$ is a factor.
Set
\[\mathcal{H}=\overline{\sum_{x\in S}L^{2}(NxN)}^{\|\cdot\|_{2}}\]
and let $p\in M$ with $\tau(p)<\infty.$ Because $M$ is a factor the isomorphism class of $pMp$ only depends upon $\tau(p).$ As $N$ is diffuse and$\tau\big|_{N}$ is semifinite we may assume that $p\in N.$ Let $p_{n}\in N$ be an increasing sequence of  projections with $\tau(p_{n})<\infty$  such that $p_{n}\geq p$ and $p_{n}\to 1$ in the strong operator topology. Let
\[M_{n}=W^{*}(p_{n}Np_{n},p_{n}\mathcal{H}p_{n}).\]
By  Lemma \ref{L:simpledisjointness}, we know that $p_{n}\mathcal{H}p_{n}$ is disjoint from $L^{2}(p_{n}Np_{n})\otimes L^{2}((p_{n}Np_{n})^{op})$ as a $p_{n}Np_{n}$-$p_{n}Np_{n}$ bimodule and thus by Theorem \ref{T:main}
\[h(M_{n},\tau_{p_{n}})\leq 0.\]
Note that
\[pMp=\overline{\bigcup_{n=1}^{\infty}pM_{n}p}^{SOT},\]
since this is an increasing union Lemma $\ref{L:increasingunions}$ implies that
\[h(pMp,\tau_{p})=\sup_{n}h(pM_{n}p:pMp).\]
Thus it is enough to show that
\[h(pM_{n}p,\tau_{p})\leq 0\]
for every $n\in \NN$ and this follows from Proposition \ref{P:compressionfact}.

We now turn to the case when $M$ is not a factor. We may find a set $J$ and central projections $z,(z_{j})_{j\in J}$ in $M$ so that
\begin{itemize}
\item $z+\sum_{j\in J}z_{j}=1$,\\
\item $zM$ has diffuse center,\\
\item $z_{j}M$ is a factor for every $j\in J.$
\end{itemize}
It is direct to show that $\tau\big|_{zM}$ and $\tau\big|_{z_{j}M}$ are semifinite for each $j\in J$. Fix a projection $p\in M$ with $\tau(p)<\infty.$  It is straightforward to see that if $pzMp\ne\{0\},$ then $pzMp$ has diffuse center and thus $h(pzMp,\tau_{pz})\leq 0.$ Let $J_{0}=\{j\in J:pz_{j}\ne 0\}$ and observe that, as $\tau(p)<\infty,$ the set $J_{0}$ is countable.
So by part $(ii)$ of Proposition \ref{P:compressionformula} it is enough to show that $h(pz_{j}Mpz_{j},\tau_{pz_{j}})\leq 0.$ for every $j\in J_{0}.$ Fix a $j\in J_{0}$ and note that, by centrality of $z_{j},$ we have that $Nz_{j}$ is a von Neumann subalgebra of $z_{j}Mz_{j}$ (even though $z_{j}$ is not  in $N$). It is also direct to check that $Nz_{j}$ is diffuse and that $\tau\big|_{Nz_{j}}$ is semifinite. As $z_{j}$ is central we have $W^{*}(\{z_{j}x:x\in S\})=Mz_{j}.$ Since we already handled the case when $M$ is a factor, it is enough to show that for every $x\in S$ we have that  $L^{2}(Nz_{j}xN)$ is disjoint from $L^{2}(z_{j}N)\otimes L^{2}((z_{j}N)^{op})$ as a $z_{j}N-z_{j}N$ bimodule. But this can be argued exactly as in Lemma \ref{L:simpledisjointness}.

\end{proof}

\begin{cor}\label{C':singularladjgadlkj} Let $\mathcal{H}$ be a real Hilbert space and let $q\in [-1,1].$ Suppose that $t\mapsto U_{t}$ is a one-parameter orthogonal group on $\mathcal{H}.$ Let $A$ be a self-adjoint closed operator on $\mathcal{H}_{\CC}$ so that $U_{t,\CC}=e^{itA}.$ Suppose that the spectral measure of $A$ is singular with respect to Lebesgue measure. Let $M$  be the continuous core of $\Gamma_{q}(\mathcal{H},U_{t})''$ with respect to the $q$-quasi-free state and $\tau$ the semifinite trace on $M$ induced by the $q$-quasi-free state on $\Gamma_{q}(\mathcal{H},U_{t})''.$  Then for every projection $p\in M$ with $\tau(p)<\infty$
we have $h(pMp)=0.$ In particular, the continuous core of $\Gamma_{q}(\mathcal{H},U_{t})''$ is not isomorphic to $L(\FF_{s})\overline{\otimes} B(\mathcal{H}')$ for any $s\in (1,\infty]$ and any Hilbert space $\mathcal{H}'.$

\end{cor}

\begin{proof} Let $\phi$ be the $q$-quasi-free state on $M.$ By definition $M=\Gamma_{q}(\mathcal{H},U_{t})''\rtimes_{\sigma^{\phi}}\RR,$
where $\sigma^{\phi}$ is the modular automorphism group. Let $B$ be the copy of $L(\RR)$ inside $M$ and consider the $B$-$B$ subbimodule of $L^{2}(M)$ given as
\[\mathcal{K}=\overline{\Span\{s_{q}(\xi)b:b\in B,\xi\in\mathcal{H}\}}^{\|\cdot\|_{2}}.\]
Here $s_{q}(\xi)$ is the canonical self-adjoint in $\Gamma_{q}(\mathcal{H},U_{t})''$ corresponding to the vector $\xi\in\mathcal{H}.$ Let $\rho\colon \RR\times \RR\to \mathcal{U}(\mathcal{K})$
be given by $\rho(t,s)\zeta=u_{t}\zeta u_{-s},$
where $(u_{x})_{x\in \RR}$  are the canonical unitaries coming from $B.$ Let $M$ be the Lebesgue measure on $\RR$ and define $\pi\colon \RR\times \RR\to \mathcal{U}(\mathcal {H}_{\CC}\otimes L^{2}(\RR,m))$ by
\[\pi(t,s)(\xi\otimes f)=U_{t}\xi\otimes \lambda(t-s)f\]
where $\lambda$ is the left regular representation. It is not hard to see that $\rho\cong \pi$  and  it is easy to argue (using that the spectral measure of $A$ is singular with respect to the Lebesgue measure) as in Corollary \ref{C:HoudSH} that $\pi$ is disjoint from $\lambda_{\RR\times \RR}.$ It follows that $\rho$ is disjoint from $\lambda_{\RR\times \RR}$ and thus $\mathcal{K}$ is disjoint from $L^{2}(B)\otimes L^{2}(B)$ as a $B$-$B$ bimodule. The proof is now completed by invoking Theorem \ref{T:seminifinitedijsiontness}.

\end{proof}

\begin{cor}\label{C:dlaghaljhdga}  Let $\mathcal{H}$ be a real Hilbert space,  let $q\in [-1,1],$ and $t\mapsto U_{t}$ a one-parameter orthogonal group on $\mathcal{H}.$ Let $A$ be a self-adjoint closed operator on $\mathcal{H}_{\CC}$ so that $U_{t,\CC}=e^{itA}.$ Let $m$ be the Lebesgue measure on $\RR$ and $t\mapsto \lambda_{t}$ be the left regular representation $\RR.$ If the spectral measure on $A$ is singular with respect to $m$, then
\[\Gamma_{q}(\mathcal{H},U_{t})''\not\cong \Gamma_{0}(L^{2}(\RR,m),\lambda_{t}).\]

\end{cor}

\begin{proof} This follows from the fact that the continuous core of $\Gamma_{0}(L^{2}(\RR,m),\lambda_{t})''$ is isomorphic to $L(\FF_{\infty})\overline{\otimes}B(\ell^{2}(\NN))$  (see \cite{DimaFAW}).

\end{proof}

Again Corollaries \ref{C':singularladjgadlkj}  and \ref{C:dlaghaljhdga} were previously only known when $q=0$ and the spectral measure of $A$ is in the same absolute continuity class of a measure $\nu$ so that all of the convolution powers of $\nu$ are singular with respect to Lebesgue measure (see \cite{DimeFullIII}).
\appendix

\section{Properties of $1$-Bounded Entropy}\label{S:properties1bddent}

In this section we prove that our general version of $1$-bounded entropy is an invariant. The proof is similar to Theorem 3.2 in \cite{JungSB}. We also establish some properties of $1$-bounded entropy which will be important for our main results.

\begin{definition}\emph{Let $F$ be a finite set and $k\in\NN.$ For $\Omega\subseteq M_{k}(\CC)^{F}$ and $\varepsilon>0$ we say that $S\subseteq \Omega$ is $\varepsilon$-orbit dense if for every $T\in\Omega,$ there is a $U\in \mathcal{U}(k),A\in S$ so that}
\[\sum_{a\in F}\|T_{a}-U^{*}A_{a}U\|_{2}^{2}<\varepsilon^{2}.\]
\emph{We let $K^{\mathcal{O}}_{\varepsilon}(\Omega,\|\cdot\|_{2})$ be the minimal cardinality of an $\varepsilon$-orbit dense subset of $S.$}
\end{definition}

\begin{definition}\emph{Let $(M,\tau)$ be a tracial von Neumann algebra and let $X,Y\subseteq M_{sa}$ with $W^{*}(X)\subseteq W^{*}(Y).$ Let $R\in [0,\infty)^{X},R'\in[0,\infty)^{Y}$ be cutoff parameters for $X,Y.$ For a natural number $m,$ positive real numbers $\gamma,\varepsilon$ and finite $F\subseteq X,G\subseteq Y$ let}
\[K_{\varepsilon}^{\mathcal{O}}(\Gamma_{R\vee R'}(F:G;m,\gamma),\|\cdot\|_{2})=\limsup_{k\to\infty}\frac{1}{k^{2}}\log K_{\varepsilon}^{\mathcal{O}}(\Gamma_{R\vee R'}(F:G;m,\gamma),\|\cdot\|_{2}),\]
\[K_{\varepsilon}^{\mathcal{O}}(\Gamma_{R\vee R'}(F:G),\|\cdot\|_{2})=\inf_{m\in\NN,\gamma>0}K_{\varepsilon}^{\mathcal{O}}(\Gamma_{R\vee R'}(F:G;m,\gamma),\|\cdot\|_{2}),\]
\[K_{\varepsilon}^{\mathcal{O}}(\Gamma_{R\vee R'}(F:Y),\|\cdot\|_{2})=\inf_{\textnormal{ finite } G\subseteq Y}K_{\varepsilon}^{\mathcal{O}}(F:G),\|\cdot\|_{2}),\]
\[h^{\mathcal{O}}(X:Y)=\sup_{\substack{\varepsilon>0,\\ F\subseteq X\textnormal{ finite}}}K_{\varepsilon}^{\mathcal{O}}(\Gamma_{R\vee R'}(F:Y),\|\cdot\|_{2}).\]
\end{definition}

By standard arguments neither of $h^{\mathcal{O}}(X:Y),h_{(A_{k})_{k=1}^{\infty}}(X:Y)$
depend upon the choice of cutoff parameters.
We will show that $h^{\mathcal{O}}(X:Y)=h_{(A_{k})_{k=1}^{\infty}}(X:Y)$ and we start by proving two lemmas that will ease the proof of this equality.

 	 We briefly summarize the rough idea of the proof of this equality in the case when $X,Y$ are finite (though we will handle the general case). Suppose that $m'\in \NN$ and $\varepsilon',\gamma'>0$ are given and that $S\subseteq \Gamma_{R_{X}\vee R_{Y}}(X:Y;m',\gamma',k)$ is $\varespilon'$-orbit dense. Then if $m'$ is larger than $m$ and $\gamma'<\gamma,$ any $B\in \Xi_{A_{k},R_{X}\vee R_{Y}}(X:Y;m,\gamma,k)$ can (tautologically) be approximated by unitary conjugates of a $B'\in S.$ Moreover, since $(A_{k},X)\in \Gamma_{L\vee R_{X}\vee R_{Y}}(a,X,Y:m,\gamma,k)$ for some $L\in (0,\infty),$ if we choose $m'$ large enough and $\gamma'$ small enough, then  if $U,V\in \mathcal{U}(k)$ both approximately conjugate $B$ to $B'$ we must have that $V^{*}U$ almost commutes with $A_{k}$ (here we use that $a\in W^{*}(X)$). From this, we can in fact show that we can bound the size of a minimal almost dense subset of $\Xi_{A_{k},R_{X}\vee R_{Y}}(X:Y;m,\gamma,k)$ by $|S|$ times the size of an almost dense subset of $\{A_{k}\}'\cap \mathcal{U}(k).$  Our first lemma shows that the size of an almost dense subset of $\{A_{k}\}'\cap \mathcal{U}(k)$ grows subexponentially.

\begin{lemma}\label{L:smallcommutationA} Let $(M,\tau)$ be a tracial von Neumann algebra and $a\in M_{sa}$ with $W^{*}(a)$ being diffuse. Fix a sequence $(A_{k})_{k=1}^{\infty}$ of microstates for $a.$
Then, for any $\varepsilon>0$ we have
\[\lim_{k\to\infty}\frac{1}{k^{2}}\log K_{\varepsilon}(\{U\in \mathcal{U}(k):UA_{k}=A_{k}U\},\|\cdot\|_{\infty})=0.\]
\end{lemma}

\begin{proof} Let $M=\sup_{k}\|a_{k}\|_{\infty}$
and fix a natural number $m.$  Since $W^{*}(a)$ is diffuse, we may find real numbers
\[-M-1=t_{0}<t_{1}<t_{2}<\cdots<t_{m-1}<t_{m}=M+1\]
with
\[\tau(\chi_{[t_{j-1},t_{j})}(a))=\frac{1}{m},\mbox{ for $1\leq j\leq m$.}\]
Since $W^{*}(a)$ is diffuse, we additionally have
\begin{equation}\label{E:diffuse}
\chi_{[t_{j-1},t_{j})}(a)=\chi_{[t_{j-1},t_{j}]}(a), \mbox{ for $1\leq j\leq m$}.
\end{equation}
For $1\leq j\leq m,$ set $P_{j,k}=\chi_{[t_{j-1},t_{j})}(A_{k}).$
Suppose that $U\in \mathcal{U}(k)$ and $[U,A_{k}]=0.$ Then $[U,P_{j,k}]=0$ as well and so we can write
\[U=\sum_{j=1}^{k}U_{j}P_{j,k}\]
where $U_{j}\in \mathcal{U}(P_{j,k}\ell^{2}(k)).$ By a result of S. Szarek (see \cite{Szarek}), there is a constant $C>0$ so that
\[K_{\varepsilon}(\mathcal{U}(P_{j,k}\ell^{2}(k)),\|\cdot\|_{\infty})\leq\left(\frac{C}{\varepsilon}\right)^{k^{2}\tr(P_{j,k})^{2}}.\]
From this it is not hard to see that
\[\frac{1}{k^{2}}\log K_{2\varepsilon}(\{U\in \mathcal{U}(k):[U,A_{k}]=0\},\|\cdot\|_{\infty})\leq\log\left(\frac{C}{\varepsilon}\right)\sum_{j=1}^{m}\tr(P_{j,k})^{2}.\]
Since $A_{k}$ is a microstate sequence for $a$ we have for all $1\leq j\leq m$
\[\limsup_{k\to\infty}\tr(P_{j,k})=\limsup_{k\to\infty}\tr(\chi_{[t_{j-1},t_{j})}(A_{k}))\leq \tau(\chi_{[t_{j-1},t_{j}]}(a))=\frac{1}{m},\]
the last equality following from $(\ref{E:diffuse}).$ Thus
\[\limsup_{k\to\infty}\frac{1}{k^{2}}\log K_{2\varepsilon}(\{u\in \mathcal{U}(k):[U,A_{k}]=0\},\|\cdot\|_{\infty})\leq \log\left(\frac{C}{\varpesilon}\right)\frac{1}{m}.\]
Since $m$ is arbitrary, we can let $m\to\infty$ to complete the proof.

\end{proof}

Lemma \ref{L:smallcommutationA} will be useful in conjunction with the following lemma.

\begin{lemma}\label{L:commutingorbitdenseA} Let $(M,\tau)$ be a tracial von Neumann algebra and $X,Y\subseteq M_{sa}$ be given. Let $R_{X}\in [0,\infty)^{X},R_{Y}\in [0,\infty)^{Y}$ be cutoff parameters. Suppose that $W^{*}(X)\subseteq W^{*}(Y)$ and that $W^{*}(X)$ is diffuse. Let $a\in W^{*}(X)$ have diffuse spectrum and let $(A_{k})_{k=1}^{\infty}$ be microstates for $a.$ Fix  finite $F\subseteq X,G_{0}\subseteq Y,$ a natural number $m'$ and $\varpesilon,\gamma'\in (0,\infty).$ Then there exists finite $G\subseteq Y,F_{0}\subseteq X$ with $F\subseteq F_{0},$  a natural number $m\in\NN$ and $\gamma,\varepsilon'>0$ so that for all sufficiently large natural numbers $k,$ there is a $S\subseteq \Gamma_{R_{X}\vee R_{Y}}(F_{0}:G;m',\gamma',k)$ with
\[|S|\leq K_{\varepsilon'}^{\mathcal{O}}(\Gamma_{R_{X}\vee R_{Y}}(F_{0}:G_{0};m',\gamma',k),\|\cdot\|_{2})\]
and
\[\Xi_{A_{k},R_{X}\vee R_{Y}}(F:G;m,\gamma,k)\subseteq_{\varespilon,\|\cdot\|_{2}}\{U^{*}(B_{x}')_{x\in F}U:B'\in S,U\in \mathcal{U}(k),UA_{k}=A_{k}U\}.\]

\end{lemma}

\begin{proof}
Set $R_{F}=\max_{x\in F}R_{X,x}.$ Let $0<\eta$ be sufficiently small so that for all $k\in \NN$ large enough and for all $U\in \mathcal{U}(k)$ with $\|U^{*}A_{k}U-A_{k}\|_{2}<\eta,$
 there is a $V\in \mathcal{U}(k)$ with $[V,A_{k}]=0$ and
$\|U-V\|_{2}<\frac{\varepsilon}{\sqrt{|F|}(R_{F}+1)}.$
We may find a finite subset $F_{0}$ of $X$ with $F\subseteq F_{0}$ such that there exists a $Q\in \CC\ip{T_{x}:x\in F_{0}}$ with
\[\|Q(x:x\in F_{0})-a\|_{2}<\frac{\eta}{4},\]
\[\|Q(x:x\in F_{0})\|_{\infty}\leq \|a\|_{\infty}.\]
This is possible by Kaplansky's density theorem, as $a\in W^{*}(X).$

	Let $D>1$ be such that
\[\|Q(B_{x}:x\in F_{0})-Q(B'_{x}:x\in F_{0})\|_{2}\leq D\|B-B'\|_{2}.\]
 for any $k\in\NN,$ and any $B,B'\in M_{k}(\CC)^{F_{0}}$ with $\|B_{x}\|_{\infty},\|B_{x}'\|_{\infty}\leq R_{X,x}$ for all $x\in F_{0}.$ We may choose a $m''\in\{m',m'+1,\cdots\}$  sufficiently large and a $\gamma''\in (0,\gamma')$  sufficiently small so that if $A',A''\in\Gamma_{L}(a;m'',\gamma'',k),$ and $k$ is sufficiently large, then there is a $U\in\mathcal{U}(k)$ so that
\[\|U^{*}A'U-A''\|_{2}<\frac{\eta}{4}.\]
We will also assume that $m''$ is sufficiently large and that $\gamma''>0$ is sufficiently small so that if
\[(B,A)\in \Gamma_{R_{X}\vee L}(F_{0},a;m'',\gamma'',k),\]
then
\[\|Q(B_{x}:x\in F_{0})-A\|_{2}<\frac{\eta}{4}.\]

We may choose a finite $G\subseteq Y,$ an $m\in \NN,\gamma>0$ so that for all large enough $k\in \NN$
\[\Gamma_{L\vee R_{X}\vee R_{Y}}(a,F:G,m,\gamma,k)\subseteq \Gamma_{L\vee R_{X}\vee R_{X}\vee R_{Y}}(a,F:F_{0}\setminus F,G_{0};m'',\gamma'',k).\]
Fix an $\varespilon'<\min\left(\frac{\eta}{4D},\frac{\varepsilon}{2}\right)$ and choose $S\subseteq \Gamma_{R_{X}\vee R_{Y}\vee L}(F_{0}:G_{0},a;m'',\gamma'',k)$ which is $2\varepsilon'$-orbit dense with respect to $\|\cdot\|_{2}$ and has
\[|S|=K_{2\varepsilon'}^{\mathcal{O}}(\Gamma_{R_{X}\vee R_{Y}\vee L}(F_{0}:G_{0},a;m'',\gamma'',k),\|\cdot\|_{2})\leq K_{\varepsilon'}^{\mathcal{O}}(\Gamma_{R_{X}\vee R_{Y}}(F_{0}:G_{0};m',\gamma',k),\|\cdot\|_{2}).\]
Given $B'\in S,$  we may find an $A_{B'}\in M_{k}(\CC)$ so that
\[(B',A_{B'})\in \Gamma_{R_{X}\vee L\vee R_{Y}}(F_{0},a:G_{0};m'',\gamma'',k).\]
By our choice of $m'',\gamma''$ we may find a $U_{B'}\in \mathcal{U}(k)$ with
\[\|U_{B'}^{*}A_{B'}U_{B'}-A_{k}\|_{2}<\frac{\eta}{4}.\]
Replacing $S$ with $\{U_{B'}^{*}B'U_{B'}:B'\in S\}$ we will assume that in fact
\[\|A_{B'}-A_{k}\|_{2}<\frac{\eta}{4}\]
for all $B'\in S.$

 Suppose we are given a $B\in \Xi_{(A_{k})_{k=1}^{\infty},R_{X}\vee R_{Y}}(F:G;m,\gamma,k).$ Then if $k$ is large enough, we may choose a
 \[\widetilde{B}\in  \Xi_{A_{k},R_{X}\vee R_{Y}}(F_{0}:G_{0};m'',\gamma'',k)\]
  with $\widetilde{B}_{x}=B_{x}$ for all $x\in F.$
For all large $k,$ we may choose a $B'\in S$ and a $V\in \mathcal{U}(k)$ with
\[\|V^{*}B'V-\widetilde{B}\|_{2}\leq 2\varepsilon'<\varepsilon.\]
Our choice of $m'',\gamma''$ imply that
\[\|A_{k}-V^{*}A_{B'}V\|_{2}\leq \frac{2\eta}{4}+\|Q(\widetilde{B}_{x}:x\in F_{0})-Q(V^{*}B_{x}'V:x\in F_{0})\|_{2}\]
and since $\varespilon'<\frac{\eta}{4D}$ we have
\[\|A_{k}-V^{*}A_{B'}V\|_{2}<\frac{3\eta}{4}.\]
Thus
\[\|A_{k}-VA_{k}V^{*}\|_{2}\leq \|A_{B'}-A_{k}\|_{2}+\|V^{*}A_{B'}V-A_{k}\|_{2}<\eta.\]
By our choice of $\eta,$  if $k$ is sufficiently large, then there is a $W\in\mathcal{U}(k)$ with $[W,A_{k}]=0$ and
\[\|W-V\|_{2}<\frac{\varepsilon}{\sqrt{|F|}(R_{F}+1)},\]
which implies that $\|B-W^{*}(B'_{x})_{x\in F}W\|_{2}<3\varepsilon.$ Since $B'\in S$ and $\varepsilon>0$ is arbitrary, the proof is complete.

\end{proof}

\begin{lemma}\label{L:orbit} Let $(M,\tau)$ be a tracial von Neumann algebra and let $X,Y\subseteq M_{sa}$ with $W^{*}(X)\subseteq W^{*}(Y)$ and $W^{*}(X)$  diffuse. Fix $a\in W^{*}(X)$ with diffuse spectrum and $(A_{k})_{k=1}^{\infty}$  microstates for $a.$ Then
\[h_{(A_{k})_{k=1}^{\infty}}(X:Y)=h^{\mathcal{O}}(X:Y).\]
\end{lemma}

\begin{proof}

Set $L=\sup_{k}\|A_{k}\|_{\infty}.$
Let us first show that
$ h^{\mathcal{O}}(X:Y) \leq h_{(A_{k})_{k=1}^{\infty}}(X:Y).$  Fix $m\in \NN,\gamma>0$ and finite $F\subseteq X,G\subseteq Y.$ Arguing as in Lemma 4.2 of  \cite{OrbFreeEnt}, we may find a $m'\in\NN,\gamma'>0$ and a finite $G_{0}\subseteq Y$ so that
$\Gamma_{R_{X}\vee R_{Y}}(F:G_{0};m',\gamma'',k)$ is contained in the unitary conjugation orbit of $\Xi_{(A_{k}),R_{X}\vee R_{Y}}(F:G;m,\gamma,k).$ So
\[K_{2\varepsilon}^{\mathcal{O}}(\Gamma_{R_{X}\vee R_{Y}}(F:G_{0};m',\gamma',k),\|\cdot\|_{2})\leq K_{\varespilon}(\Xi_{(A_{k}),R_{X}\vee R_{Y}}(F:G;m,\gamma,k),\|\cdot\|_{2}).\]
Taking $\frac{1}{k^{2}}\log$ of both sides and letting $k\to\infty$ we have
\begin{align*}
K_{2\varepsilon}^{\mathcal{O}}(\Gamma_{R_{X}\vee R_{Y}}(F:Y),\|\cdot\|_{2})&\leq K_{2\varepsilon}^{\mathcal{O}}(\Gamma_{R_{X}\vee R_{Y}}(F:G_{0}),\|\cdot\|_{2}) \\
&\leq K_{2\varepsilon}^{\mathcal{O}}(\Gamma_{R_{X}\vee R_{Y}}(F:G_{0};m',\gamma'),\|\cdot\|_{2})\\
&\leq K_{\varespilon}(\Xi_{(A_{k})_{k=1}^{\infty},R_{X}\vee R_{Y}}(F:G;m,\gamma),\|\cdot\|_{2}).
\end{align*}
Now taking the infimum over all $G,m,\gamma$ we have
\[K_{2\varepsilon}^{\mathcal{O}}(\Gamma_{R_{X}\vee R_{Y}}(F:Y),\|\cdot\|_{2})\leq K_{\varepsilon}(\Xi_{(A_{k})_{k=1}^{\infty},R_{X}\vee R_{Y}}(F:Y),\|\cdot\|_{2})\]
and taking the supremum over all $\varepsilon,F$ shows that
\[h^{\mathcal{O}}(X:Y)\leq h_{(A_{k})_{k=1}^{\infty}}(X:Y).\]

We turn to proving that $h_{(A_{k})_{k=1}^{\infty}}(X:Y,\|\cdot\|_{2})\leq h^{\mathcal{O}}(X:Y,\|\cdot\|_{2}).$
Fix $\varepsilon>0,$  finite $F\subseteq X,G_{0}\subseteq Y$  and set $R_{F}=\max_{x\in X}R_{X,x}.$ Fix a natural number $m'$ and a positive real number $\gamma'>0.$ Let $F_{0},G,m,\varespilon',\gamma$ be as in Lemma \ref{L:commutingorbitdenseA} for this $\varepsilon,F,G_{0},m',\gamma'.$ Given a large enough natural number $k,$ let $S$ be as in the conclusion of Lemma \ref{L:commutingorbitdenseA}. Choose a  $D\subseteq\{U\in \mathcal{U}(k):[U,A_{k}]=0\}$ which is $\frac{\varepsilon}{\sqrt{|F|}(R_{F}+1)}$-dense with respect to $\|\cdot\|_{\infty}$ and so that
\[|D|=K_{\frac{\varepsilon}{\sqrt{|F|}(R_{F}+1)}}(\{U\in \mathcal{U}(k):[U,A_{k}]=0\},\|\cdot\|_{\infty}).\]
Given $B\in \Xi_{A_{k},R_{X}\vee R_{Y}}(F:G;m,\gamma,k),$ Lemma \ref{L:commutingorbitdenseA} allows us to find a $B'\in S$ and a $U\in \mathcal{U}(k)$ with $[U,A_{k}]=0$ and so that
\[\|B-(U^{*}B'_{x}U)_{x\in F}\|_{2}<\varespilon.\]
Choose a $W\in D$ with $\|W-V\|_{\infty}<\frac{\varepsilon}{\sqrt{|F|}(R_{F}+1)},$  we then have
\[\|B-(W^{*}B_{x}W)_{x\in F}\|_{2}<3\varepsilon.\]
Thus
\[\Xi_{(A_{k}),R_{X}\vee R_{Y}}(F:G;m,\gamma,k)\subseteq_{3\varpesilon,\|\cdot\|_{2}}\{W^{*}(B'_{x})_{x\in F}W:W\in D,B'\in S\},\]
so
\begin{align*}
K_{6\varepsilon}(\Xi_{A_{k},R_{X}\vee R_{Y}}(F:G;m,\gamma,k),\|\cdot\|_{2})
&\leq |S||D|\\
&\leq K_{\varepsilon'}^{\mathcal{O}}(\Gamma_{R_{X}\vee R_{Y}}(F_{0}:G_{0};m',\gamma',k),\|\cdot\|_{2})\times \\
&K_{\frac{\varepsilon}{\sqrt{|F|}(R_{F}+1)}}(\{U\in \mathcal{U}(k):[U,A_{k}]=0\}),\|\cdot\|_{\infty}).
\end{align*}
Applying $\frac{1}{k^{2}}\log$ to both sides of this inequality, letting $k\to\infty$ and applying  Lemma \ref{L:smallcommutationA} we see that
\[K_{6\varepsilon}(\Xi_{(A_{k})_{k=1}^{\infty},R_{X}\vee R_{Y}}(F:G;m,\gamma),\|\cdot\|_{2})\leq  K_{\varepsilon'}^{\mathcal{O}}(\Gamma_{R_{X}\vee R_{Y}}(F_{0}:G_{0};m',\gamma'),\|\cdot\|_{2}).\]
A fortiori,
\[K_{6\varepsilon}(\Xi_{(A_{k})_{k=1}^{\infty},R_{X}\vee R_{Y}}(F:Y),\|\cdot\|_{2})\leq K_{\varepsilon'}^{\mathcal{O}}(\Gamma_{R_{X}\vee R_{Y}}(F_{0}:G_{0};m',\gamma'),\|\cdot\|_{2}).\]
Taking the infimum over all $G_{0},m',\gamma'$ we see that
\[K_{6\varepsilon}(\Xi_{(A_{k})_{k=1}^{\infty}R_{X}\vee R_{Y}}(F:Y),\|\cdot\|_{2})\leq K_{\varepsilon'}^{\mathcal{O}}(\Gamma_{R_{X}\vee R_{Y}}(F_{0}:Y),\|\cdot\|_{2})\leq h^{\mathcal{O}}(X:Y).\]
Now taking the supremum over $\varepsilon>0$ we have
\[h_{(A_{k})_{k=1}^{\infty}}(F:Y)\leq h^{\mathcal{O}}(X:Y).\]
We then take the supremum over all $F$ to complete the proof.

\end{proof}

\begin{cor}\label{C:independlajgalj} Let $(M,\tau)$ be a tracial von Neumann algebra, and let $X,Y\subseteq M_{sa}$ be finite with $W^{*}(X)\subseteq W^{*}(Y).$  Let $a,b\in W^{*}(X)_{sa}$ be such that $W^{*}(a),W^{*}(b)$ are diffuse. Let $(A_{k})_{k=1}^{\infty},(B_{k})_{k=1}^{\infty}$ be microstates for $a,b$ respectively. Then
\[h_{(A_{k})_{k=1}^{\infty}}(X:Y)=h_{(B_{k})_{k=1}^{\infty}}(X:Y).\]
\end{cor}
\begin{proof} From Lemma \ref{L:orbit} we have
\[h_{(A_{k})_{k=1}^{\infty}}(X:Y)=h^{\mathcal{O}}(X:Y)=h_{(B_{k})_{k=1}^{\infty}}(X:Y).\]

\end{proof}
Because of Corollary \ref{C:independlajgalj} we use
\[h(X:Y)=h_{(A_{k})_{k=1}^{\infty}}(X:Y)\]
for any sequence of microstates $(A_{k})_{k=1}^{\infty}$ for an element $a\in W^{*}(X)_{sa}$ with  diffuse spectrum.
We wish to show that if $W^{*}(X)\subseteq W^{*}(Y),$ then
\[h(X:Y)=h(W^{*}(X):W^{*}(Y)).\]
The following Lemma will be useful in the proof.
\begin{lemma}\label{L:extension} Let $(M,\tau)$ be a  tracial von Neumann algebra, and let $Y\subseteq M_{sa}.$ Fix  finite $F,G\subseteq M_{sa}$ with $G\subseteq W^{*}(Y).$ Let $R_{Y}\in [0,\infty)^{Y},R_{F}\in [0,\infty)^{F},R_{G}\in [0,\infty)^{G}$ be cutoff parameters. Then for every $m\in\NN,\gamma>0,$ there is a $m'\in\NN,\gamma'>0$ and a finite $Y_{0}\subseteq Y$ so that for every $k\in \NN$
\[\Gamma_{R_{F}\vee R_{Y}}(F:Y_{0};m',\gamma',k)\subseteq \Gamma_{R_{F}\vee R_{G}}(F:G;m,\gamma,k).\]
\end{lemma}
\begin{proof} Choose a cutoff parameter $R_{M}\in [0,\infty)^{M}$ with $R_{M}\big|_{Y}=R_{Y}.$ We may find a $\eta>0$ so that if $(y_{a})_{a\in G}\in M_{sa}^{G}$ and
\[\|a-y_{a}\|_{2}<\eta\mbox{ for all $a\in G$,}\]
\[\|y_{a}\|_{\infty}\leq \|a\|_{\infty}\mbox{ for all $a\in G$,}\]
then for all monomials $P\in \CC\ip{X_{b},T_{a}:b\in F,a\in G}$ of degree at most $m$ we have
\[|\tau(P(b,a:b\in F,a\in G))-\tau(P(b,y_{a}:b\in F,y_{a}\in G))|<\gamma.\]
By Kaplansky's density theorem we may choose a finite $Y_{0}\subseteq Y$ so that there every $a\in G,$ there is a self-adjoint $P_{a}\in \CC\ip{T_{c}:c\in Y_{0}}$ with
\[\|P_{a}(c:c\in Y_{0})-c\|_{2}<\eta\]
\[\|P_{a}(c:c\in Y_{0})\|_{\infty}\leq \|a\|_{\infty}.\]

Now choose a $m'\in\NN,\gamma'>0$ so that
\[\Gamma_{R_{F}\vee R_{Y}}(F:Y_{0};m',\gamma',k)\subseteq \Gamma_{R_{F}\vee R_{Y}\vee R_{M}}(F:Y_{0},(P_{a}(c:c\in Y_{0}))_{c\in F};m,\gamma,k).\]
Now let $B\in \Gamma_{R_{F}\vee R_{Y}}(F:Y_{0};m',\gamma',k)$
and choose $C\in M_{k}(\CC)^{G}$ so that
\[(B,C)\in\Gamma_{R_{F}\vee R_{M}}(F,(P_{a}(c:c\in Y_{0}))_{a\in G};m,\gamma,k).\]
Then for every $P\in \CC\ip{X_{b},Y_{a};b\in F,a\in G}$ we have
\[|\tr(P(B_{b},C_{a}:b\in F,a\in G))-\tau(P(b,a:b\in F,a\in G))|\leq \gamma+\]
\[|\tr(P(B_{b},C_{a}:b\in F,a\in G))-\tau(P(b,P_{a}(c:c\in Y_{0})):b\in F,a\in G))|\leq 2\gamma.\]
Thus $B\in \Gamma_{R_{F}\vee R_{G}}(F:G;m,2\gamma,k)$
and since $\gamma$ was arbitrary the proof is complete.
\end{proof}
From the above lemma, it is easy to remove the dependence upon $Y.$
\begin{lemma}\label{L:Yindependadjlgakjg} Let $(M,\tau)$ be a tracial von Neumann algebra. Let $X,Y\subseteq M_{sa},$ with $W^{*}(X)$ diffuse. Then
\[h(X:Y)=h(X:W^{*}(Y)).\]
\end{lemma}

\begin{proof} It is clear that
\[h(X:W^{*}(Y))\leq h(X:Y).\]
Let us prove the reverse inequality.  Fix $a\in W^{*}(X)$ with diffuse spectrum, and let $(A_{k})_{k=1}^{\infty}$ be microstates for $a.$ Let $L=\sup_{k}\|A_{k}\|_{\infty}$ and let $R_{X}\in [0,\infty)^{X},R_{Y}\in [0,\infty)^{Y},R_{W^{*}(Y)}\in [0,\infty)^{W^{*}(Y)}$ be cutoff parameters. Fix $\varepsilon>0,$ and let $F\subseteq X,G\subseteq W^{*}(Y)_{sa}$ be given finite sets. Given $m\in\NN,\gamma>0,$ Lemma \ref{L:extension} allows us to find $m'\in\NN,\gamma'>0$ and a finite $Y_{0}\subseteq Y$ so that for all $k$
\[\Gamma_{L\vee R_{X}\vee R_{Y}}(a,F:Y_{0};m',\gamma',k)\subseteq \Gamma_{L\vee R_{X}\vee R_{W^{*}(Y)}}(a,F:G;m,\gamma,k).\]
Since this holds for all $k,$ we have
\[ K_{2\varepsilon}(\Xi_{(A_{k}),R_{X}\vee R_{Y}}(F:Y_{0};m',\gamma'),\|\cdot\|_{2})\leq K_{\varepsilon}( \Xi_{(A_{k}),R_{X}\vee R_{W^{*}(Y)}}(F:G;m,\gamma),\|\cdot\|_{2}).\]
A fortiori,
\[K_{2\varepsilon}(\Xi_{(A_{k}),R_{X}\vee R_{Y}}(F:Y),\|\cdot\|_{2})\leq K_{2\varepsilon}(\Xi_{a_{k},R_{X}\vee R_{Y}}(F:Y_{0}),\|\cdot\|_{2})\leq K_{\varepsilon}( \Xi_{(A_{k})_{k=1}^{\infty},R_{X}\vee R_{W^{*}(Y)}}(F:G;m,\gamma),\|\cdot\|_{2}).\]
Taking the infimum over all $m,\gamma$ implies that
\[K_{2\varepsilon}(\Xi_{(A_{k}),R_{X}\vee R_{Y}}(F:Y),\|\cdot\|_{2})\leq  K_{\varepsilon}( \Xi_{(A_{k}),R_{X}\vee R_{W^{*}(Y)}}(F:G),\|\cdot\|_{2}).\]
Since this holds for all $G,$ we see that
\[K_{2\varepsilon}(\Xi_{(A_{k}),R_{X}\vee R_{Y}}(F:Y),\|\cdot\|_{2})\leq  K_{\varepsilon}( \Xi_{(A_{k}),R_{X}\vee R_{W^{*}(Y)}}(F:W^{*}(Y)),\|\cdot\|_{2}).\]
Now taking the supremum over $\varepsilon,F$ completes the proof.
\end{proof}

We now prove that $h(X:Y)$ only depends upon $W^{*}(X),W^{*}(Y),$ provided that $W^{*}(X)\subseteq W^{*}(Y).$ We remark that the proof is closely modeled on Jung's proof of Theorem 3.2 in \cite{JungSB}.

\begin{theorem}\label{T:indepedAppAkldghaklhgdkla} Let $(M,\tau)$ be a tracial von Neumann algebra, and $X,Y\subseteq M_{sa},j=1,2.$ Suppose that $W^{*}(X)\subseteq W^{*}(Y).$
Then
\[h(X:Y)=h(W^{*}(X):W^{*}(Y)).\]
\end{theorem}
\begin{proof} Choose an element $a\in W^{*}(X)_{sa}$ with diffuse spectrum and a sequence $(A_{k})_{k=1}^{\infty}$ of microstates for $a.$ By Lemma \ref{L:Yindependadjlgakjg}, it suffices to show that
\[h(X:W^{*}(Y))=h(W^{*}(X):W^{*}(Y)).\]
Set $N=W^{*}(Y).$ It is clear that
\[h(X:N)\leq h(W^{*}(X):N),\]
so it suffices to show that
\[h(W^{*}(X):N)\leq h(X:N).\]
 Let $R_{N}\in [0,\infty)^{N}$ be a cutoff parameter and fix a finite $F\subseteq W^{*}(X).$ It is enough to show that
 \[h(F:N)\leq h(X:N).\]
Let $\varepsilon>0,$ by Kaplansky's density theorem we may find a finite $X_{0}\subseteq X$ and polynomials $P_{b}\in \CC\ip{T_{x}:x\in X_{0}}$ for $b\in F$ so that
\[\|b-P_{b}(x:x\in X_{0})\|_{2}<\frac{\varepsilon}{\sqrt{|F|}},\]
\[\|P_{b}(x:x\in X_{0})\|_{\infty}\leq \|b\|_{\infty}.\]
For all $b\in F,$ choose $D_{b}(P)>1$ so that for any $k\in\NN$ and any $T,T'\in M_{k}(\CC)^{X_{0}}$with
\[\|T_{x}\|_{\infty},\|T'_{x}\|_{\infty}\leq R_{N,x}\mbox{ for all $x\in X_{0}$}\]
we have
\[\|P_{b}(T_{x}:x\in X_{0})-P_{b}(T_{x}':x\in X_{0})\|_{2}\leq D_{b}(P)\|T-T\|_{2}.\]
Set
\[D=\max_{b\in F}D_{b}(P).\]

	Suppose we are given $m\in\NN,\gamma>0$ and a finite $G\subseteq N_{sa}.$ Let
\[S\subseteq \Xi_{A_{k},R_{N}\vee R_{N}}(X_{0}:G;m,\gamma,k)\]
be $\frac{\varepsilon}{D\sqrt{|F|}}$-dense with respect to $\|\cdot\|_{2}$ and so that
\[|S|=K_{\frac{\varepsilon}{D\sqrt{|F|}}}(\Xi_{A_{k},R_{N}\vee R_{N}}(X_{0}:G;m,\gamma,k),\|\cdot\|_{2}).\]
Suppose that
\[C\in\Xi_{A_{k},R_{N}\vee R_{N}\vee R_{N}}(F:G,X_{0};m,\gamma,k)\]
and choose $B\in M_{k}(\CC)^{X_{0}}$ so that
\[(C,B)\in\Xi_{A_{k},R_{N}\vee R_{N}\vee R_{N}}(F,X_{0}:G;m,\gamma,k).\]
 If $m\in\NN$ is sufficiently large and $\gamma>0$ is sufficiently small, we have that
 \[\max_{b\in F}\|C_{b}-P_{b}(B_{x}:x\in X_{0})\|_{2}\leq \frac{2\varepsilon}{\sqrt{|F|}}.\]
 Now choose $B'\in S$ so that $\|B-B'\|_{2}\leq\frac{\varepsilon}{D\sqrt{|F|}},$
 we then have for any $b\in F$
 \[\|C_{b}-P_{b}(B'_{x}:x\in X_{0})\|_{2}\leq \frac{2\varepsilon}{\sqrt{|F|}}+\|P_{b}(B_{x}:x\in X_{0})-P_{b}(B'_{x}:x\in X_{0})\|_{2}\leq \frac{3\varepsilon}{\sqrt{|F|}}.\]
 Thus
 \[\Xi_{A_{k},R_{N}\vee R_{N}\vee R_{N}  }(F:G,X_{0};m,\gamma,k)\subseteq_{3\varepsilon,\|\cdot\|_{2}}\{(P_{a}(B_{x}:x\in X_{0}))_{a\in F}:B\in S\}.\]
 Since this holds for all $k$ we have:
\[K_{6\varepsilon}(\Xi_{(A_{k})_{k=1}^{\infty},R_{N}\vee R_{N}\vee R_{N} }(F:G,X_{0};m,\gamma),\|\cdot\|_{2})\leq K_{\frac{\varepsilon}{D\sqrt{|F|}}}(\Xi_{(A_{k})_{k=1}^{\infty},R_{N}\vee R_{N}}(X_{0}:G;m,\gamma),\|\cdot\|_{2}).\]
 A fortiori,
 \begin{align*}
K_{6\varepsilon}(\Xi_{(A_{k})_{k=1}^{\infty},R_{N}\vee R_{N}}(F:N),\|\cdot\|_{2})&\leq K_{6\varepsilon}(\Xi_{(A_{k})_{k=1}^{\infty},R_{N}\vee R_{N}}(F:G,X_{0}),\|\cdot\|_{2})\\
 &\leq  K_{\frac{\varepsilon}{D\sqrt{|F|}}}(\Xi_{(A_{k})_{k=1}^{\infty},R_{N}\vee R_{N}}(X_{0}:G;m,\gamma),\|\cdot\|_{2}).
 \end{align*}
 Since $D$ does not depend upon $m,\gamma,G$ we can let $m\to\infty,\gamma\to 0$ and take the infimum over all $G$ to see that:
 \[K_{6\varepsilon}(\Xi_{(A_{k})_{k=1}^{\infty},R_{N}\vee R_{N}}(F:N),\|\cdot\|_{2})\leq K_{\frac{\varepsilon}{D\sqrt{|F|}}}(\Xi_{(A_{k})_{k=1}^{\infty},R_{N}\vee R_{N}}(X_{0}:N),\|\cdot\|_{2})\leq h(X_{0}:N)\leq h(X:N).\]
Now taking the supremum over all $\varepsilon,F$ we see that
\[h(M:N)\leq h(X:N).\]

\end{proof}

We have now completed our proof that $h(X:Y)$ only depends upon the von Neumann algebras generated by $X,Y,$ provided $W^{*}(X)\subseteq W^{*}(Y).$ We mention that is clear from Theorem \ref{T:indepedAppAkldghaklhgdkla} that if $M$ is finitely generated, then $h(M)<\infty$ if and only if $M$ is strongly $1$-bounded in the sense of Jung in \cite{JungSB}. In particular, if $M=W^{*}(F)$ for $F\subseteq M_{sa}$ finite, and $\delta_{0}(F)>1,$ then $h(M)=\infty.$ We  turn to  other important properties of $h(N:P)$ for von Neumann subalgebras $N\subseteq P$ of a tracial von Neumann algebra $M.$
\begin{lemma}\label{L:increasingunions} Let $(M,\tau)$ be a tracial von Neumann algebra, and let $N\subseteq P$ be von Neumann subalgebras of $M.$ Suppose that $M$ is diffuse. Suppose that $N_{\alpha}$ is an increasing net of diffuse von Neumann subalgebras of $N$ with
\[N=\overline{\bigcup_{\alpha}N_{\alpha}}^{SOT},\]
then
\[h(N:P)=\lim_{\alpha}h(N_{\alpha}:P).\]
\end{lemma}

\begin{proof}
This is easy from the fact that
\[h(X:Y)=h(W^{*}(X):P),\]
 for any $X\subseteq M_{sa},$ (e.g. take $X=\bigcup_{\alpha}(N_{\alpha})_{sa}$).
\end{proof}

\begin{cor}\label{C:increasingunions} Let $(M,\tau)$ be a tracial von Neumann algebra. Let $M_{\alpha}$ be an increasing net of diffuse von Neumann subalgebras  with
\[M=\overline{\bigcup_{\alpha}M_{\alpha}}^{SOT}.\]
Then
\[h(M)\leq \liminf_{\alpha}h(M_{\alpha}).\]
\end{cor}
\begin{proof} From Lemma \ref{L:increasingunions} it follows that
\[h(M)=h(M:M)=\lim_{\alpha}h(M_{\alpha}:M)\leq \liminf_{\alpha}h(M_{\alpha}:M_{\alpha})=\liminf_{\alpha}h(M_{\alpha}).\]

\end{proof}

We remark that when $h(M_{\alpha})=0$ for all $\alpha,$ Corollary \ref{C:increasingunions} was obtained by Hadwin-Li in \cite{HadwinLi}.

\begin{lemma}\label{L:diffuseintersection} Let $N_{j},j=1,2$ be von Neumann subalgebras of a tracial von Neumann algebra $(M,\tau)$ with $N_{1}\cap N_{2}$ diffuse. Then
\[h(N_{1}\vee N_{2}:M)\leq h(N_{1}:M)+h(N_{2}:M).\]
In particular,
\[h(N_{1}\vee N_{2})\leq h(N_{1})+h(N_{2}).\]

\end{lemma}
\begin{proof} The general case follows easily from the same arguments in the finitely-generated case due to Jung (see Corollary 4.2 in \cite{JungSB}).
For the ``in particular'' part, we have
\begin{align*}
h(N_{1}\vee N_{2})=h(N_{1}\vee N_{2}:N_{1}\vee N_{2})&\leq h(N_{1}:N_{1}\vee N_{2})+h(N_{2}:N_{1}\vee  N_{2})\\
&\leq h(N_{1}:N_{1})+h(N_{2}:N_{2})\\
&=h(N_{1})+h(N_{2}).
\end{align*}

\end{proof}

Lastly we state an inequality for compressions, as well as one for direct sums.

\begin{proposition}\label{P:compressionformula}

(i): Let $(M_{j},\tau_{j})_{j=1}^{\infty}$ be diffuse tracial von Neumann algebras, and $\mu_{j},j=1,2,\dots$ be such that
\[\sum_{j=1}^{\infty}\mu_{j}=1.\]
Define $\tau$ on $M=\bigoplus_{j=1}^{\infty}M_{j}$ by
\[\tau(x)=\sum_{j=1}^{\infty}\mu_{j}\tau_{j}(x_{j}).\]
Then
\[h(M)\leq \sum_{j=1}^{\infty}\mu_{j}^{2}h(M_{j}),\]
with the convention that if one of the terms on the right-hand side is $-\infty,$ then the sum is $-\infty.$

(ii): Let $M$ be $\textrm{II}_{1}$-factor with canonical trace $\tau.$ Let $p\in M$ be a nonzero orthogonal projection. Define $\tau_{p}$ on $pMp$ by $\tau_{p}(x)=\frac{\tau(x)}{\tau(p)}.$ Then
\[h(pMp)\leq \frac{1}{\tau(p)^{2}}h(M).\]

\end{proposition}

\begin{proof}

(i): If one of $h(M_{j},\tau_{j})=-\infty,$ then $M_{j}$ does not embed into an ultrapower of the hyperfinite $\textrm{II}_{1}$-factor, and hence neither does $M$ and $h(M)=-\infty.$ So we will assume that $h(M_{j})\geq 0$ for all $j.$  Let us first handle the case of two algebras. So assume we are given $(M_{j},\tau_{j}),j=1,2$ tracial von Neumann algebras and a tracial state $\tau\colon M_{1}\oplus M_{2}\to\CC.$
Let $R_{j}\in [0,\infty)^{M_{j}}$ be cutoff parameters and define $R\in [0,\infty)^{M}$ by
\[R_{(a,b)}=\max(R_{1,a},R_{2,b}).\]
Let $z_{1}=(1,0),z_{2}=(0,1).$ Fix $a_{j}\in (M_{j})_{sa}$ with diffuse spectrum, set $a=(a_{1},a_{2})$ and observe that $a$ has diffuse spectrum.
Fix microstates $A_{k}^{(j)}$ for $a_{j},j=1,2.$ We may find a sequence of microstates $(A_{k})_{k=1}^{\infty}$ for $a$ so that there  are sequences $(l^{(j)}_{k})_{k=1}^{\infty},j=1,2$ of integers with
\[l^{(1)}_{k}+l^{(2)}_{k}=k,\]
\[\frac{l^{(j)}_{k}}{k}\to \tau(z_{j}),j=1,2,\]
\[A_{k}=\begin{bmatrix}
A_{l_{k}^{(1)}}^{(1)}&0\\
0&A_{l^{(2)}_{k}}^{(2)}
\end{bmatrix}.\]
Fix a finite $F\subseteq M_{sa},$ and $\varepsilon>0.$ Let $F_{j},G_{j}\subseteq (M_{j})_{sa},j=1,2$ be finite sets so that
\[F\subseteq F_{1}\oplus F_{2}\]
and let $m'\in\NN,\gamma'>0.$ It is not hard to see from our choices we may choose a $m\in \NN,\gamma>0$ so that
\[\Xi_{A_{k},R\vee R}(F:G_{1}\oplus G_{2},m,\gamma,k)\subseteq_{\varepsilon,\|\cdot\|_{2}}\Xi_{A_{l^{(1)}_{k}}^{(1)},R_{1}\vee R_{1}}(F_{1}:G_{1};m',\gamma',l^{(k)}_{1})\oplus \Xi_{A_{l^{(2)}_{k}}^{(2)},R_{2}\vee R_{2}}(F_{2}:G_{2};m',\gamma',l^{(k)}_{2}).\]
Thus we have
\begin{align*}
K_{6\varepsilon}(\Xi_{(A_{k}),R\vee R}(F:G_{1}\oplus G_{2};m,\gamma),\|\cdot\|_{2})&\leq \tau(z_{1})^{2}K_{\varespilon}(\Xi_{(A_{k}^{(1)}),R_{1}\vee R_{1}}(F_{1}:G_{1};m',\gamma'),\|\cdot\|_{2})\\
&+\tau(z_{2})^{2}K_{\varepsilon}(\Xi_{(A_{k}^{(2)}),R_{2}\vee R_{2}}(F_{2}:G_{2};m',\gamma'),\|\cdot\|_{2}).
\end{align*}
A fortiori,
\begin{align*}
K_{6\varepsilon}(\Xi_{(A_{k}),R\vee R}(F:G_{1}\oplus G_{2}),\|\cdot\|_{2})&\leq \tau(z_{1})^{2}K_{\varespilon}(\Xi_{(A_{k}^{(1)}),R_{1}\vee R_{1}}(F_{1}:G_{1};m',\gamma'),\|\cdot\|_{2})\\
&+\tau(z_{2})^{2}K_{\varepsilon}(\Xi_{(A_{k}^{(2)}),R_{2}\vee R_{2}}(F_{2}:G_{2};m',\gamma'),\|\cdot\|_{2}).
\end{align*}
and taking the infimum over $m',\gamma',G_{1},G_{2}$ we see that
\begin{align*}
K_{6\varepsilon}(\Xi_{(A_{k}),R\vee R}(F:M_{1}\oplus M_{2}),\|\cdot\|_{2})&\leq\tau(z_{1})^{2}K_{\varespilon}(\Xi_{(A_{k}^{(1)}),R_{1}\vee R_{1}}(F_{1}:M_{1}),\|\cdot\|_{2})\\
&+\tau(z_{2})^{2}K_{\varepsilon}(\Xi_{(A_{k}^{(2)}),R_{2}\vee R_{2}}(F_{2}:M_{2}),\|\cdot\|_{2}).
\end{align*}
Taking the supremum over $\varespilon$ shows that
\[h(F:M_{1}\oplus M_{2})\leq \tau(z_{1})^{2}h(M_{1})+\tau(z_{2})^{2}h(M_{2})\]
and now taking the supremum over $F$ proves the case of two algebras.

Now let us handle the general case. Fix diffuse, abelian, von Neumann subalgebras $A_{j}\subseteq M_{j}.$ Let
\[M_{\leq N}=\left(\bigoplus_{j=1}^{N}M_{j}\right)\oplus \bigoplus_{j=N+1}^{\infty}A_{j},\]
By Lemma \ref{L:increasingunions}
\[h(M:M)=\sup_{N}h(M_{\leq N}:M).\]
Using that each $M_{j}$ embeds into an ultrapower of the hyperfinite $\textrm{II}_{1}$-factor it is not hard to argue that
\[h(M_{\leq N}:M)=h(M_{\leq N}:M_{\leq N}).\]
By the case of two algebras and induction
\[h(M_{\leq N}:M_{\leq N})\leq \sum_{j=1}^{N}\tau(p_{j})^{2}h(M_{j}),\]
(here we are using that the $1$-bounded entropy of any abelian von Neumann algebra is zero). Taking the supremum over $n$ completes the proof.

(ii):  Again we may reduce to the case that $M$ embeds into an ultrapower of $\R.$  Let $a\in M_{sa}$ be an element with diffuse spectrum, since the isomorphism class of $pMp$ only depends upon the trace of $p,$ we may assume that $p$ is a spectral projection of $a.$  Let $R\in [0,\infty)^{M}$ be a cutoff parameter. Let $n$ be the smallest integer so that $n\tau(p)\geq 1.$ Fix partial isometries $v_{1},\dots,v_{n},v_{n+1}\in M$ so that
\[v_{j}^{*}v_{j}=p,j=1,\dots,n\]
\[v_{n+1}^{*}v_{n+1}\leq p\]
\[1=\sum_{j=1}^{n+1}v_{j}v_{j}^{*}.\]
Fix a sequence $l_{k}$ of integers so that
\[\frac{l_{k}}{k}\to \tau(v_{n+1}^{*}v_{n+1}).\]

	Fix a sequence of microstates $(A_{k,p})$ for $pa.$ We may assume that there are projections $E_{k}\in M_{k}(\CC)$ with $\tr(E_{k})=\frac{l_{k}}{k},$ so that $[E_{k},A_{k,p}]=0$ and so that
\[A_{nk+l_{k}}=\begin{bmatrix}
A_{k,p}&0&0&\dots &0&0\\
0& A_{k,p}&0&\dots&0&0\\
\vdots&\dots&\ddots&\dots&\dots&\vdots\\
0&0&0&\dots& A_{k,p}&0\\
0&0&0&\dots&0& E_{k}A_{k,p}
\end{bmatrix}\]
is a sequence of microstates for $a.$ Here we are abusing notation and regarding $E_{k}A_{k,p}$ as an element in $M_{l_{k}}(\CC).$ Let $Q_{k}\in M_{nk+l_{k}}(\CC)$ be the orthogonal projection onto the first $k$ coordinates.  Let $\varepsilon>0,$ and fix a finite $F\subseteq pMp.$ Set
\[\widetilde{F}=\{v_{i}xv_{j}^{*}:x\in F,j=1,\dots,n\}\cup\{v_{1},\dots,v_{n}\}.\]
Let $\widetilde{G}\subseteq M_{sa}$ be a given finite set, $\widetilde{m}\in\NN,\widetilde{\gamma}>0.$ It is not hard to show that there is a finite $G\subseteq (pMp)_{sa}$ a $m\in\NN,\gamma>0$ so that for all $k\in\NN$
\[\Xi_{A_{k,p},R\vee R}(F:G;m,\gamma,k)\subseteq_{\varepsilon}Q_{k}\Xi_{A_{k},R\vee R}(\widetidle{F}:\widetilde{G};\widetilde{m},\widetilde{\gamma},nk+l_{k})Q_{k}.\]
Again we are abusing notation by regarding $Q_{k}M_{nk+l_{k}}(\CC)Q_{k}$ as $M_{k}(\CC).$ We may now argue as in $(i)$ to complete the proof.

\end{proof}

We remark that one can define a \emph{lower} $1$-bounded entropy by taking a limit infimum instead of a limit supremum. It is not hard to argue that if the lower $1$-bounded entropy is the upper $1$-bounded entropy then we have equality in $(i).$

We end this section by clarifying the equivalence of finiteness of $1$-bounded entropy and being strongly $1$-bounded in the sense of Jung. It turns out that, assuming the given generating set is a nonamenability set, our methods are robust enough to remove the assumption of having an element with finite free entropy from Jung's formulation of being strongly $1$-bounded.
\begin{definition}\emph{Let $(M,\tau)$ be a tracial von Neumann algebra. A finite subset $F\subseteq M$ is a} nonamenability set for $M$ \emph{if there is a constant $K>0$ so that}
\[\|\xi\|_{2}\leq K\sum_{x\in F}\|x\xi-\xi x\|_{2}\mbox{ \emph{for all $\xi\in L^{2}(M)\otimes L^{2}(M)$.}}\]

\end{definition}

Connes showed in \cite{Connes} that every nonamenable $\textrm{II}_{1}$-factor contains a nonamenability set.

We will need a preliminary lemma.
For a tracial von Neumann algebra $(M,\tau)$ and $x\in M_{n}(M)_{sa},$ we let $\mu_{x}$ be its spectral measure with respect to $\Tr\otimes \tau$ defined by
\[\mu_{x}(E)=\Tr\otimes \tau(\chi_{E}(x)).\]
If $I,J$ are finite sets and $M$ is a von Neumann algebra, we let $M_{I,J}(M)$ be the set of all $I\times J$ matrices over $M.$
\begin{lemma}\label{L:orbitseparateappendix} Let $(M,\tau)$ be a nonamenable tracial von Neumann algebra and let $F\subseteq M_{sa}$ be finite and so that $W^{*}(F)=M.$ Suppose that $F$ is a nonamenability set for $M.$ Let $R\in [0,\infty)^{F}$ be a cutoff parameter. Then there exists $D>0$ so that for every $\varepsilon>0$
\[\inf_{m\in \NN,\gamma>0}\limsup_{k\to\infty}\sup_{B\in \Gamma_{R}(F;m,\gamma,k)}\frac{1}{k^{2}}\log K_{D\varepsilon}(\{U\in \mathcal{U}(k):\|[U,B]\|_{2}<\varepsilon\},\|\cdot\|_{2})=0.\]

\end{lemma}

\begin{proof}
Define $\Delta\colon L^{2}(M)\otimes L^{2}(M^{op})\to [L^{2}(M)\otimes L^{2}(M^{op})]^{\oplus F}$ by
\[\Delta(\xi)=(b\xi-\xi b)_{b\in F}\]
and observe that $\Delta$ may be regarded as an element of $M_{F,\{1\}}(M\otimes M^{op}).$ Since $F$ is a nonamenability set, we may choose a $\kappa>0$ so that $\mu_{|\Delta|}([0,\kappa])=0.$
Given $B\in M_{k}(\CC)_{sa}^{F}$ (for some $k\in \NN$) we define $\Delta_{B}\colon L^{2}(M_{k}(\CC),\tr)\to L^{2}(M_{k}(\CC),\tr)^{\oplus F}$ by
\[\Delta_{B}(A)=(B_{b}A-AB_{b})_{b\in F}.\]
 Fix $\alpha>0.$ We may choose a $m_{0}\in \NN$ and a $\gamma_{0}>0$ so that for every $k\in \NN$ and every $B\in \Gamma_{R}(F;m_{0},\gamma_{0},k)$ we have $\mu_{|\Delta_{B}|}([0,\kappa])\leq \alpha.$  Suppose that $U\in \mathcal{U}(k),$ $B\in \Gamma_{R}(F;m_{0},\gamma_{0},k)$ and that $\|UB-BU\|_{2}<\varespilon.$ We then have
\[\|\chi_{[0,\kappa]}(|\Delta_{B}|)U-U\|_{2}\leq \frac{\varepsilon}{\kappa}.\]
So
\[\{U\in \mathcal{U}(k):\|[U,B]\|_{2}<\varepsilon\}\subseteq_{\frac{\varepsilon}{\kappa}}\chi_{[0,\kappa]}(|\Delta_{B}|)\Ball(L^{2}(M_{k}(\CC),\tr),\|\cdot\|_{2}).\]
Thus
\[K_{2\left(\frac{1}{\kappa}+1\right)\varepsilon}(\{U\in \mathcal{U}(k):\|[U,B]\|_{2}<\varepsilon\},\|\cdot\|_{2})\leq K_{\varepsilon}(\chi_{[0,\kappa]}(|\Delta_{B}|)\Ball(L^{2}(M_{k}(\CC),\tr),\|\cdot\|_{2}),\|\cdot\|_{2}).\]
By a volume-packing argument, the right hand side of the above inequality is at most $\left(\frac{3+\varepsilon}{\varepsilon}\right)^{\alpha k^{2}}.$ Thus we have
\[\limsup_{k\to\infty}\sup_{B\in \Gamma_{R}(F;m_{0},\gamma_{0},k)}K_{2\left(\frac{1}{\kappa}+1\right)\varepsilon}(\{U\in \mathcal{U}(k):\|[U,B]\|_{2}<\varepsilon\},\|\cdot\|_{2})\leq \alpha\log\left(\frac{3+\varepsilon}{\varepsilon}\right).\]
So
\[\inf_{m\in \NN,\gamma>0}\limsup_{k\to\infty}\sup_{B\in \Gamma_{R}(F;m,\gamma,k)}\frac{1}{k^{2}}\log K_{2\left(\frac{1}{\kappa}+1\right)\varepsilon}(\{U\in \mathcal{U}(k):\|[U,B]\|_{2}<\varepsilon\},\|\cdot\|_{2})\leq  \alpha\log\left(\frac{3+\varepsilon}{\varepsilon}\right).\]
Since $\alpha>0$ was arbitrary we can let $\alpha\to 0$ and take $D=2\left(\frac{1}{\kappa}+1\right)$ to complete the proof.

\end{proof}

We now relate finiteness of $1$-bounded entropy to being strongly $1$-bounded as defined by Jung. A byproduct of our techniques is that we are able to replace the assumption of having an element with finite free entropy in our generating set from the definition of strongly $1$-bounded with the assumption that our generating set is a nonamenability set. Unfortunately, we are unable to unconditionally show that if  $F$ is $1$-bounded (we will define what it means to be $1$-bounded shortly), then $W^{*}(F)$ is strongly $1$-bounded. However, we remark that in many natural examples of nonamenable von Neumann algebras (e.g. group von Neumann algebras), the ``obvious" finite set of generators for $M$ is a nonamenability set. Of course, as shown in \cite{JungSB}, any amenable von Neumann algebra is strongly $1$-bounded, so the question of whether we can unconditionally remove the assumption of having an element of finite free entropy from the definition of strongly $1$-bounded reduces to technical issues of existence (or nonexistence) of a nonamenability set. We use some of the same notation as in \cite{JungSB}. Namely, if $(M,\tau)$ is a tracial von Neumann algebra and $F\subseteq M_{sa},$ we set for $\varepsilon>0,$ and $R>\max_{x\in F}\|x\|,$
\[K_{\varepsilon}(F)=\inf_{\substack{m\in \NN,\\ \gamma>0}}\limsup_{k\to\infty}\frac{1}{k^{2}}\log K_{\varepsilon}(\Gamma_{R}(F;m,\gamma,k),\|\cdot\|_{2}).\]
Here we are identifying $R$ with the cutoff parameter in $[0,\infty)^{F}$ which is $R$ in every coordinate.
Recall that $F$ is \emph{1-bounded} (as defined by Jung in \cite{JungSB}) if there is a $C>0$ so that
\[K_{\varepsilon}(F)\leq C+\log(1/\varepsilon).\]
Given $a\in M_{sa}$ we use $\chi(a)$ for the free entropy of $a$ as defined by Voiculescu in \cite{FreeEntropyDimensionII}.
\begin{proposition}\label{P:equivstrongly1bddappendix} Let $(M,\tau)$ be a diffuse tracial von Neumann algebra and let $F\subseteq M_{sa}$ be  finite and such that $W^{*}(F)=M.$ Consider the following conditions:
\begin{enumerate}
\item $M$ is strongly $1$-bounded, \label{I:strongly1bdd}
\item $h(M)<\infty,$\label{I:1bddent}
\item $F$ is  $1$-bounded. \label{I:Fbdd}
\end{enumerate}
Then $(\ref{I:strongly1bdd})$ and $(\ref{I:1bddent})$ are equivalent and imply  (\ref{I:Fbdd}). If $F$ is a nonamenability set, then (\ref{I:Fbdd}) is equivalent to (\ref{I:1bddent}) and (\ref{I:strongly1bdd}).

\end{proposition}

\begin{proof}Since $M$ is  diffuse, we can find an $a\in M_{sa}$ with diffuse spectrum and with $\chi(a)>-\infty.$ Let $(A_{k})_{k=1}^{\infty}$ be a sequence of microstates for $a.$ Fix
\[R>\max(\|a\|,\max_{b\in F}\|b\|).\]

(\ref{I:strongly1bdd}) implies (\ref{I:1bddent}):
Suppose $M$ is strongly $1$-bounded. By Theorem A.9, we have
\[h(M)=h(F\cup\{a\})=\sup_{t>0}K_{t}(\Xi_{(A_{k})_{k}}(F\cup\{a\})).\]
Since $M$ is strongly $1$-bounded, Lemma 2.1 of \cite{JungSB} implies that the right most expression in the above equalities is finite and thus $h(M)<\infty.$

(\ref{I:1bddent}) implies (\ref{I:strongly1bdd}):
 Suppose that $h(M)<\infty.$ Lemma 2.2 of \cite{JungSB} implies that there is a $C>0$ so that for all sufficiently small $\varespilon>0,$
\[K_{\varepsilon}(\{a\}\cup F)\leq C+\log(1/\varepsilon)+h(F),\]
 and thus $M$ is strongly $1$-bounded.

(\ref{I:strongly1bdd}) implies  (\ref{I:Fbdd}): This is a rephrasing of Theorem 3.2 of \cite{JungSB}.

(\ref{I:Fbdd}) implies (\ref{I:1bddent}) when $F$ is a nonamenability set: We will use the orbital version of $1$-bounded entropy. By a result of S. Szarek in \cite{Szarek} we may choose an $A>0$ so that for all $\delta>0$ we have
\[K_{\delta}(\mathcal{U}(k),\|\cdot\|_{2})\geq \left(\frac{A}{\delta}\right)^{k^{2}}.\]
 Since $F$ is  $1$-bounded, we may choose a $C>0$ so that for all $0<\varepsilon<1$
\[K_{\varepsilon}(\Gamma_{R}(F))\leq C+\log(1/\varepsilon).\]
Let $D>0$ be as in the preceding Lemma for this $F.$ Fix $0<\varpesilon<1$ and for $k,m\in \NN, \gamma>0$ let
\[\alpha(m,\gamma,k)=\sup_{B\in \Gamma_{R}(F;m,\gamma,k)}\frac{1}{k^{2}}\log K_{2D\varepsilon}(\{U\in \mathcal{U}(k):\|[U,B]\|_{2}<2\varepsilon\},\|\cdot\|_{2}),\]
\[\alpha(m,\gamma)=\limsup_{k\to\infty}\alpha(m,\gamma,k).\]
Let $S\subseteq\Gamma_{R}(F;m,\gamma,k)$ be a maximal $\varespilon$-orbit separated subset. For all $B\in S,\beta>0$ let
\[\Omega_{B}(\beta)=\{U\in \mathcal{U}(k):\|[U,B]\|_{2}<\beta\}.\]
For each $B\in S,$ let $T_{B}\subseteq \mathcal{U}(k)$ be a maximal subset subject to the condition that for distinct $V,W\in T_{B}$ we have $V\Omega_{B}(\varepsilon)\cap W\Omega_{B}(\varepsilon)=\varnothing.$ By maximality, we have
\[\mathcal{U}(k)\subseteq \bigcup_{V\in T_{B}}V\Omega_{B}(\varepsilon)\Omega_{B}(\varepsilon)^{*}\subseteq \bigcup_{V\in T_{B}}V\Omega_{B}(2\varepsilon).\]
Thus
\[\left(\frac{A}{2D\varepsilon}\right)^{k^{2}}\leq K_{2D\varpesilon}(\mathcal{U}(k),\|\cdot\|_{2})\leq |T_{B}|K_{2D\varepsilon}(\Omega_{B}(2\varepsilon),\|\cdot\|_{2}).\]
So
\[|T_{B}|\geq \left(\exp(-\alpha(m,\gamma,k))\frac{A}{2D\varepsilon}\right)^{k^{2}}.\]
Now consider $T=\{U^{*}BU:B\in S, U\in T_{B}\},$ we claim that $T$ is $\varepsilon$-separated. Suppose that $B_{j}\in S,j=1,2$ and $U_{j}\in T_{B_{j}},j=1,2$ have
\[\varepsilon> \|U_{1}^{*}B_{1}U_{1}-U_{2}^{*}B_{2}U_{2}\|_{2}.\]
Then
\[\varepsilon> \|U_{2}U_{1}^{*}B_{1}U_{1}U_{2}^{*}-B_{2}\|_{2},\]
by choice of $S$ this implies that $B_{1}=B_{2}.$ We then have $U_{2}^{*}U_{1}\in \Omega_{B_{1}}(\varepsilon)$ and so $U_{1}\Omega_{B_{1}}\cap U_{2}\Omega_{B_{1}}(\varpesilon)\ne \varnothing$ and this implies that $U_{1}=U_{2}.$ Since $T$ is $\varpesilon$-separated, it is easy to see that
\[K_{\varepsilon/2}(\Gamma_{R}(F;m,\gamma,k),\|\cdot\|_{2})\geq |T|.\]
Thus
\begin{align*}
K_{\varepsilon/2}(\Gamma_{R}(F;m,\gamma,k),\|\cdot\|_{2})\geq \sum_{B\in S}|T_{B}|&\geq |S|\left(\exp(-\alpha(m,\gamma,k))\frac{A}{2D\varepsilon}\right)^{k^{2}}\\
&\geq  K_{\varpesilon}^{\mathcal{O}}(\Gamma_{R}(F;m,\gamma,k),\|\cdot\|_{2})\left(\exp(-\alpha(m,\gamma,k))\frac{A}{2D\varepsilon}\right)^{k^{2}} .
\end{align*}
Taking $\frac{1}{k^{2}}\log$ of both sides and letting $k\to\infty$ implies that
\[K_{\varepsilon/2}(\Gamma_{R}(F;m,\gamma))\geq K_{\varepsilon}^{\mathcal{O}}(F;m,\gamma)-\alpha(m,\gamma)+\log\left(\frac{A}{2D}\right)+\log(1/\varepsilon).\]
So
\[K_{\varepsilon}^{\mathcal{O}}(F;m,\gamma)-\alpha(m,\gamma)+\log\left(\frac{A}{2D}\right)+\log(1/\varepsilon)\leq K_{\varpesilon/2}(\Gamma_{R}(F;m,\gamma)\leq C+\log(2)+\log(1/\varepsilon).\]
By Lemma \ref{L:orbitseparateappendix} we have that $\alpha(m,\gamma)\to 0$ as $m\to\infty$ and $\gamma\to 0,$ so
\[K_{\varpesilon}^{\mathcal{O}}(F)\leq C+\log(2)-\log\left(\frac{A}{2D}\right).\]
Now taking the supremum over all $\varepsilon>0$ we find that $h(F)<\infty.$

\end{proof}
\section{$1$-Bounded Entropy with Respect to Unbounded Generators}\label{S:unbounded}
Our goal in this appendix is to show that
\[h((x_{i})_{i\in I}:M,\mbox{meas})=h(W^{*}(x_{i}:i\in I):M),\]
whenever the $x_{i}$ are unbounded operators affiliated to $(M,\tau).$ See Definitions \ref{D:measuremicrostates} and Definition \ref{D:1bddentrunbddop1} to recall the necessary definitions. The following two lemmas will be useful in the proof of the above equality.
\begin{lemma}\label{L:projcommlasdjg} Let $T,Y\in M_{k}(\CC),$ and $\varpesilon>0.$ Suppose that $Q\in M_{k}(\CC)$ is a projection with
\[\|Q(T-Y)\|_{\infty}\leq\varepsilon,\]
\[\tr(1-Q)\leq \varepsilon.\]
Then there exists a projection $P\in M_{k}(\CC)$ so that
\[P\leq \chi_{[0,\|Y\|_{\infty}+\varepsilon]}(|T^{*}|),\]
\[\tr(1-P)\leq 2\varepsilon,\]
\[\|P(T-Y)\|_{\infty}\leq \varepsilon.\]

\end{lemma}

\begin{proof} Let $R=\|Y\|_{\infty}$ and set $P=\chi_{[0,R+\varepsilon]}(|T^{*}|)\wedge Q.$
We claim that
\begin{equation}\label{E:wedgezeroupperboundB}
(1-\chi_{[0,R+\varpesilon]}(|T^{*}|))\wedge Q=0.
\end{equation}
Indeed, suppose that
\[\xi\in \chi_{(R+\varpesilon,\infty)}(|T^{*}|)(\CC^{k})\cap Q(\CC^{k}),\]
but that $\xi\ne 0.$ Let $T^{*}=U|T^{*}|$ be the polar decomposition. Then
\begin{align*}
(R+\varepsilon)\|\xi\|_{2}^{2}&<\ip{|T^{*}|\xi,\xi}\\
&=\ip{U^{*}T^{*}\xi,\xi}\\
&=\ip{T^{*}Q\xi,U\xi}\\
&\leq \varepsilon\|\xi\|_{2}^{2}+\ip{Y^{*}Q\xi,U\xi}\\
&\leq (R+\varepsilon)\|\xi\|_{2}^{2}
\end{align*}
and this is a contradiction. Thus we have proved $(\ref{E:wedgezeroupperboundB}).$ It is not hard to show that $(\ref{E:wedgezeroupperboundB})$ implies that
\[\tr(1-\chi_{[0,R+\varpesilon]}(|T^{*}|))\leq \tr(1-Q)\leq \varpesilon,\]
so $\tr(1-P)\leq 2\varpesilon.$
Additionally we have $\|P(T-Y)\|_{\infty}\leq \|Q(T-Y)\|_{\infty}<\varepsilon.$

\end{proof}

We also need a lemma that allows one to produce microstates for unbounded operators from microstates from bounded operators.

\begin{lemma}\label{L:measuremicrostateskladhgaklj} Let $(M,\tau)$ be a tracial von Neumann algebra and let $R_{M}\in [0,\infty)^{M_{sa}}$ be a cutoff parameter. Let $J$ be a finite set and $(y_{j})_{j\in J}$ be a collection of measurable self-adjoint operators affiliated with $M.$ Fix finite $F\subseteq C_{b}(\RR),G\subseteq M_{sa},$ an $m\in \NN$ and a $\gamma,\eta>0.$ Fix a $R>0$ with
\[\max_{j\in J}\tau(\chi_{[R,\infty)}(|y_{j}|))<\eta,\mbox{ and } R>\max_{x\in G}R_{M,x}.\]
Then there exists a $R_{0}'>0$ with the following property. For all $R'>R_{0}'$ and  $\psi\in C_{c}(\RR,\RR)$ with $\|\psi\|_{C_{b}(\RR)}\leq R'$ and $\psi(t)=t$ for all $|t|\leq R',$ there is an $m'\in \NN$ and a $\gamma'>0$ so that for every  \[(A,X)\in \Gamma_{R_{M}}(G,(\phi(y_{j}))_{j\in J,\phi\in F\cup\{\psi\}};m',\gamma',k)\]
 we have
\[(A,(X_{\psi,i})_{i\in I})\in \Gamma_{R}^{\eta}(G,(y_{i})_{i\in I};F,m,\gamma,k).\]

\end{lemma}

\begin{proof} Let
\[\Psi_{R'}=\{\psi\in C_{c}(\RR,\RR):\|\psi\|_{C_{b}(\RR)}\leq R',\psi(t)=t \mbox{ if } |t|\leq R'\}.\]
Choose a $\beta\in C_{c}(\RR,\RR)$ with $\chi_{[R,\infty)}\leq \beta\leq 1$ and with $\max_{j\in J}\tau(\beta(|y_{j}|))<\eta.$
It is easy to see that
\[\lim_{R'\to\infty}\sup_{\psi\in \Psi_{R'}}\max_{j\in J}\tau(\beta(|\psi(y_{j})|))<\eta,\]
\[\lim_{R'\to\infty}\sup_{\psi\in \Psi_{R'}}|\tau(P(x,\phi(y_{j}):x\in G,\phi\in F,j\in J))-\tau(P(x,\phi(\psi(y_{j})):x\in G,\phi\in F,j\in J))|=0,\]
for all $P\in\CC\ip{S_{x},T_{\phi,j}:x\in G,\phi\in F,j\in J}.$  So we may choose a $R_{0}'>0$ so that for all $R\geq R_{0}'$ and for all $\psi\in \Psi_{R'}$ we have
\[\max_{j\in J}\tau(\beta(|\psi(y_{j})|))<\eta,\]
\[|\tau(P(x,\phi(y_{j}):x\in G,\phi\in F,j\in J))-\tau(P(a_{j},\phi(\psi(y_{j})):x\in G,\phi\in F,j\in J))|<\gamma/2\]
for all monomials $P\in \CC\ip{S_{x},T_{\phi,j}:x\in G,\phi\in F,j\in J}$ of degree at most $m.$

	Now fix $R'>R_{0}'$ and $\psi\in \Psi_{R'}.$ We may choose an $m'\in \NN,\gamma'>0$ so that if
	\[(A,X)\in \Gamma_{R_{M}}(G,(\phi(y_{j}))_{j\in J,\phi\in F\cup\{\psi\}};m',\gamma',k),\] then
\[\max_{j\in J}\tr(\beta(|X_{\psi,j}|))<\eta\]
and
\[|\tau(P(x,\phi(\psi(y_{j}))):x\in G,\phi\in F,j\in J))-\tr(P(A_{x},\phi(X_{\psi,j})):x\in G,\phi\in F,j\in J)|<\frac{\gamma}{2}\]
for all monomials $P\in \CC\ip{S_{x},T_{\phi,j}:x\in G,\phi\in F,j\in J}$ of degree at most $m.$
It follows that $(A,(X_{\psi,j})_{j\in J})\in \Gamma_{R}^{\eta}(G,(y_{j})_{j\in J};F,m,\gamma,k).$

\end{proof}

\begin{proposition} Let $(M,\tau)$ be a tracial von Neumann and $N$ a diffuse von Neumann subalgebra. Let $Y=(y_{i})_{i\in I}$ be measurable self-adjoint operators affiliated with $M$ so that $W^{*}(Y_{i})=N.$ Then
\[h((y_{i})_{i\in I}:M,\mbox{\emph{meas}})=h(N:M).\]
\end{proposition}

\begin{proof} Fix $a\in M$  with diffuse spectrum, and a sequence $(A_{k})_{k=1}^{\infty}$ of microstates for $a.$ Let $X=(\phi(y_{i}))_{i\in I,\phi\in C_{b}(\RR,\RR)}$ and
define $R_{M}\in [0,\infty)^{I_{0}\times C_{b}(\RR,\RR)}$
by $R_{M,i,\phi}=2\|\phi\|_{C_{b}(\RR)}.$
Since $W^{*}(\phi(y_{i}):i\in I,\phi\in C_{b}(\RR,\RR))=N$ and $X$ consists of bounded operators we have
\[h(X:M)=h(N:M).\]
Let $I_{0}\subseteq I,G\subseteq M_{sa}$ be  given finite sets, and let $\varepsilon>0.$ Let $\eta\in (0,\varepsilon)$ and $R>0.$ Choose a $\psi\in C_{c}(\RR,\RR)$ so that $\psi(t)=t$ for all $|t|\leq R$ and $\|\psi\|_{C_{b}(\RR)}\leq R.$
Given  a finite $F\subseteq C_{b}(\RR,\RR)$ containing $\psi$, let $X_{0}=(\phi(y_{i}))_{i\in I_{0},\phi \in F}.$
Given $m\in \NN,\gamma>0,$ we may choose a $m'\in \NN,\gamma'>0$ so that if
\[T\in \Xi_{A_{k},R}^{\eta}((y_{i})_{i\in I_{0}}:G;F,m',\gamma',k),\]
then
\[(\phi(T_{i}))_{i\in I_{0},\phi\in F}\in\Xi_{A_{k},R_{M}\vee R_{M}}((\phi(y_{i}))_{i\in I_{0},\phi\in F}:G;m,\gamma,k).\]
Let
\[S\subseteq \Xi_{A_{k},R_{M}\vee R_{M}}(X_{0}:G;m,\gamma,k)\]
be $\varepsilon$-dense with respect to $\|\cdot\|_{\infty}.$
Given $T\in \Xi_{A_{k},R}^{\eta}((y_{i})_{i\in I_{0}}:G;F,m',\gamma',k),$ choose a $X\in S$ with
\[\|(\phi(T_{i}))_{i\in I_{0},\phi \in F}-X\|_{\infty}<\varepsilon.\]
Then
\[\|\chi_{[R,\infty)}(|T_{i}|)(T_{i}-X_{\psi,i})\|_{\infty}=\|\chi_{[R,\infty)}(|T_{i}|)(\psi(T_{i})-X_{\psi,i})\|_{\infty}<\varespilon\]
and so
\[\Xi_{A_{k},R}^{\eta}((y_{i})_{i\in I_{0}}:G;F,m',\gamma',k)\subseteq_{\varepsilon,\mbox{meas}}\{(X_{\psi,i})_{i\in I_{0}}:X\in S\}.\]
Thus for all $\eta<\varepsilon,$
\[K_{2\varepsilon}(\Xi_{(A_{k})_{k=1}^{\infty},R}^{\eta}((y_{i})_{i\in I_{0}}:G;F,m',\gamma'),\mbox{meas})\leq K_{\varepsilon}(\Xi_{(A_{k})_{k=1}^{\infty},R_{M}\vee R_{M}}(X_{0}:G;m,\gamma),\|\cdot\|_{\infty}).\]
So we have
\[K_{2\varepsilon}(\Xi_{(A_{k})_{k=1}^{\infty},R}^{\eta}((y_{i})_{i\in I_{0}}:M),\mbox{meas})\leq K_{\varepsilon}(\Xi_{(A_{k})_{k=1}^{\infty},R_{M}\vee R_{M}}(X_{0}:M),\|\cdot\|_{\infty}).\]
A fortiori,
\[K_{2\varepsilon}(\Xi_{(A_{k})_{k=1}^{\infty},R}^{\eta}((y_{i})_{i\in I_{0}}:M),\mbox{meas})\leq h(X:M).\]
Taking the supremum over all $R>0,$ then the infimum over all $\eta$ proves that
\[K_{2\varepsilon}(\Xi_{(A_{k})_{k=1}^{\infty}}((y_{i})_{i\in I_{0}}:M),\mbox{meas})\leq  h(X:M).\]
by Corollary \ref{C:switching}. Now taking the supremum over all $\varepsilon,I_{0}$ proves that
\[h(Y:M,\mbox{meas})\leq h(X:M).\]

We now turn to the reverse inequality. Fix finite $\Phi_{0}'\subseteq C_{b}(\RR,\RR),I_{0}'\subseteq I.$ It is enough to show that
\[h((\phi(y_{i}))_{i\in I_{0}',\phi\in \Phi_{0}'}:M)\leq h(Y:M,\mbox{meas}).\]
Let $\varepsilon',\gamma,\eta>0$ and a finite $G_{0}\subseteq M_{sa}$ be given. Choose an $R>\max_{x\in \{a\}\cup G}R_{M,x}$ sufficiently large so that
\[\tau(\chi_{[R,\infty)}(|y_{i}|))<\eta\]
for all $i\in I_{0}'.$
Let $R'_{0}>0$ be as in Lemma \ref{L:measuremicrostateskladhgaklj} with $J=I_{0}',$ $F=\Phi_{0}',$ for $G=\{a\}\cup G_{0}$ and for this $m,\gamma,\eta,R.$ We may find an $R'>R_{0}'$ and a $\psi\in C_{c}(\RR)$ with $\psi(t)=t$ for all $|t|\leq R'$ so that
\[\max_{i\in I_{0}',\phi \in \Phi_{0}'}\|\phi(\psi(y_{i}))-\phi(y_{i})\|_{2}<\frac{\varepsilon'}{1+|I_{0}'|^{1/2}|\Phi_{0}'|^{1/2}}.\]
Let $m',\gamma'$ be as in the conclusion to Lemma \ref{L:measuremicrostateskladhgaklj}. Taking $m'$ larger and $\gamma'$ smaller, if necessary, we may assume that
\[\max_{i\in I_{0}',\phi\in \Phi_{0}'}\|\phi(X_{\psi,i})-X_{\phi,i}\|_{2}<\frac{\varepsilon'}{1+|I_{0}'|^{1/2}|\Phi_{0}'|^{1/2}}\]
for all $X\in \Gamma_{R_{M}}((\phi(y_{i}))_{i\in I_{0}',\phi\in \Phi_{0}'\cup\{\psi\}};m',\gamma',k).$

Let $\kappa'\in (0,1/4)$ be sufficiently small so that for all $l\in \NN$ and all $T,S\in M_{l}(\CC)_{sa}$ with $\|T-S\|_{\infty}<5\kappa'$ and $\|T\|_{\infty},\|S\|_{\infty}\leq 2(R'+1)$ we have
	 \[\max_{\phi\in \Phi_{0}'}\|\phi(T)-\phi(S)\|_{\infty}<\frac{\varepsilon'}{1+|I_{0}'|^{1/2}|\Phi_{0}'|^{1/2}}.\]
Let $\varepsilon\in (0,\kappa'),\eta\in(0,\vaerpsilon)$ be given and choose an $S\subseteq\Xi_{A_{k},R}^{\eta}((y_{i})_{i\in I_{0}'}:G_{0};\Phi'_{0}\cup\{\psi\},m,\gamma,k)$ which is $\varepsilon$-measure dense and has
\[|S|=K_{\varepsilon}(\Xi_{A_{k},R}^{\eta}((y_{i})_{i\in I_{0}'}:G_{0};\Phi_{0}'\cup\{\psi\},m,\gamma,k),\mbox{meas}).\]
For every integer $1\leq l\leq 2\varepsilon k,$ let $\Omega_{l}\subseteq \Gr(l,k-l)$ be $\frac{\kappa'}{1+R'}$-dense with respect to $\|\cdot\|_{\infty}$ and so that
\[|\Omega_{l}|=K_{\frac{\kappa'}{1+R'}}(\Gr(k-l,l),\|\cdot\|_{\infty}).\]
Here $\Gr(k-l,l)$ is the space of orthogonal projections of rank $l.$ By a result of S. Szarek (see \cite{Szarek}) there is a $C>0$ so that
\[|\Omega_{l}|\leq \left(\frac{C(1+R')}{\kappa'}\right)^{2l(k-l)}.\]
Since $\varepsilon<1/4,$ for all integers $l$ with $1\leq l\leq2\varepsilon k$ we have
\[|\Omega_{l}|\leq \left(\frac{C(1+R')}{\kappa'}\right)^{4k^{2}\varepsilon(1-2\varepsilon)}.\]
For each $E\in\Omega_{l},$ choose a
\[D_{E}\subseteq (R'+\varepsilon)(1-E)\Ball(M_{k}(\CC),\|\cdot\|_{\infty})\]
which is $\kappa'$-dense with respect to $\|\cdot\|_{\infty}.$ We may choose $D_{E}$ with
\[|D_{E}|\leq \left(\frac{3R'+4\kappa'}{\kappa'}\right)^{4k^{2}\varepsilon}.\]
Now let
\[X\in  \Xi_{A_{k},R_{M}\vee R_{M}\vee R_{M}}((\phi(y_{i}))_{i\in I_{0}',\phi\in \Phi_{0}'}:G_{0},(\psi(y_{i}))_{i\in I_{0}'};m',\gamma',k).\]
Choose  $Y\in M_{k}(\CC)^{I_{0}'}$ so that
\[(X,Y)\in \Xi_{A_{k},R_{M}\vee R_{M}\vee R_{M}}((\phi(y_{i}))_{i\in I_{0}',\phi\in \Phi_{0}'},(\psi(y_{i}))_{i\in I_{0}'}:G_{0};m',\gamma',k).\]
By Lemma \ref{L:measuremicrostateskladhgaklj} we also have that $Y\in \Xi_{A_{k},R}^{\eta}((y_{i})_{i\in I_{0}'}:G_{0};\Phi_{0}',m,\gamma,k).$
By Lemma \ref{L:projcommlasdjg},  there exists a $T\in S$ and  orthogonal projections $Q_{i}\in M_{k}(\CC),i\in I_{0}'$ so that for all $i\in I_{0}'$
\[\tr(Q_{i})\leq 2\varepsilon,\]
\[\|Q_{i}(T_{i}-Y_{i})\|_{\infty}<\varepsilon,\]
\[Q_{i}\leq \chi_{[0,R'+\varepsilon]}(|T_{i}^{*}|).\]
Note that
\[\|Q_{i}(\chi_{[0,R'+\varepsilon]}(|T_{i}^{*}|)T_{i}-Y_{i})\|_{\infty}=\|Q_{i}(T_{i}-Y_{i})\|_{\infty}<\varepsilon.\]
Let $1\leq l\leq 2k\varepsilon$ be such that $k\tr(Q_{i})=k-l$
and choose $E_{i}\in \Omega_{l}$ so that
\[\|Q_{i}-E_{i}\|_{\infty}<\frac{\kappa'}{1+R'}.\]
Note that
\[\|\chi_{[0,R'+\varepsilon]}(|T_{i}^{*}|)T_{i}\|_{\infty}=\|T_{i}^{*}\chi_{[0,R'+\varepsilon]}(|T_{i}^{*}|)\|_{\infty}\leq R'+\varepsilon\]
and thus
\[\|E_{i}(\chi_{[0,R'+\varepsilon]}(|T_{i}^{*}|)T_{i}-Y_{i})\|_{\infty}\leq 3\kappa'+\|Q_{i}(\chi_{[0,R'+\varepsilon]}(|T_{i}^{*}|)T_{i}-Y_{i})\|_{\infty}<4\kappa'.\]
Choose a $B_{i}\in D_{E_{i}}$ with
\[\|(1-E_{i})Y_{i}-B_{i}\|_{\infty}<\kappa',\]
we then obtain
\[\|Y_{i}-E_{i}\chi_{[0,R'+\varepsilon]}(|T_{i}^{*}|)T_{i}-B_{i}\|_{\infty}<5\kappa'.\]
So by our choice of $\kappa'$ we have
\[\|\phi(Y_{i})-\phi(E_{i}\chi_{[0,R'+\varepsilon]}(|T_{i}^{*}|)T_{i}+B_{i})\|_{\infty}<\frac{\varepsilon'}{1+|I_{0}'|^{1/2}|\Phi_{0}'|^{1/2}},\]
which implies, by our choice of $m',\gamma',$ that for all $i\in I_{0},\phi\in \Phi_{0}'$
\begin{align*}
\|X_{\phi,i}-\phi(E_{i}\chi_{[0,R'+\varepsilon]}(|T_{i}^{*}|)T_{i}+B_{i})\|_{2}&<\frac{\varepsilon'}{1+|I_{0}'|^{1/2}|\Phi_{0}'|^{1/2}}+\|\phi(Y_{i})-\phi(E_{i}\chi_{[0,R'+\varepsilon]}(|T_{i}^{*}|)T_{i}+B_{i})\|_{2}\\
&<\frac{2\vaerpsilon'}{1+|I_{0}'|^{1/2}|\Phi_{0}'|^{1/2}}.
\end{align*}
Thus
\[\Xi_{A_{k},R_{M}}((\phi(y_{i}))_{i\in I_{0}',\phi\in \Phi_{0}'}:G_{0}';m',\gamma',k)\subseteq_{2\varepsilon',\|\cdot\|_{2}}\bigcup_{T\in S}\prod_{i\in I_{0}',\phi \in \Phi_{0}'}\bigcup_{l=1}^{\lfloor{2k\varepsilon\rfloor}}\{\phi(E\chi_{[0,R'+\varepsilon]}(|T_{i}^{*}|)T_{i}+B):E\in \Omega_{l},B\in D_{E}\},\]
so
\begin{align*}
K_{8\varepsilon'}(\Xi_{A_{k},R_{M}}((\phi(y_{i}))_{i\in I_{0}',\phi\in \Phi_{0}'}:G_{0}';m',\gamma',k),\|\cdot\|_{2})&\leq |S|\prod_{i\in I_{0}',\phi\in \Phi_{0}'}\left(\sum_{l=1}^{\lfloor{2k\varepsilon\rfloor}}|\Omega_{l}|\left(\frac{3R'+4\kappa'}{\kappa
'}\right)^{4k^{2}\varepsilon}\right)\\
&\leq 2^{|I_{0}'||\Phi_{0}'|}|S|(k\varepsilon)^{|I_{0}'||\Phi_{0}'|}\left(\frac{C(1+R')}{\kappa'}\right)^{4k^{2}|I_{0}'||\Phi_{0}'|\varepsilon(1-2\varepsilon)}\\
&\times\left(\frac{3R'+4\kappa'}{\kappa'}\right)^{4k^{2}|I_{0}'||\Phi_{0}'|\varepsilon}.
\end{align*}
Hence
\begin{align*}
K_{8\varepsilon'}(\Xi_{(A_{k})_{k=1}^{\infty},R_{M}}((\phi(y_{i}))_{i\in I_{0}',\phi\in \Phi_{0}'}:G_{0}';m',\gamma'),\|\cdot\|_{2})&\leq K_{\varepsilon}(\Xi_{(A_{k})_{k=1}^{\infty},R}^{\eta}((y_{i})_{i\in I_{0}'}:G_{0};m,\gamma),\mbox{meas})\\
&+4|I_{0}'||\Phi_{0}'|\varepsilon(1-2\varepsilon)\log\left(\frac{C(1+R')}{\kappa'}\right)\\
&+4|I_{0}'||\Phi_{0}'|\varepsilon\log\left(\frac{3R'+4\kappa'}{\kappa'}\right).
\end{align*}
A fortiori,
\begin{align*}
K_{8\varepsilon'}(\Xi_{(A_{k})_{k=1}^{\infty},R_{M}}((\phi(y_{i}))_{i\in I_{0}',\phi\in \Phi_{0}'}:M),\|\cdot\|_{2})&\leq K_{\varepsilon}(\Xi_{(A_{k})_{k=1}^{\infty},R}^{\eta}((y_{i})_{i\in I_{0}'}:G_{0};m,\gamma),\mbox{meas})\\
&+4|I_{0}'||\Phi_{0}'|\varepsilon(1-2\varepsilon)\log\left(\frac{C(1+R')}{\kappa'}\right)+4|I_{0}'||\Phi_{0}'|\varepsilon\log\left(\frac{3R'+4\kappa
'}{\kappa'}\right).
\end{align*}
Since the left-hand side is now independent of $G_{0},m,\gamma,$ we can take the infimum over all $G_{0}m,\gamma$ to see that
\begin{align*}
K_{8\varepsilon'}(\Xi_{(A_{k})_{k=1}^{\infty},R_{M}}((\phi(y_{i}))_{i\in I_{0}',\phi\in \Phi_{0}'}:M),\|\cdot\|_{2})&\leq K_{\varepsilon}(\Xi_{(A_{k})_{k=1}^{\infty},R}^{\eta}((y_{i})_{i\in I_{0}'}:M),\mbox{meas})\\
&+4|I_{0}'||\Phi_{0}'|\varepsilon(1-2\varepsilon)\log\left(\frac{C(1+R')}{\kappa'}\right)+4|I_{0}'||\Phi_{0}'|\varepsilon\log\left(\frac{3R'+4\kappa'}{\kappa'}\right)\\
&\leq K_{\varepsilon}(\Xi_{(A_{k})_{k=1}^{\infty}}^{\eta}((y_{i})_{i\in I_{0}'}:M),\mbox{meas})\\
&+4|I_{0}'||\Phi_{0}'|\varepsilon(1-2\varepsilon)\log\left(\frac{C(1+R')}{\kappa'}\right)+4|I_{0}'||\Phi_{0}'|\varepsilon\log\left(\frac{3R'+4\kappa'}{\kappa'}\right).
\end{align*}
We can now take  the infimum over $\eta>0,$ and then let $\vaerpsilon\to 0$ to see that
\[K_{8\varepsilon'}(\Xi_{(A_{k})_{k=1}^{\infty},R_{M}}((\phi(y_{i}))_{i\in I_{0}',\phi\in \Phi_{0}'}:M),\|\cdot\|_{2})\leq h(Y:M,\mbox{meas}).\]
Letting $\varepsilon'\to 0$ and then taking the supremum over $I_{0}',\Phi_{0}'$ completes the proof.

\end{proof}


\end{document}